\documentclass[11pt, twoside, a4paper]{article}
\usepackage[english]{babel}
\usepackage[utf8]{inputenc}
\usepackage{amsmath, amssymb, amsthm}            
\usepackage{amstext, amsfonts, a4}
\usepackage{enumitem,comment}
\numberwithin{equation}{section}

\usepackage[margin=2.5cm]{geometry}
\usepackage{graphicx}
\usepackage{xcolor}
\usepackage[lofdepth,lotdepth]{subfig}
\usepackage[hidelinks]{hyperref}
\hypersetup{
 colorlinks,
 linkcolor={red!50!black},
 citecolor={blue!50!black},
 urlcolor={blue!80!black}
}

\theoremstyle{plain}
\newtheorem{theorem}{Theorem}[section]
\newtheorem{proposition}[theorem]{Proposition}    
\newtheorem{lemma}[theorem]{Lemma}          
\newtheorem{corollary}[theorem]{Corollary}

\newtheorem{remark}[theorem]{Remark}

\theoremstyle{definition}

\newcommand{\cA}{\mathcal A} \newcommand{\cB}{\mathcal B}  
\newcommand{\cD}{\mathcal D} \newcommand{\cE}{\mathcal E} \newcommand{\cF}{\mathcal F} 
\newcommand{\cG}{\mathcal G}  \newcommand{\cI}{\mathcal I} 
 \newcommand{\cK}{\mathcal K} \newcommand{\cL}{\mathcal L} 
\newcommand{\cM}{\mathcal M} \newcommand{\cN}{\mathcal N} \newcommand{\cO}{\mathcal O} 
 \newcommand{\cQ}{\mathcal Q} \newcommand{\cR}{\mathcal R} 
\newcommand{\cS}{\mathcal S}   
   
\newcommand{\cY}{\mathcal Y} \newcommand{\cZ}{\mathcal Z} 

\newcommand{\R}{\mathbb{R}}

\newcommand{\N}{\mathbb{N}}

\renewcommand{\epsilon}{\varepsilon}
\newcommand{\dd}{\,\mathrm{d}}
\newcommand{\Diff}{\mathrm{D}}
\newcommand{\app}{\mathrm{app}}
\newcommand{\BS}{\mathrm{BS}}
\newcommand{\weakto}{\rightharpoonup}
\newcommand{\Ker}{\mathrm{Ker}}
\newcommand{\Ran}{\mathrm{Ran}}
\newcommand{\TS}{\textstyle}

\newcommand{\I}{\mathrm{I}}
\newcommand{\II}{\mathrm{II}}
\newcommand{\III}{\mathrm{III}}

\usepackage[normalem]{ulem}

\title{Fast relaxation of a viscous vortex in an external flow}
\author{Martin Donati and Thierry Gallay}
\date{\today}

\begin{document}

\maketitle

\begin{abstract}
We study the evolution of a concentrated vortex advected by a smooth, divergence-free 
velocity field in two space dimensions.  In the idealized situation where
the initial vorticity is a Dirac mass, we compute an approximation of the
solution which accurately describes, in the regime of high Reynolds numbers, the
motion of the vortex center and the deformation of the streamlines under the
shear stress of the external flow. For ill-prepared initial data, corresponding
to a sharply peaked Gaussian vortex, we prove relaxation to the previous
solution on a time scale that is much shorter than the diffusive time, due to
enhanced dissipation inside the vortex core.
\end{abstract}

\section{Introduction}\label{sec1}

We revisit the classical problem of the evolution of a concentrated vortex in a
background flow, which was carefully studied in the monographs \cite{TK91,TKK07}
and the previous works \cite{TT65,LT87}. We assume that the external velocity
field is smooth, divergence-free, and uniformly bounded together with its 
derivatives. Our goal is to give a rigorous description of the solution of the
two-dimensional Navier-Stokes equations in such a background flow, for
concentrated initial data corresponding either to a point vortex or to a sharply
peaked Gaussian vortex. In both cases the solution remains concentrated for
quite a long time provided the kinematic viscosity $\nu > 0$ is sufficiently
small. The leading order approximation is a Lamb-Oseen vortex whose center is
advected by the external flow, whereas the vortex core spreads diffusively due
to viscosity. Higher order corrections describe the deformation of the
streamlines under the external shear stress, and appear to be sensitive to the
choice of the initial data.

From the point of view of mathematical analysis, it is convenient to consider
first the idealized situation where the initial vorticity is just a Dirac mass.
Despite the singular nature of such data, the initial value problem is
globally well-posed, as can be seen by adapting to the present case the
results that are known for the two-dimensional vorticity equation in the space
of finite measures \cite{GMO88,GW05,GaGa05}. By construction, the size of the
vortex core vanishes at initial time, and is therefore infinitely small compared
to the typical length scale $d_0 > 0$ defined by the external flow. For such
{\em well-prepared} initial data, the approximate solution constructed in
\cite{TT65,TK91} depends only on the ``normal'' time scale associated with
the external field, and describes the deformation of the vortex core under
the external shear stress. 

The situation is quite different if the initial vorticity is a radially
symmetric vortex patch or vortex blob with finite extension $\ell_0\ll
d_0$. Such data can be described as {\em ill-prepared}, in the sense that the
resulting solution exhibits a transient regime during which the initially
symmetric vortex gets deformed to adapt its shape to the external strain. The
streamlines near the vortex core become elliptical, with an eccentricity that
undergoes damped oscillations on a short time scale until it settles down to the
value predicted by the well-prepared solution. This evolution is illustrated by a
numerical simulation in Figure~\ref{fig:relax}. In a second stage, the vorticity
distribution inside the core slowly relaxes to a Gaussian profile under the
action of viscosity.  This two-step process was carefully studied by Le Diz\`es
and Verga \cite{LDV02} in the related case of a co-rotating vortex pair, for
which the deformation of the vortex cores is just the first stage of a complex
dynamics eventually leading to vortex merging \cite{MLDL05}. In the
perturbative approach of Ting and Klein \cite{TK91}, a two-time analysis is
necessary to obtain an accurate description of the solution in the ill-prepared
case.

\renewcommand{\thesubfigure}{}
\captionsetup[subfigure]{labelformat=simple,labelsep=colon,
listofformat=subsimple}
\captionsetup{lofdepth=2}
\makeatletter
\renewcommand{\p@subfigure}{}
\makeatother

\begin{figure}
    \centering
    \subfloat[$t=0$]{\includegraphics[width=0.14\linewidth]{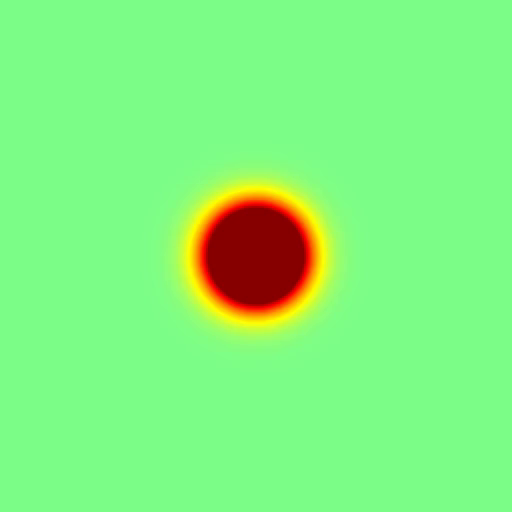}}\hspace{0.5mm}
    \subfloat[$t=0.032$]{\includegraphics[width=0.14\linewidth]{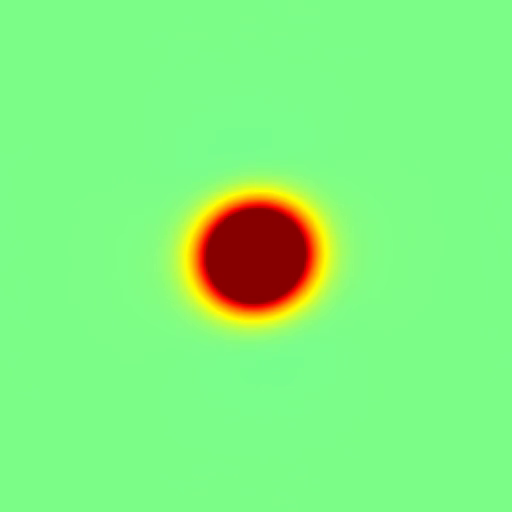}}\hspace{0.5mm}
    \subfloat[$t=0.096$]{\includegraphics[width=0.14\linewidth]{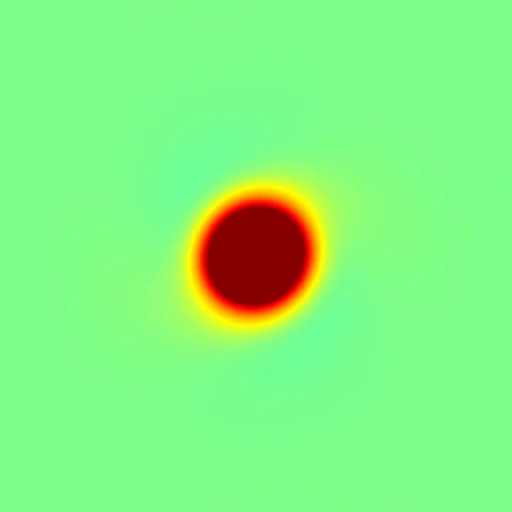}}
    \hspace{5mm}
    \subfloat[$t=0$]{\includegraphics[width=0.14\linewidth]{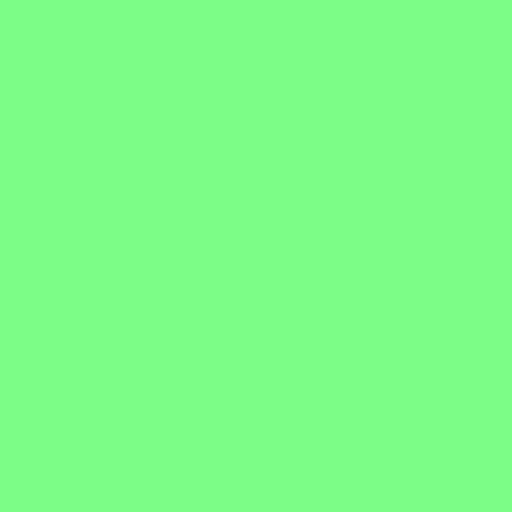}}\hspace{0.5mm}
    \subfloat[$t=0.032$]{\includegraphics[width=0.14\linewidth]{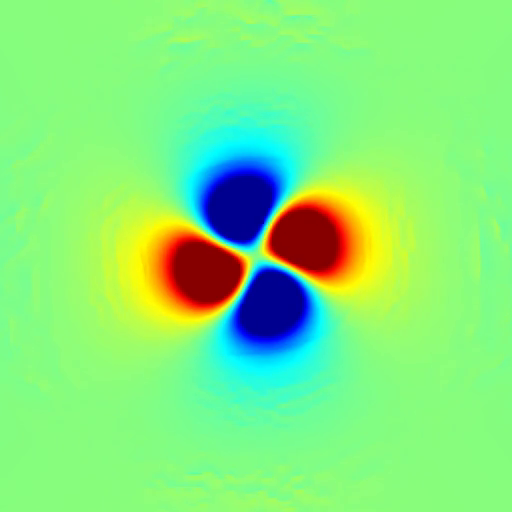}}\hspace{0.5mm}
    \subfloat[$t=0.096$]{\includegraphics[width=0.14\linewidth]{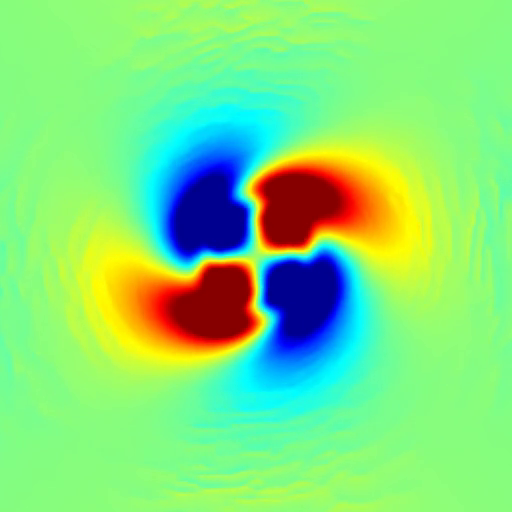}}
    
    \subfloat[$t=0.144$]{\includegraphics[width=0.14\linewidth]{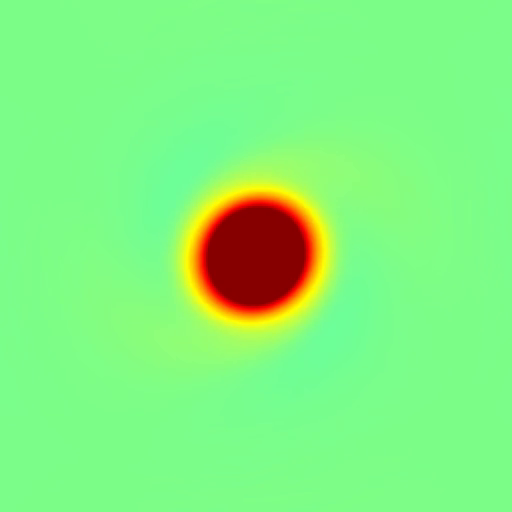}}\hspace{0.5mm}
    \subfloat[$t=0.224$]{\includegraphics[width=0.14\linewidth]{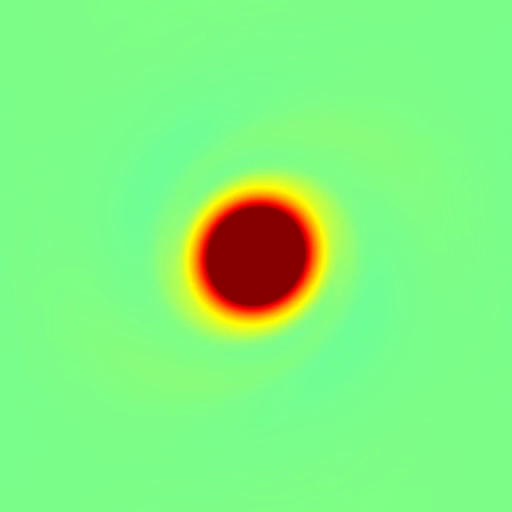}}\hspace{0.5mm}
    \subfloat[$t=0.4$]{\includegraphics[width=0.14\linewidth]{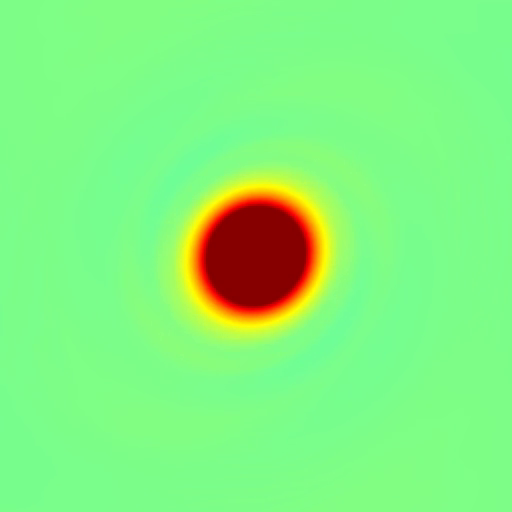}}
    \hspace{5mm}
    \subfloat[$t=0.144$]{\includegraphics[width=0.14\linewidth]{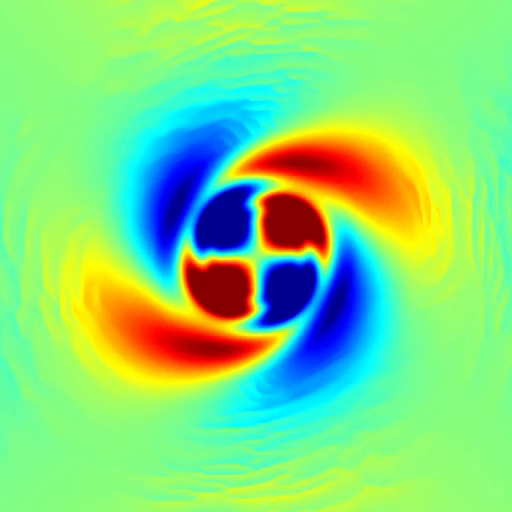}}\hspace{0.5mm}
    \subfloat[$t=0.224$]{\includegraphics[width=0.14\linewidth]{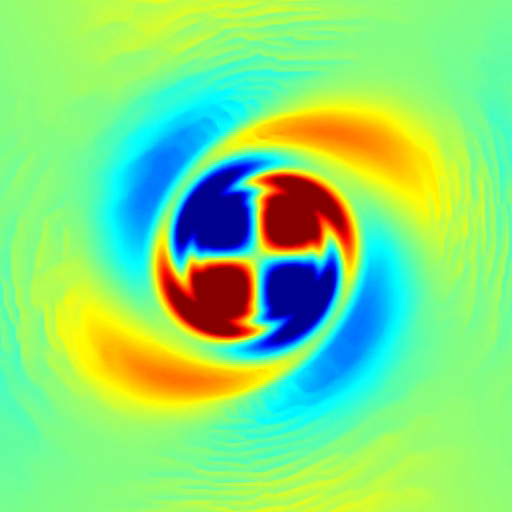}}\hspace{0.5mm}
    \subfloat[$t=0.4$]{\includegraphics[width=0.14\linewidth]{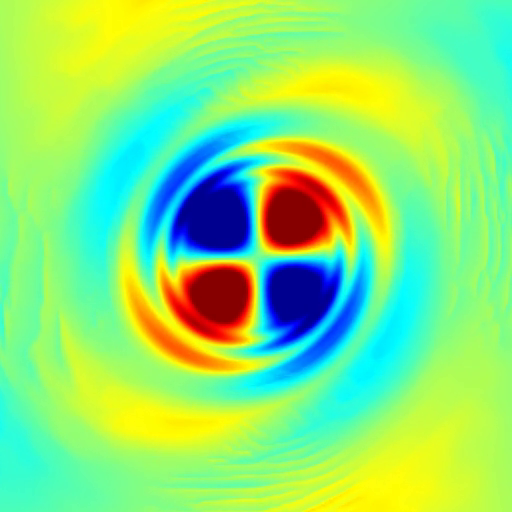}}
    
    \subfloat[$t=0.64$]{\includegraphics[width=0.14\linewidth]{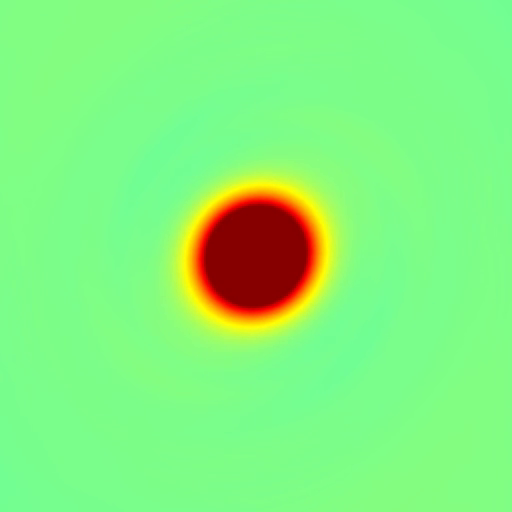}}\hspace{0.5mm}
    \subfloat[$t=1.1$]{\includegraphics[width=0.14\linewidth]{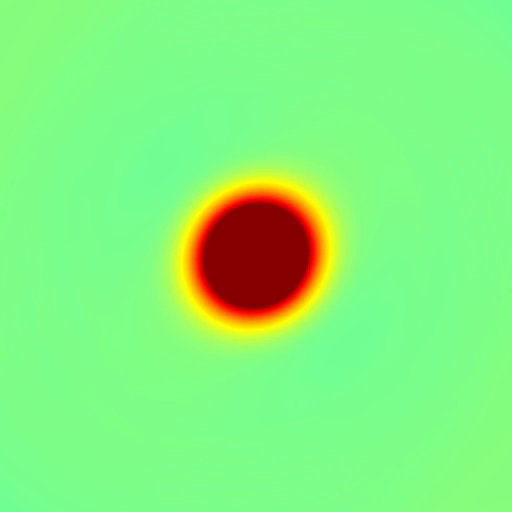}}\hspace{0.5mm}
    \subfloat[$t=7$]{\includegraphics[width=0.14\linewidth]{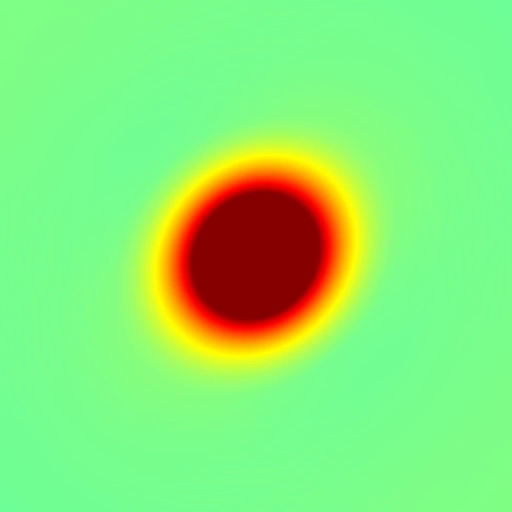}}
    \hspace{5mm}
    \subfloat[$t=0.64$]{\includegraphics[width=0.14\linewidth]{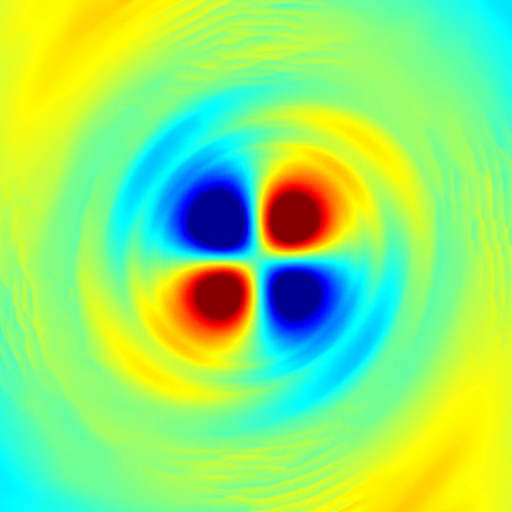}}\hspace{0.5mm}
    \subfloat[$t=1.1$]{\includegraphics[width=0.14\linewidth]{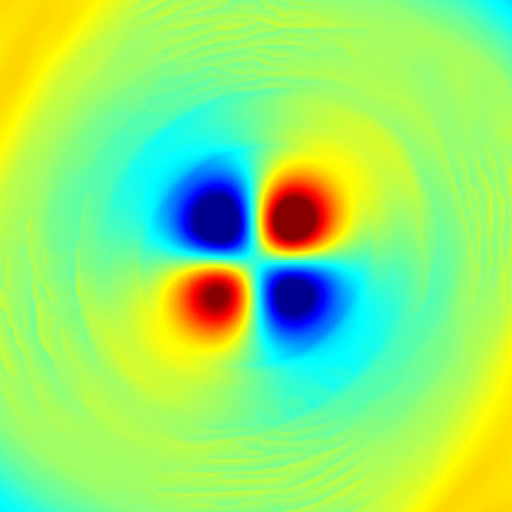}}\hspace{0.5mm}
    \subfloat[$t=7$]{\includegraphics[width=0.14\linewidth]{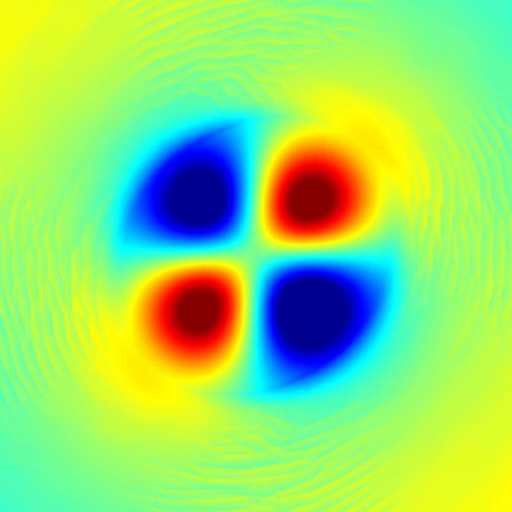}}
    \caption{Numerical simulation of a vortex in an external field with Gaussian initial data.
    The vorticity distribution (left) and the deviation from the Lamb-Oseen vortex (right)
    are represented at nine different times, using standard color codes for the vorticity
    levels. The final state at $t = 7$ is close to the approximate solution defined
    in \eqref{def:omapp}. This simulation is made with the free software
    \href{http://basilisk.fr/}{Basilisk}, and the external field is chosen as
    in Section~\ref{ssecA1}.}\label{fig:relax}
\end{figure}

The purpose of this paper is twofold. First, we show that the techniques
introduced in \cite{Ga11} to study the solution of the two-dimensional
Navier-Stokes equations with a finite collection of point vortices as initial
data can be adapted to the emblematic case of a single vortex in an external
flow, which is at the same time simpler and more general. In particular, if the
initial vorticity is a Dirac mass, we construct perturbatively an accurate
approximation of the solution, and we verify that the exact solution remains
close to it over a long time interval if the viscosity is small enough. Next, we
consider ill-prepared data for which the initial vorticity is a sharply
concentrated Gaussian function, and we prove that the resulting solution rapidly
relaxes towards the approximate solution computed in the well-prepared case.
That part of the analysis relies on enhanced dissipation estimates for the
linearized Navier-Stokes equations at the Lamb-Oseen vortex, which are due to
Li, Wei, and Zhang \cite{LWZ20}. Such estimates were already applied in
\cite{Ga18} to prove axisymmetrization near a Gaussian vortex in the
regime of high Reynolds numbers, but to our knowledge they were never
used to study the relaxation of a circular vortex towards a non-symmetric
metastable state in an external strain. 

\medskip
We now present our results in a more precise way. We give ourselves a
smooth, time-dependent velocity field $f = (f_1,f_2): \R^2 \times [0,T] \to \R^2$
which is uniformly bounded together with its derivatives with respect to the space
variable $x = (x_1,x_2) \in \R^2$ and the time $t \in [0,T]$. We assume that $f$
is divergence-free, namely
\[
  \nabla\cdot f(x,t) \,:=\, \partial_{x_1} f_1(x,t) + \partial_{x_2} f_2(x,t)
  \,=\, 0\,, \qquad \forall\,(x,t) \in \R^2\times [0,T]\,.
\]
The characteristic time $T_0 > 0$ of the velocity field $f$ is defined by
the classical formula
\begin{equation}\label{eq:T0def}
  \frac{1}{T_0} \,=\, \sup_{t \in [0,T]}\|\Diff f(\cdot,t)\|_{L^\infty(\R^2)}\,,
\end{equation}
where $\Diff f$ denotes the first order differential of $f$ with respect to the
space variable. To avoid trivial situations, we suppose from now on that 
$T_0 < \infty$, which means that $\Diff f \not\equiv 0$. 

We consider the evolution of a concentrated vortex embedded
in the external flow described by the velocity field $f$. The vorticity
distribution $\omega(x,t)$ is a scalar function satisfying the evolution equation
\begin{equation}\label{eq:NSf}
  \partial_t\omega(x,t) + \bigl(u(x,t) + f(x,t)\bigr) \cdot \nabla
  \omega(x,t) \,=\, \nu \Delta\omega(x,t)\,, \qquad \forall\,(x,t)
  \in \R^2 \times (0,T)\,,
\end{equation}
where the parameter $\nu > 0$ is the kinematic viscosity of the fluid. 
The velocity field $u = (u_1,u_2)$ associated with $\omega$ is given by the
Biot-Savart formula
\begin{equation}\label{eq:BS}
  u(x,t) \,=\, \frac{1}{2\pi} \int_{\R^2} \frac{(x-y)^\perp}{|x-y|^2}\,
  \omega(y,t)\dd y\,, \qquad \forall\,(x,t) \in \R^2 \times (0,T)\,,
\end{equation}
where we use the notation $x^\perp = (-x_2,x_1)$ and $|x|^2 = x_1^2 + x_2^2$ for
all $x = (x_1,x_2) \in \R^2$. We denote $u = \BS[\omega]$ and we observe that
$\nabla \cdot u = 0$ and $\partial_{x_1} u_2 - \partial_{x_2} u_1 = \omega$.
Equations \eqref{eq:NSf} and \eqref{eq:BS} form a closed system, which corresponds
when $f \equiv 0$ to the usual two-dimensional incompressible Navier-Stokes
equations in vorticity form. We refer the reader to \cite{MajBer02,MarPul94} for
general results on these equations.

\begin{remark}\label{rem:f}
Equation \eqref{eq:NSf} appears in at least two physical contexts. The first
one is the evolution of a finite number of isolated vortices under the Navier-Stokes equations without external field. The total vorticity 
can be decomposed as $\omega = \omega_1 + \ldots + \omega_N$, and the first component 
$\omega_1$ solves equation~\eqref{eq:NSf} with $f = \BS[\omega_2]+\ldots + \BS[\omega_N]$. 
In other words, if we focus on one particular vortex, we are naturally led to an 
advection-diffusion equation of the form \eqref{eq:NSf} involving the velocity field $f$ 
created by the other vortices. This is a standard point of view, see for example 
\cite{Mar98,GaGa05,Ga11}. Alternatively, following \cite{TK91,TKK07}, we can consider 
the evolution of a single vortex in a background potential flow $f$, which is typically 
due to an inflow condition at infinity. In that case the dynamics of the 
vortex does not influence the external flow because the associated vorticity $\omega_f := 
\partial_{x_1} f_2 - \partial_{x_2} f_1$ vanishes identically.
\end{remark}

We first consider the idealized situation where the initial vorticity is a
Dirac mass, which means that $\omega_0 = \Gamma \delta_{z_0}$ for some $\Gamma \in \R^*$
and some $z_0 \in \R^2$. Without loss of generality, we assume henceforth that
$\Gamma > 0$. Adapting the results of \cite{GMO88,GaGa05}, which hold
for $f \equiv 0$, it is not difficult to verify that Eq.~\eqref{eq:NSf} has a unique
(mild) solution $\omega \in C^0\bigl((0,T],L^1(\R^2) \cap L^\infty(\R^2)\bigr)$ such that
\begin{equation}\label{eq:omunique}
  \sup_{0 < t  \le T} \|\omega(\cdot,t)\|_{L^1} \,<\, \infty\,, \qquad
  \text{and} \qquad \omega(\cdot,t)\dd x \,\weakto\, \Gamma \delta_{z_0}
  \quad \text{as }\, t \to 0\,,
\end{equation}
where the half-arrow $\weakto$ denotes the weak convergence of measures. 
In the simple case where $f \equiv 0$, the solution takes the explicit form
\begin{equation}\label{eq:LambOseen}
  \omega(x,t) \,=\, \frac{\Gamma}{\nu t}\,\Omega_0\biggl(\frac{x-z_0}{\sqrt{\nu t}}
  \biggr)\,, \qquad 
  u(x,t) \,=\, \frac{\Gamma}{\sqrt{\nu t}}\,U_0\biggl(\frac{x-z_0}{\sqrt{\nu t}}
  \biggr)\,,  
\end{equation}
for all $(x,t) \in \R^2\times (0,+\infty)$, where the vorticity $\Omega_0$ and the
velocity $U_0 = \BS[\Omega_0]$ are given by 
\begin{equation}\label{eq:OmU0}
  \Omega_0(\xi) \,=\, \frac{1}{4\pi}\exp\Bigl(-\frac{|\xi|^2}{4}\Bigr)\,,
  \qquad U_0(\xi) \,=\, \frac{1}{2\pi}\,\frac{\xi^\perp}{|\xi|^2}\biggl(1 -
  \exp\Bigl(-\frac{|\xi|^2}{4}\Bigr)\biggr)\,, \qquad \forall\,\xi \in \R^2\,.
\end{equation}
Note that $u \cdot\nabla\omega \equiv 0$, so that $\omega$ actually solves
the linear heat equation $\partial_t\omega = \nu\Delta\omega$. The self-similar
solution \eqref{eq:LambOseen} of the two-dimensional vorticity equation is referred
to as the {\em Lamb-Oseen vortex} with circulation $\Gamma > 0$, centered at the point
$z_0 \in \R^2$. More generally, for solutions of \eqref{eq:NSf} and \eqref{eq:BS}, 
the {\em total circulation} is the conserved quantity defined by
\[
  \Gamma \,:=\, \int_{\R^2}\omega(x,t)\dd x\,.
\]
The dimensionless ratio $\Gamma/\nu$ is called the {\em circulation Reynolds number}. 

In the more interesting situation where $f \not\equiv 0$, no explicit expression
is available in general, but if the viscosity is weak enough so that the
diffusion length $\sqrt{\nu t}$ is small compared to the characteristic length
defined by the external flow, we can approximate the solution of \eqref{eq:NSf}
by a sharply concentrated Lamb-Oseen vortex which is simply advected by the
external velocity field. This fact is rigorously stated in the following result. 

\begin{proposition}\label{prop1}
Fix $\Gamma > 0$ and $z_0 \in \R^2$. There exist positive constants $K_0,\delta_0$
such that, if $0 < \nu/\Gamma < \delta_0$, the unique solution of \eqref{eq:NSf},
\eqref{eq:BS} satisfying \eqref{eq:omunique} has the following property:
\begin{equation}\label{eq:prop1}
  \frac{1}{\Gamma}\int_{\R^2}\biggl|\,\omega(x,t) - \frac{\Gamma}{\nu t}\,\Omega_0\biggl(
  \frac{x-\hat z(t)}{\sqrt{\nu t}}\biggr)\biggr|\dd x \,\le\, K_0\,
  \frac{\sqrt{\nu t}}{d}\,, \qquad \forall\,t \in (0,T)\,,
\end{equation}
where $d = \sqrt{\Gamma T_0}$ and $\hat z(t)$ is the unique solution of the
differential equation
\begin{equation}\label{eq:originz}
  \hat z'(t) \,=\, f(\hat z(t),t)\,, \qquad \hat z(0) = z_0\,.
\end{equation}
\end{proposition}

\begin{remark}\label{rem:prop1}
The dimensionless constants $K_0$ and $\delta_0$ depend only on the ratio $\cR := T/T_0$
and on the quantity
\begin{equation}\label{eq:cKdef}
  \cK \,:=\, \frac{T_0}{d}\,\sum_{m=0}^2\,\sum_{k=0}^4 \,T_0^m d^k\,
  \|\partial_t^m \Diff^{k}f\|_{L^\infty(\R^2\times[0,T])}\,,
\end{equation}
which measures the intensity of the external flow. We expect that
$K_0 \to \infty$ and $\delta_0 \to 0$ as $\cR \to \infty$ or $\cK \to
\infty$. We emphasize, however, that estimate \eqref{eq:prop1} holds uniformly in
$\nu$ provided the inverse Reynolds number $\nu/\Gamma$ is sufficiently
small. In particular we see that $\omega(\cdot,t)\dd x \weakto\Gamma \delta_{\hat z(t)}$
for all $t \in (0,T)$ as $\nu \to 0$. Note that the assumption
that $\cK < \infty$ may be too strong for some applications, for instance 
if we consider the evolution of $N$ vortices starting from Dirac masses 
as initial data, see Remark~\ref{rem:f}. This is mainly a technical issue, however, and while it is convenient to assume that $\cK < \infty$ in the general 
situation considered here, it is also possible to obtain similar results 
in particular cases where this condition is not exactly met, see \cite{Ga11}.
\end{remark}

\begin{remark}\label{rem:d}
The quantity $d = \sqrt{\Gamma T_0}$ can be interpreted as the effective size of
a vortex of circulation $\Gamma$ in an external field, namely the size of the
neighborhood of the vortex center in which the external strain is weaker than
the strain of the vortex itself. It should not be confused with the size of the vortex
core, which depends on the vorticity distribution and can be considerably smaller. 
In the setting of Proposition~\ref{prop1}, the latter quantity is proportional to the
diffusion length $\sqrt{\nu t}$, which is indeed much smaller than $d$ if 
$\delta_0 T/T_0 \ll 1$. Under these assumptions, estimate \eqref{eq:prop1} 
provides a good approximation of the solution $\omega(x,t)$ of \eqref{eq:NSf}. 
\end{remark}

Estimate \eqref{eq:prop1} is simple and elegant, but does not describe the
deformation of the vortex core under the action of the external flow, which is
the main phenomenon we want to study in this paper. Therefore we need a more
precise asymptotic expansion of the solution of \eqref{eq:NSf}, which includes
non-radially symmetric corrections that were neglected in \eqref{eq:prop1}.
To this end, we propose the following approximation of a Gaussian vortex of
circulation $\Gamma > 0$ and core size $\ell > 0$, located at a point
$z \in \R^2$, and undergoing the strain of an external velocity field $f$:
\begin{equation}\label{def:omapp}
  \omega_\app\bigl(\Gamma,\ell,z,f\,; x\bigr) \,=\, \frac{\Gamma}{\ell^2}\,
  \Omega_0\Bigl(\frac{x-z}{\ell}\Bigr) + w_2\Bigl(\frac{|x-z|}{\ell}\Bigr)
  \bigl(a_f(z)\sin(2\theta) - b_f(z) \cos(2\theta)\bigr)\,.
\end{equation}
Here, for all $x \in \R^2$, we denote by $\theta$ the polar angle of the
rescaled variable $(x-z)/\ell$, which is adapted to the description of the
vortex core. The strain rates $a_f(z), b_f(z)$ are defined by 
\begin{equation}\label{def:ab}
  a_f(z) \,=\, \frac12 \bigl(\partial_1 f_1 - \partial_2 f_2\bigr)(z)\,, \qquad
  b_f(z) \,=\, \frac12 \bigl(\partial_1 f_2 + \partial_2 f_1\bigr)(z)\,,
\end{equation}
and the smooth function $w_2 : (0,+\infty) \to (0,+\infty)$ can be expressed
in terms of the solution of a linear differential equation, see
Remark~\ref{rem:w2} and Figure~\ref{fig:w2}. For our purposes it is enough
to know that $w_2(r) = \cO(r^2)$ as $r \to 0$ and $w_2(r) \sim (r^4/8)e^{-r^2/4}$ as
$r \to +\infty$.

\begin{remark}\label{rem:Burgers}
The expression \eqref{def:omapp} is not new and appears in related contexts, 
in particular in the large-Reynolds-number expansion of Burgers vortices, 
see \cite{RS84,MKO94} and Section~\ref{ssecA1} below.
\end{remark}

Identifying the core size $\ell$ with the diffusion length $\sqrt{\nu t}$, we see
that the first term in the right-hand side of \eqref{def:omapp} is exactly
the Lamb-Oseen vortex \eqref{eq:LambOseen}, which is a radially symmetric
function of the rescaled variable $(x-z)/\ell$; in contrast, the correction term
involving $w_2$ depends explicitly on the polar angle $\theta$. In view of
\eqref{eq:T0def} the strain rates \eqref{def:ab} are bounded by $T_0^{-1}$, so that
\begin{equation}\label{eq:w2est}
  \int_{\R^2} \Big| w_2\Bigl(\frac{|x-z|}{\ell}\Bigr)
  \bigl(a_f(z)\sin(2\theta) - b_f(z) \cos(2\theta)\bigr)\Bigr| \dd x \,\le\,
  C\,\frac{\ell^2}{T_0}\,,
\end{equation}
for some constant $C > 0$. If $\ell^2 = \nu t \ll d^2 = \Gamma T_0$, as
is the case under the assumptions of Proposition~\ref{prop1}, we deduce
that the Lamb-Oseen vortex is the leading term in the approximation
\eqref{def:omapp}. We also observe that, since the correction term is a
linear function of $\cos(2\theta)$ and $\sin(2\theta)$, the streamlines of the
corresponding velocity field are elliptical in a first approximation, which is of
course a well-known fact \cite{LDV02,MLDL05}.

We are now in a position to state our first main result, which subsumes Proposition~\ref{prop1}. 

\begin{theorem}\label{thm1}
Fix $\Gamma > 0$ and $z_0 \in \R^2$. There exist positive constants $K_1,\delta_1$
such that, if $0 < \nu/\Gamma < \delta_1$, the unique solution of \eqref{eq:NSf},
\eqref{eq:BS} satisfying \eqref{eq:omunique} has the following property:
\begin{equation}\label{eq:thm1}
  \frac{1}{\Gamma}\int_{\R^2}\Bigl|\,\omega(x,t) - \omega_\app\bigl(\Gamma,\sqrt{\nu t},z(t),f(t)
  \,; x\bigr)\Bigr|\dd x \,\le\, K_1\,\epsilon(t)^2 \bigl(\epsilon(t) + \delta\bigr)\,,
  \qquad \forall\,t \in (0,T)\,,
\end{equation}
where $\epsilon(t) = \sqrt{\nu t}/d$, $d = \sqrt{\Gamma T_0}$, $\delta = \nu/\Gamma$,
and $z(t)$ is the unique solution of the ODE
\begin{equation}\label{eq:modifz}
  z'(t) \,=\, f(z(t),t) + \nu t \Delta f(z(t),t)\,, \qquad 
\end{equation}
with initial condition $z(0) = z_0$. 
\end{theorem}

Estimate \eqref{eq:thm1} shows that the solution of \eqref{eq:NSf} stays very
close to the approximation \eqref{def:omapp} with $\ell = \sqrt{\nu t}$ and
$f = f(\cdot,t)$, provided the vortex position $z(t)$ evolves according to the
ODE \eqref{eq:modifz}, which contains the viscous correction term
$\nu t \Delta f$. We observe that, if the external velocity field $f$ is
irrotational, then $\Delta f = \nabla^\perp\omega_f = 0$ so that
\eqref{eq:modifz} reduces to \eqref{eq:originz}. In the general case, the
solutions of \eqref{eq:originz} and \eqref{eq:modifz} do not coincide, but
they stay close to each other, and a simple calculation that is postponed
to Section~\ref{sssec334} shows that estimate \eqref{eq:thm1} implies
\eqref{eq:prop1}.

\begin{remark}\label{rem:zbar}
The most natural way of locating the position of a concentrated vortex that
evolves according to \eqref{eq:NSf} is to use the center of vorticity $\bar z(t)$,
which satisfies
\begin{equation}\label{eq:zbardef}
  \bar z(t) \,=\, \frac{1}{\Gamma}\int_{\R^2} x\,\omega(x,t)\dd x\,, \qquad
  \bar z'(t) \,=\, \frac{1}{\Gamma}\int_{\R^2} f(x,t)\,\omega(x,t)\dd x\,.
\end{equation}
Under the assumptions of Theorem~\ref{thm1}, we show in Section~\ref{sssec334}
that $|\bar z(t) - z(t)| \le Cd\epsilon^3\bigl(\epsilon + \delta\bigr)$ for some
constant $C > 0$. This means that the motion of the center of vorticity is
accurately described by the ODE \eqref{eq:modifz}, and that estimate \eqref{eq:thm1}
still holds if $z(t)$ is replaced by $\bar z(t)$. In contrast, using the
naive vortex position $\hat z(t)$ deteriorates the precision of our approximate 
solution, as can be seen from the right-hand side of \eqref{eq:prop1} which is
$\cO(\epsilon)$ instead of $\cO(\epsilon^2)$.  
\end{remark}

\begin{remark}\label{rem:longT}
If the external flow $f$ is globally defined and satisfies uniform 
bounds, then following closely the proof of Theorem~\ref{thm1} one can verify that
estimate~\eqref{eq:thm1} actually holds as long as $t \le cT_0 \ln(1/\delta)$, 
where $c > 0$ is a small constant, see also Remark~\ref{rem:logtime}. This logarithmic time scale agrees with the confinement 
result obtained in \cite{CS24}, and is expected to be optimal in general. Indeed, 
this is the time at which instabilities appear in the dynamics of concentrated vortices, 
see for instance \cite{D24}, except in particularly stable configurations \cite{DG24}.
\end{remark}

We now consider the different situation where the initial vorticity
is not a Dirac mass, but a Gaussian vortex of circulation $\Gamma > 0$ and
small characteristic length $\ell_0 > 0$. To facilitate the comparison with
the previous results, it is convenient to fix an initial time $t_0 \in (0,T)$ and
to assume that $\ell_0 = \sqrt{\nu t_0}$. Our second  main result can be stated
as follows. 

\begin{theorem}\label{thm2}
Fix $\Gamma > 0$, $z_0 \in \R^2$, and $t_0 \in (0,T)$. There exist positive
constants $K_2,\delta_2,c_2$ such that, if $0 < \nu/\Gamma < \delta_2$, the unique
solution of \eqref{eq:NSf} and \eqref{eq:BS} with initial data
\begin{equation}\label{eq:inGauss}
  \omega(x,t_0) \,=\, \frac{\Gamma}{\nu t_0}\,\Omega_0\Bigl(\frac{x-z_0}{\sqrt{
  \nu t_0}}\Bigr)\,, \qquad \forall\,x \in \R^2\,,
\end{equation}
satisfies, for all $t \in [t_0,T]$, the estimate
\begin{equation}\label{eq:thm2}
  \frac{1}{\Gamma}\int_{\R^2}\Bigl|\,\omega(x,t) - \omega_\app\bigl(\Gamma,\sqrt{\nu t},z(t),f(t)
  \,; x\bigr)\Bigr|\dd x \,\le\, K_2\,\epsilon(t)^2 \biggl\{\delta^{1/6}
  \Bigl(\log\frac{1}{\delta}\Bigr)^{1/2} \!+ \Bigl(\frac{t_0}{t} \Bigr)^{\beta}\biggr\}\,, 
\end{equation}
where $\epsilon(t) = \sqrt{\nu t}/d$, $d = \sqrt{\Gamma T_0}$, $\delta = \nu/\Gamma$, 
$\beta = c_2\delta^{-1/3}$, and $z(t)$ is the unique solution of the ODE \eqref{eq:modifz}
with initial condition $z(t_0) = z_0$. 
\end{theorem}

\begin{remark}\label{rem:initom}
It is important to realize that the left-hand side of \eqref{eq:thm2} does not
vanish at initial time $t_0$, unless the strain rates $a_0 := a_{f(t_0)}(z_0)$ and
$b_0 := b_{f(t_0)}(z_0)$ are both equal to zero. Indeed, it follows from
\eqref{def:omapp} and \eqref{eq:inGauss} that
\[
  \omega(x,t_0) - \omega_\app\bigl(\Gamma,\sqrt{\nu t_0},z_0,f(t_0)\,; x\bigr)
  \,=\, w_2\biggl(\frac{|x-z_0|}{\sqrt{\nu t_0}}\biggr)
  \Bigl(b_0\cos(2\theta) - a_0\sin(2\theta)\Bigr)\,,
\]
and the $L^1$ norm of the right-hand side is proportional to $\nu t_0(a_0^2 + b_0^2)^{1/2}$. 
In that sense our initial data \eqref{eq:inGauss} are ill-prepared if $(a_0,b_0) \neq (0,0)$:
being radially symmetric around the point $z_0$, they do not take into account the strain
of the external velocity field $f(\cdot,t_0)$. 
\end{remark}

Since $\beta = c_2\delta^{-1/3}$ we have $(t_0/t)^\beta \le \delta$ when
$t \ge t_0(1+\tau_\delta)$, where $\tau_\delta = c_3 \delta^{1/3}\log(1/\delta)$
for some $c_3 > 0$. The right-hand side of \eqref{eq:thm2} is therefore of
size $\epsilon(t)^2 \delta^{1/6}\bigl( \log(1/\delta)\bigr)^{1/2}$ as soon as
$t \ge t_0(1+\tau_\delta)$. In other words, the solution of \eqref{eq:NSf}
rapidly relaxes towards the approximate solution \eqref{def:omapp},
which takes into account the effect of the external strain, and remains close to
it up to the final time $T$. This description agrees with the numerical observations
in Figure~\ref{fig:relax}. That the relaxation rate $\beta$ depends on the
inverse Reynolds number $\delta = \nu/\Gamma$ is a consequence of the {\em
  enhanced dissipation} effect in the vortex core, see \cite{Ga18} and
Section~\ref{sec4}. 

Theorem~\ref{thm2} can be seen as an extension of Theorem~\ref{thm1}, in the sense
that the latter is obtained from the former by taking, at least formally,
the limit $t_0 \to 0$. This connection can be made rigorous if we write
the approximation formula \eqref{eq:thm2} in a slightly more precise form,
see Section~\ref{sec4}. The comparison of \eqref{eq:thm1} and \eqref{eq:thm2} also
shows that the solution starting from a Dirac mass can be considered as a
canonical model for the deformation of a concentrated vortex in an
external field, in the sense that it attracts solutions of \eqref{eq:NSf}
starting from ill-prepared initial data. 

\begin{remark}\label{rem:nonGauss}
In the context of Theorem~\ref{thm2}, it is perfectly natural to start with a
radially symmetric vortex, but a priori there is no reason to restrict oneself
to the Gaussian case. As a matter of fact, numerical experiments show that
relaxation to the well-prepared solution occurs for a large class of initial
profiles, even though the damped oscillations that are observed in the transient
period after initial time strongly depend on the choice of the profile
\cite{LDV02}. In this paper we consider the particular initial data \eqref{eq:inGauss}
because we want to use the enhanced dissipation estimates of \cite{LWZ20}, which have
been established so far only in the Gaussian case.
\end{remark}

The proof of our results relies on the construction of an approximate solution
of the initial value problem in self-similar variables, which is performed in
Section~\ref{sec2}. This part of the argument closely follows the previous works
\cite{Ga11,DG24} where particular situations were considered. The proof of
Theorem~\ref{thm1} is carried out in Section~\ref{sec3}, first under the
simplifying assumption that $T/T_0 \ll 1$, and then for any $T > 0$. In both
cases the desired control on the solution is obtained by an energy estimate in
some weighted $L^2$ space, but the construction of the weight function is much
more complicated if $T$ is not small compared to $T_0$. Improving upon the
results of \cite{Ga11}, we construct a Gaussian-like weight that provides
an accurate control on the solution as far as the decay at infinity is
concerned, and implies in particular the $L^1$ estimate \eqref{eq:thm1}. 
In Section~\ref{sec4}, we show how these arguments can be combined with
the enhanced dissipation estimates obtained by Li, Wei, and Zhang \cite{LWZ20}
to yield a proof of Theorem~\ref{thm2}. Finally, a few auxiliary results
are collected in the Appendix. In particular, we clarify the link between our
approximate solution \eqref{def:omapp} and the Burgers vortex in an asymmetric
strain, and we investigate the motion of the center of vorticity under the
assumptions of Theorem~\ref{thm1}.

\section{Self-similar variables and approximate solution}\label{sec2}

We first explain the common strategy in the proofs of Theorems~\ref{thm1}
and \ref{thm2}. Fix $\Gamma > 0$, $z_0 \in \R^2$, and let $\omega(x,t)$
be the solution of \eqref{eq:NSf} and \eqref{eq:BS} satisfying either
\eqref{eq:omunique} or \eqref{eq:inGauss}. In both cases, the solution
is sharply concentrated near a time-dependent point $z(t) \in \R^2$
if the viscosity $\nu > 0$ is sufficiently small. To desingularize the
problem, it is useful to make the self-similar change of coordinates
\begin{equation}\label{eq:OmU}
  \omega(x,t)  \,=\, \frac{\Gamma}{\nu t}\,\Omega\biggl(\frac{x - z(t)}{
  \sqrt{\nu t}},t\biggr)\,, \qquad
  u(x,t) \,=\, \frac{\Gamma}{\sqrt{\nu t}}\,U\biggl(\frac{x - z(t)}{
  \sqrt{\nu t}},t\biggr)\,.
\end{equation}
In what follows we denote
\begin{equation}\label{eq:xidef}
  \xi \,=\, \frac{x - z(t)}{\sqrt{\nu t}}\,, \qquad
  \delta \,=\, \frac{\nu}{\Gamma}\,, \qquad
  \epsilon \,=\, \frac{\sqrt{\nu t}}{d}\,, \qquad
  d = \sqrt{\Gamma T_0}\,.
\end{equation}
The new space variable $\xi$ describes the position with respect to the vortex
center $z(t)$ measured in units of the diffusion length $\sqrt{\nu t}$. As already
explained, the small parameter $\delta$ is the inverse Reynolds number,
and the time-dependent aspect ratio $\epsilon$ compares the size
$\sqrt{\nu t}$ of the vortex core to the effective size $d$ of the
vortex.

As is easily verified, the evolution equation satisfied by the
rescaled vorticity $\Omega(\xi,t)$ is 
\begin{equation}\label{eq:Omevol}
  t\partial_t\Omega(\xi,t) + \biggl\{\frac{1}{\delta}\,U(\xi,t) 
  + \sqrt{\frac{t}{\nu}}\,\Bigl(f\bigl(z(t) + \sqrt{\nu t}\,\xi,t\bigr) - z'(t)\Bigr)
  \biggr\}\cdot\nabla\Omega(\xi,t) \,=\, \cL\Omega(\xi,t)\,,
\end{equation}
where $\cL$ is the diffusion operator defined by
\begin{equation}\label{def:cL}
  \cL \,=\, \Delta_\xi + \frac{1}{2}\,\xi\cdot\nabla_\xi + 1\,.
\end{equation}
The position $z(t)$ of the vortex center is unknown at this stage, but will be
chosen so as to minimize the quantity $f(z(t) + \sqrt{\nu t}\,\xi,t) - z'(t)$ in
an appropriate sense. The natural choice $z'(t) = f(z(t),t)$ gives the leading order
approximation, but higher order corrections will be needed to achieve the desired
precision. We also observe that the rescaled velocity $U(\xi,t)$ is divergence-free
and satisfies $\partial_1 U_2 - \partial_2 U_1 = \Omega$, which means that $U$ is
obtained from $\Omega$ by the Biot-Savart formula \eqref{eq:BS}, namely $U = \BS[\Omega]$. 

It is important to realize that \eqref{eq:Omevol} is not a regular
evolution equation at time $t = 0$, due to the singular time derivative
$t\partial_t\Omega$ in the left-hand side. Nevertheless, if we adapt to the
present case the results of \cite{GW05,GaGa05}, which hold for $f(z,t) = 0$ and
$z(t) = 0$, it is not difficult to show that \eqref{eq:Omevol} has a unique
(mild) solution $\Omega \in C^0((0,T],L^1(\R^2) \cap L^\infty(\R^2))$ that
satisfies $\|\Omega(\cdot,t) - \Omega_0\|_{L^1} \to 0$ as $t \to 0$.
This is precisely the solution we study in Theorem~\ref{thm1}. Note that
the Gaussian profile \eqref{eq:LambOseen} is, up to normalization, the only
possibility for the initial vorticity at time $t = 0$. The situation
considered in Theorem~\ref{thm2} is much different: the Cauchy problem
for equation~\eqref{eq:Omevol} is well-posed at any positive
time $t_0 > 0$, and we could therefore choose arbitrary initial data at
$t = t_0$. However, for reasons that are explained in Remark~\ref{rem:nonGauss}
above, our choice is to take the same initial vorticity $\Omega_0$ as
in Theorem~\ref{thm1}.

Since $\delta = \nu/\Gamma$ and $\Gamma > 0$ is fixed, it is clear that the
evolution equation \eqref{eq:Omevol} becomes highly singular in the vanishing
viscosity limit $\nu \to 0$, and this is actually the main problem in the
proof of both Theorems~\ref{thm1} and \ref{thm2}. To overcome this
difficulty, we use the approach introduced in \cite{Ga11,GaS24,DG24}
which relies on the construction of an approximate solution of the form:
\begin{equation}\label{def:Omapp}
  \begin{split}
  \Omega_\app(\xi,t) \,&=\, \Omega_0(\xi) + \epsilon(t)^2\,\Omega_2(\xi,t)
  + \epsilon(t)^3\,\Omega_3(\xi,t) +  \epsilon(t)^4\,\Omega_4(\xi,t)\,,\\[1mm]
  U_\app(\xi,t) \,&=\, U_0(\xi) + \epsilon(t)^2\,U_2(\xi,t)
  + \epsilon(t)^3\,U_3(\xi,t) +  \epsilon(t)^4\,U_4(\xi,t)\,,
  \end{split}
\end{equation}
where $\epsilon(t) = \sqrt{\nu t}/d$. The vorticity profiles $\Omega_j$
and the velocity profiles $U_j = \BS[\Omega_j]$ depend on the small parameter
$\delta > 0$, and will be determined so that equality \eqref{eq:Omevol} holds
up to corrections terms of size $\cO(\epsilon^5/\delta + \delta \epsilon^2)$.

\begin{remark}\label{rem:eps}
It is not obvious at this point that the aspect ratio $\epsilon(t)$
is the correct parameter for our perturbative expansion, since
it does not appear explicitly in the evolution equation \eqref{eq:Omevol}.
However this parameter naturally occurs when expanding the external
velocity field in \eqref{eq:Omevol}, as we now demonstrate.
The calculation also shows that there is no term 
proportional to $\epsilon(t)$ in \eqref{def:Omapp} if we assume 
that $z'(t) = f(z(t),t) + \cO\bigl(\epsilon(t)\bigr)$, as we shall
always do.
\end{remark}

\subsection{Expansion of the external velocity}\label{ssec21}

We first rewrite the evolution equation \eqref{eq:Omevol} in the
equivalent form
\begin{equation}\label{eq:Omevol2}
  \delta t\partial_t\Omega(\xi,t) + \Bigl(U(\xi,t) + E(f,z\,; \xi,t)\Bigr)
  \cdot \nabla \Omega(\xi,t) \,=\, \delta\cL\Omega(\xi,t)\,,
\end{equation}
where $E(f,z\,; \xi,t) \,=\, \delta \sqrt{t/\nu} \bigl(f(z(t) +
\sqrt{\nu t}\,\xi,t) - z'(t)\bigr)$. In view of \eqref{eq:xidef},
we have 
\[
  \delta \sqrt{\frac{t}{\nu}} \,=\, \frac{\sqrt{\nu t}}{\Gamma} \,=\,
  \epsilon(t)\,\frac{d}{\Gamma} \,=\, \epsilon(t)\,\frac{T_0}{d}\,,
\]
so that $E(f,z\,; \xi,t) = \cO(\epsilon)$. To obtain a better approximation, we
use a fourth-order Taylor expansion of the quantity $f(z(t) + \sqrt{\nu t}\,\xi,t)$
in powers of the diffusion length $\sqrt{\nu t} = d\epsilon(t)$, which leads to
the following result.

\begin{lemma}\label{lem:fexpand}
For all $(\xi,t) \in \R^2 \times [0,T]$ we have the expansion
\begin{equation}\label{eq:Efz}
  E(f,z\,; \xi,t) \,=\, \sum_{k=1}^4 \epsilon(t)^k\,E_k(f,z\,; \xi,t) + \cR_E(f,z\,; \xi,t)\,,
\end{equation}
where
\begin{equation}\label{def:Ek}
  \begin{split}
  E_1(f,z\,; \xi,t) \,&=\, \frac{T_0}{d}\,\bigl(f(z(t),t) - z'(t)\bigr)\,,\\
  E_2(f,z\,; \xi,t) \,&=\, T_0\,\Diff f(z(t),t)[\xi]\,,\\
  E_3(f,z\,; \xi,t) \,&=\, \frac12\, T_0\,d\,\Diff^2 f(z(t),t)[\xi,\xi]\,,\\
  E_4(f,z\,; \xi,t) \,&=\, \frac16\, T_0\,d^2\,\Diff^3 f(z(t),t)[\xi,\xi,\xi]\,.
  \end{split}
\end{equation}
Moreover the remainder in \eqref{eq:Efz} satisfies the estimate
\[
  \bigl|\cR_E(f,z\,; \xi,t)\bigr| \,\le\, \frac{\epsilon(t)^5}{24}\,
  T_0\,d^3\,\|\Diff^4f\|_{L^\infty(\R^2)}\,|\xi|^4\,,
  \qquad \forall\, (\xi,t) \in \R^2 \times [0,T]\,.
\]
\end{lemma}

\begin{proof} The proof is a straightforward calculation which can be omitted.
\end{proof}
  
As is clear from \eqref{def:Ek}, if $z'(t) = f(z(t),t)$, the leading term $E_1$
in \eqref{eq:Efz} vanishes, so that $E(f,z\,; \xi,t) = \cO(\epsilon^2)$.  For
any $k \in \{1,2,3\}$, the term $E_{k+1}$ is a homogeneous polynomial of degree $k$
in the variable $\xi \in \R^2$, the coefficients of which are linear
combinations of $k$-th order derivatives of $f$ evaluated at the point
$(z(t),t)$. The following more precise information about $E_2$ and $E_3$ will be
needed.

\begin{lemma}\label{lem:E2}
For $(\xi,t) \in \R^2 \times [0,T]$ the quantity $E_2 := E_2(f,z\,; \xi,t)$
satisfies    
\begin{equation}\label{eq:E2}
  E_2\,=\, T_0\Bigl(a(t)\cos(2\theta)+b(t)\sin(2\theta)\Bigr)\xi
  + T_0 \Bigl(b(t)\cos(2\theta)-a(t)\sin(2\theta) + c(t)\Bigr)\xi^\perp\,,
\end{equation}
where $\theta$ is the polar angle of the variable $\xi \in \R^2$, and
$a(t), b(t), c(t)$ denote the following derivatives of $f$ evaluated
at $(z(t),t)$: 
\[
  a \,=\, \frac12 \bigl(\partial_1 f_1 - \partial_2 f_2\bigr)\,,\qquad
  b \,=\, \frac12 \bigl(\partial_1 f_2 + \partial_2 f_1\bigr)\,,\qquad
  c \,=\, \frac12 \bigl(\partial_1 f_2 - \partial_2 f_1\bigr)\,.
\]
\end{lemma}

\begin{proof}
Since $f$ is divergence-free, the Jacobian matrix $\Diff f(z(t),t)$
takes the form
\[
  \Diff f \,=\, \begin{pmatrix} \partial_1 f_1 & \partial_2 f_1 \\
  \partial_1 f_2 & \partial_2 f_2 \end{pmatrix} \,=\,
  \begin{pmatrix} a & b-c\\ b+c & -a\end{pmatrix}\,,  
\]
where $a,b,c$ are as in the statement. We deduce that
\[
  \xi\cdot \Diff f[\xi] \,=\, a\bigl(\xi_1^2-\xi_2^2\bigr) +2b\xi_1 \xi_2\,, \qquad
  \xi^\perp\cdot \Diff f[\xi] \,=\, b\bigl(\xi_1^2 - \xi_2^2\bigr) -2a\xi_1\xi_2
  + c\bigl(\xi_1^2 + \xi_2^2\bigr)\,,
\]
and \eqref{eq:E2} immediately follows. 
\end{proof}

\begin{lemma}\label{lem:E3}
For $(\xi,t) \in \R^2 \times [0,T]$ the quantity $E_3 := E_3(f,z\,; \xi,t)$
satisfies    
\begin{equation}\label{eq:E3}
    \xi\cdot E_3 \,=\, T_0 d \,|\xi|^3 \Bigl(\frac18 \Delta f_1 \cos(\theta)
    + \frac18 \Delta f_2 \sin(\theta) + A\cos(3\theta) + B\sin(3\theta)
    \Bigr)\,,
\end{equation}
where $\theta$ is the polar angle of the variable $\xi \in \R^2$, and $A = \frac38
\partial_1^2 f_1 - \frac18 \partial_2^2 f_1$, $B = \frac18 \partial_1^2 f_2
- \frac38 \partial_2^2 f_2$. All derivatives of $f$ are evaluated at $(z(t),t)$. 
\end{lemma}

\begin{proof}
From the definition of $E_3$ in \eqref{def:Ek} we readily obtain
\[
  \frac{\xi\cdot E_3}{T_0d} \,=\, \frac{\xi_1}{2}\Bigl(\xi_1^2\partial_1^2 f_1 
  + 2\xi_1\xi_2\partial_1\partial_2 f_1 + \xi_2^2\partial_2^2 f_1\Bigr) +
  \frac{\xi_2}{2}\Bigl(\xi_1^2 \partial_1^2 f_2 + 2\xi_1 \xi_2 \partial_1\partial_2
  f_2  + \xi_2^2\partial_2^2 f_2\Bigr)\,.
\]
Introducing polar coordinates $\xi = |\xi|\bigl(\cos(\theta),\sin(\theta)\bigr)$
and using the elementary identities
\begin{align*}
  \cos^3(\theta) \,&=\, \frac34\,\cos(\theta) + \frac14\,\cos(3\theta)\,,
  \hskip 30pt \cos^2(\theta)\sin(\theta) \,=\, \frac14\,\sin(\theta)
  + \frac14\,\sin(3\theta)\,,\\
  \sin^3(\theta) \,&=\, \frac34\,\sin(\theta) - \frac14\,\sin(3\theta)\,,
  \hskip 32pt \cos(\theta)\sin^2(\theta) \,=\, \frac14\,\cos(\theta)
  - \frac14\,\cos(3\theta)\,,
\end{align*}
we arrive at \eqref{eq:E3} after straightforward calculations. 
\end{proof}

\subsection{Functional framework}\label{ssec22}

This section is almost entirely taken from the previous works
\cite{Ga11,GaS24,DG24}, and is reproduced here for the reader's
convenience. Our goal is to introduce the function spaces in which the
approximate solution \eqref{def:Omapp} will be constructed, and to study the
properties of a pair of linear operators in that framework. We first define the
weighted $L^2$ space
\begin{equation}\label{def:cY}
  \cY \,=\, \biggl\{\Omega\in L^2(\R^2)\,;\, \int_{\R^2}
  |\Omega(\xi)|^2\,e^{|\xi|^2/4}\,\dd \xi < \infty\biggr\}\,,
\end{equation}
which is a Hilbert space equipped with the natural scalar product.  If we use
polar coordinates $(r,\theta)$ such that $\xi = \bigl( r\cos\theta,r\sin\theta\bigr)$,
we have the direct sum
decomposition
\begin{equation}\label{Ydecomp}
  \cY \,= \,\bigoplus_{n=0}^{\infty}\,\cY_n\,,
\end{equation}
where $\cY_0$ is the subspace of all radially symmetric functions in $\cY$, and,
for each $n \ge 1$, the subspace $\cY_n$ contains all $\Omega \in \cY$ of the
form $\Omega = a(r)\cos(n\theta) + b(r)\sin(n\theta)$.  It is clear that the
decomposition \eqref{Ydecomp} is orthogonal, in the sense that
$\cY_n \perp \cY_{n'}$ if $n \neq n'$. We also introduce the dense subset
$\cZ \subset \cY$ defined by
\begin{equation}\label{def:cZ}
  \cZ \,=\, \Bigl\{\Omega:\R^2\to \R \,;\, \xi \mapsto e^{|\xi|^2/4}
  \Omega(\xi) \in \cS_*(\R^2)\Bigr\}\,,
\end{equation}
where $\cS_*(\R^2)$ is the space of smooth functions on $\R^2$ with moderate
growth at infinity. In other words a function $\Omega \in C^\infty(\R^2)$
belongs to $\cS_*(\R^2)$ if, for any multi-index $\alpha=(\alpha_1,\alpha_2)
\in \N^2$ there exists $C>0$ and $N\in \N$ such that $|\partial^{\alpha}
\Omega(\xi)| \le C(1+|\xi|)^N$ for all $\xi\in \R^2$.

The linear operators we are interested in are the diffusion operator $\cL$
introduced in \eqref{def:cL} and the advection operator $\Lambda$ defined by the
formula
\begin{equation}\label{def:Lambda}
  \Lambda \Omega \,=\, U_0\cdot \nabla \Omega + \BS[\Omega]\cdot
  \nabla \Omega_0\,,
\end{equation}
where $\Omega_0,U_0$ are given by \eqref{eq:OmU0}. If we consider the
operators $\cL,\Lambda$ as acting on the function space \eqref{def:cY} with
maximal domain, we have the following result: 

\begin{proposition}\label{prop:Lam} {\bf\cite{GW02,GW05,Ma11}}\\[0.5mm]
1) The linear operator $\cL$ is {\em self-adjoint} in $\cY$, with purely 
discrete spectrum
\[
  \sigma(\cL) \,=\, \Bigl\{-\frac{n}{2}\,\Big|\, n = 0,1,2, \dots\Bigr\}\,.
\]
The kernel of $\cL$ is one-dimensional and spanned by the Gaussian function
$\Omega_0$. More generally, for any $n \in \N$, the eigenspace corresponding
to the  eigenvalue $\lambda_n = -n/2$ is spanned by the $n+1$ Hermite functions
$\partial^\alpha \Omega_0$ where $\alpha = (\alpha_1,\alpha_2) \in \N^2$ and 
$\alpha_1 + \alpha_2 = n$. \\[1mm]
2) The linear operator $\Lambda$ is {\em skew-adjoint} in $\cY$, so that 
$\Lambda^* = -\Lambda$. Moreover,
\begin{equation}\label{eq:KerLam}
  \Ker(\Lambda) \,=\, \cY_0 \oplus \bigl\{\beta_1\partial_1 \Omega_0 + \beta_2
  \partial_2 \Omega_0 \,\big|\, \beta_1, \beta_2 \in \R\bigr\}\,,
\end{equation}
where $\cY_0 \subset \cY$ is the subspace of all radially symmetric 
elements of $\cY$. 
\end{proposition}

Another important feature of both operators $\cL,\Lambda$ is rotation
invariance. As is easily verified, if $\Omega \in \cY_n \cap \cZ$ for
some $n \ge 0$, then $\Omega$ belongs to the domain of $\cL$ and
$\cL \Omega \in \cY_n \cap \cZ$. The same property holds for the
integro-differential operator $\Lambda$, and can be established using
the definitions \eqref{eq:OmU0} and \eqref{def:Lambda} together with the
properties of the Biot-Savart law. 

As we shall see in the next section, the construction of the approximate
solution \eqref{def:Omapp} requires solving elliptic equations of
the form 
\begin{equation}\label{eq:Lamell}
  \Lambda \Omega \,=\, F\,, \qquad \text{for some}\,~ F \in \cY\,.
\end{equation}
Since the operator $\Lambda$ is skew-adjoint in $\cY$ we have $\Ker(\Lambda)^\perp =
\overline{\Ran(\Lambda)}$, where $\Ran(\Lambda)$ is the range of $\Lambda$.
Thus a necessary condition for the solvability of \eqref{eq:Lamell} is that
$F \perp \Ker(\Lambda)$. In view of \eqref{eq:KerLam}, this is equivalent to
\begin{equation}\label{eq:solvcond}
  \int_0^{2\pi} F\bigl(r\cos(\theta),r\sin(\theta)\bigr)\dd \theta \,=\, 0
  \quad \forall\,r > 0\,, \quad \text{and}\quad \int_{\R^2} \xi_j F(\xi)\dd \xi
  \,=\, 0 \quad \forall j \in \{1,2\}\,.
\end{equation}
It turns out that, in the subspace \eqref{def:cZ}, the solvability conditions
above are also sufficient. This is the content of the following result,
whose proof is recalled in Section~\ref{ssecA2}.

\begin{proposition}\label{prop:solv}
If $F \in \cZ \cap \Ker(\Lambda)^\perp$, there is a unique $\Omega \in \cZ
\cap \Ker(\Lambda)^\perp$ such that $\Lambda\Omega = F$.
\end{proposition}

\subsection{Construction of the approximate solution}\label{ssec23}

We now explain how to construct the vorticity profiles $\Omega_2, \Omega_3, \Omega_4$
in \eqref{def:Omapp} so as to obtain a precise approximate solution of
\eqref{eq:Omevol2}. From now on, we denote by $z(t)$ the unique solution of the
ODE \eqref{eq:modifz}, and we use the decomposition $z'(t) = z_0'(t) + \epsilon^2 z_2'(t)$,
where
\begin{equation}\label{eq:zdecomp}
  z_0'(t) \,:=\, f\bigl(z(t),t\bigr)\,, \qquad
  z_2'(t) \,:=\, d^2 \Delta f\bigl(z(t),t\bigr)\,.
\end{equation}
Taking into account the correction term $z_2'$ in \eqref{eq:zdecomp}, the velocity
expansion \eqref{eq:Efz} can be written in the slightly modified form 
\begin{equation}\label{eq:Efz2}
  E(f,z\,; \xi,t) \,=\, \sum_{k=1}^4 \epsilon^k\,\hat E_k(f,z\,; \xi,t) + \cR_E(f,z\,; \xi,t)\,,
\end{equation}
where $\hat E_1 = 0$, $\hat E_2 = E_2$, $\hat E_4 = E_4$, and
\begin{equation}\label{def:hatE3}
  \hat E_3(f,z\,; \xi,t) \,=\, T_0d\,\Bigl(\frac12\,\Diff^2 f(z(t),t)[\xi,\xi]
  - \Delta f(z(t),t)\Bigr)\,.
\end{equation}

If $\Omega_\app$ is an approximate solution of \eqref{eq:Omevol2}, we define
\begin{equation}\label{eq:remdef}
  \cR_\app \,:=\, \delta\bigl(t\partial_t\Omega_\app - \cL\Omega_\app\bigr) +
  \bigl(U_\app + E(f,z)\bigr)\cdot \nabla \Omega_\app\,.
\end{equation}
Our goal is to choose $\Omega_\app$ and $U_\app = \BS[\Omega_\app]$ so as to minimize
the remainder $\cR_\app$. Since the external velocity $E(f,z)$ has a power
series expansion in the (time-dependent) parameter $\epsilon$, it is natural
to expand $\Omega_\app$ and $U_\app$ in powers of $\epsilon$ too, as in
\eqref{def:Omapp}. Note that equation \eqref{eq:Omevol2} also involves the small
parameter $\delta$, which means that the profiles $\Omega_k$ and $U_k$ in
\eqref{def:Omapp} are functions of $\delta$. This dependence will
not be indicated explicitly, but it is understood that all quantities
we consider are uniformly bounded as $\delta \to 0$. 

To obtain a more explicit expression of $\cR_\app$, we insert the expansions
\eqref{def:Omapp} and \eqref{eq:Efz2} into \eqref{eq:remdef}, and we use the facts
that $\delta t = \epsilon^2 T_0$ and $t\partial_t \epsilon^k = (k/2)\epsilon^k$ for
all $k \in \N$, see \eqref{eq:xidef}. Recalling also the definition \eqref{def:Lambda}
of the operator $\Lambda$, we arrive at the decomposition
\begin{equation}\label{eq:remdef2}
  \cR_\app \,=\, \sum_{k=2}^8 \epsilon^k \cR_k + \cR_E \cdot \nabla\Omega_\app\,,
\end{equation}
where
\begin{equation}\label{eq:Rkdef}
  \begin{split}
  \cR_2 \,&=\, \delta\bigl(1-\cL\bigr)\Omega_2 + \Lambda\Omega_2 + E_2\cdot
  \nabla \Omega_0\,, \\
  \cR_3 \,&=\, \delta\bigl({\TS \frac32}-\cL\bigr)\Omega_3 + \Lambda\Omega_3 +
  \hat E_3\cdot\nabla \Omega_0\,, \\[1mm]
  \cR_4 \,&=\, \delta\bigl(2-\cL\bigr)\Omega_4 + \Lambda\Omega_4 +
  E_4\cdot\nabla \Omega_0 + \bigl(U_2 + E_2\bigr)\cdot\nabla\Omega_2 +
  T_0\partial_t \Omega_2\,.
  \end{split}
\end{equation}
The exact expression of the higher-order terms $\cR_k$ for $k \ge 5$ is not important
for our analysis. We now determine the profiles $\Omega_2, \Omega_3, \Omega_4$ so as
to minimize the quantities $\cR_2, \cR_3, \cR_4$. Since the expressions $E_2$,
$\hat E_3$, $E_4$ involve derivatives of the external field $f$, which is
time-dependent, the vorticity profiles $\Omega_j$ also depend on time, as
indicated in \eqref{def:Omapp}. However, they are determined by solving
``elliptic'' equations which can be studied at frozen time. 

\subsubsection{Second order vorticity profile}\label{sssec231}
We take $\Omega_2 = \bar\Omega_2 + \delta\tilde\Omega_2$, where
$\bar\Omega_2,\tilde\Omega_2 \in \cY_2 \cap \cZ$ satisfy
\begin{equation}\label{def:Om2}
  \Lambda\bar\Omega_2 + E_2 \cdot\nabla\Omega_0 \,=\, 0\,, \qquad
  \text{and}\qquad \Lambda\tilde\Omega_2 + (1-\cL)\bar\Omega_2 \,=\, 0\,.
\end{equation}
This is indeed possible because, using \eqref{eq:E2} and the identity
$\nabla\Omega_0 = -(\xi/2)\Omega_0$, we see that
\begin{equation}\label{eq:E2nab}
  E_2 \cdot\nabla\Omega_0 \,=\, -\frac12\,T_0\Omega_0 |\xi|^2 \Bigl(
  a(t)\cos(2\theta) + b(t)\sin(2\theta)\Bigr)\,,
\end{equation}
where $a,b$ are as in Lemma~\ref{lem:E2}. In view of the definitions
\eqref{def:cY}--\eqref{def:cZ}, this expression shows that $E_2 \cdot\nabla
\Omega_0 \in \cY_2 \cap \cZ$. Applying Lemma~\ref{lem:Lambda}, we deduce
that there exists a unique $\bar\Omega_2 \in \cY_2 \cap \cZ$ such that
$\Lambda\bar\Omega_2 + E_2 \cdot\nabla\Omega_0 = 0$. Now, as already
observed, the diffusion operator $\cL$ leaves the subspace $\cY_2 \cap \cZ$
invariant. So, applying Lemma~\ref{lem:Lambda} again, we see that there
exists a unique $\tilde\Omega_2 \in \cY_2 \cap \cZ$ such that
$\Lambda\tilde\Omega_2 + (1-\cL)\bar\Omega_2 = 0$. Finally, combining
\eqref{eq:Rkdef} and \eqref{def:Om2}, we obtain
\begin{equation}\label{eq:R2est}
  \cR_2 \,=\, \delta^2 (1 - \cL)\tilde\Omega_2 \,=\, \cO(\delta^2)\,.
\end{equation}

\begin{remark}\label{rem:w2}
Using \eqref{eq:E2nab} and the formulas reproduced in Section~\ref{ssecA2},
it is straightforward to verify that $\bar\Omega_2(\xi,t) = T_0 w_2(|\xi|)
\bigl(a(t)\sin(2\theta) - b(t)\cos(2\theta)\bigr)$ where
\[
  w_2(r) \,=\, h(r)\Bigl(\varphi_2(r) + \frac{r^2}{2}\Bigr)\,,
 \qquad   h(r) \,=\, \frac{r^2/4}{e^{r^2/4}-1}\,, \qquad r > 0\,,
\]
and $\varphi_2$ is the unique solution of the differential equation
\[
  -\varphi_2''(r) -\frac{1}{r}\,\varphi_2'(r) + \Bigl(\frac{4}{r^2} - h(r)\Bigr)
  \varphi_2(r) \,=\, \frac{r^2}{2}\,h(r)\,, \qquad r > 0\,,
\]
such that $\varphi_2(r) = \cO(r^2)$ as $r \to 0$ and $\phi(r) = \cO(r^{-2})$
as $r \to +\infty$. In particular $w_2(r) > 0$ for all $r > 0$, $w_2(r) =
\cO(r^2)$ as $r \to 0$, and $w_2(r) \sim (r^4/8)e^{-r^2/4}$ as $r \to +\infty$.
The graph of the function $w_2$ is represented in Figure~\ref{fig:w2}. 
\end{remark}

\begin{figure}
    \centering
    \includegraphics[width=0.45\linewidth]{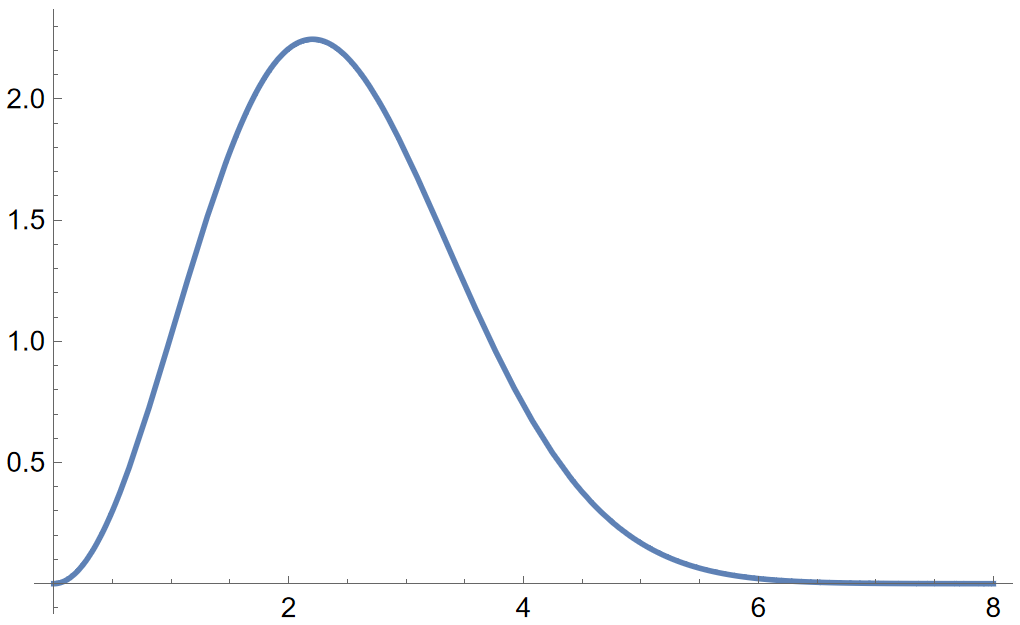}
    \put(-15,12){$r$}
    \put(-120,100){$w_2(r)$}
    \caption{The function $w_2$, which enters the definition of the approximate
     solution \eqref{def:omapp} and describes to leading order the deviation
     of the vorticity distribution from the Gaussian profile, is represented as a 
     function of the radius $r = |\xi|$.}\label{fig:w2}
\end{figure}

\subsubsection{Third order vorticity profile}\label{sssec232}
Let $\cY_1' = \cY_1 \cap \Ker(\Lambda)^\perp$ be the subspace introduced 
in \eqref{eq:Y1Kerperp}. We take $\Omega_3 = \bar\Omega_3 + \delta\tilde\Omega_3$,
where $\bar\Omega_3,\tilde\Omega_3 \in \bigl(\cY_1' \oplus \cY_3)\cap \cZ$ satisfy
\begin{equation}\label{def:Om3}
  \Lambda\bar\Omega_3 + \hat E_3 \cdot\nabla\Omega_0 \,=\, 0\,, \qquad
  \text{and}\qquad \Lambda\tilde\Omega_3 + \bigl({\TS \frac32}-\cL\bigr)
  \bar\Omega_3 \,=\, 0\,.
\end{equation}
The strategy for solving \eqref{def:Om3} is the
same as before. Using \eqref{def:hatE3} and Lemma~\ref{lem:E3}, we easily
find
\begin{equation}\label{eq:hE3dot}
  \begin{split}
  \hat E_3\cdot\nabla\Omega_0  \,&=\, -\frac12 T_0 d \,\Omega_0|\xi|^3
  \Bigl(\frac18 \Delta f_1 \cos(\theta) + \frac18 \Delta f_2 \sin(\theta)
  + A\cos(3\theta) + B\sin(3\theta)\Bigr) \\
  &\quad\, + \frac12 T_0 d \,\Omega_0|\xi| \bigl(\Delta f_1 \cos(\theta) +
  \Delta f_2 \sin(\theta)\bigr)\,,
  \end{split}
\end{equation}
where the first line is the expression of $E_3\cdot\nabla\Omega_0$,
and the second one is the correction due to the additional velocity
$z_2'(t)$ of the vortex center. This shows that $E_3\cdot\nabla\Omega_0 \in
\bigl(\cY_1 \oplus \cY_3)\cap \cZ$, which is not quite sufficient since
we cannot invert the operator $\Lambda$ in the full subspace $\cY_1 \cap \cZ$.
Actually, thanks to the correction term in \eqref{eq:hE3dot}, we have
$E_3\cdot\nabla\Omega_0 \in \bigl(\cY_1' \oplus \cY_3)\cap \cZ$, because
a direct calculation shows that
\[
  \int_{\R^2} \xi_j \bigl(\hat E_3\cdot\nabla\Omega_0\bigr)\dd\xi
  \,=\, -\frac18 T_0 d\,\Delta f_j \int_0^\infty e^{-r^2/4}\Bigl(\frac{r^5}{8}
  - r^3\Bigr)\dd r \,=\, 0\,, \qquad \text{for }\,j = 1,2\,.
\]
Thus, applying Lemmas~\ref{lem:Lambda} and \ref{lem:Lambda2}, we conclude that
there exists a unique $\bar\Omega_3 \in \bigl(\cY_1' \oplus \cY_3)\cap \cZ$ such
that $\Lambda\bar\Omega_3 + \hat E_3 \cdot\nabla\Omega_0 = 0$.  The profile
$\tilde\Omega_3$ is then constructed as before, and we arrive at
\begin{equation}\label{eq:R3est}
  \cR_3 \,=\, \delta^2 \bigl({\TS \frac32} - \cL\bigr)\tilde\Omega_3
  \,=\, \cO(\delta^2)\,.
\end{equation}

\begin{figure}
    \centering
    \includegraphics[width=0.35\linewidth]{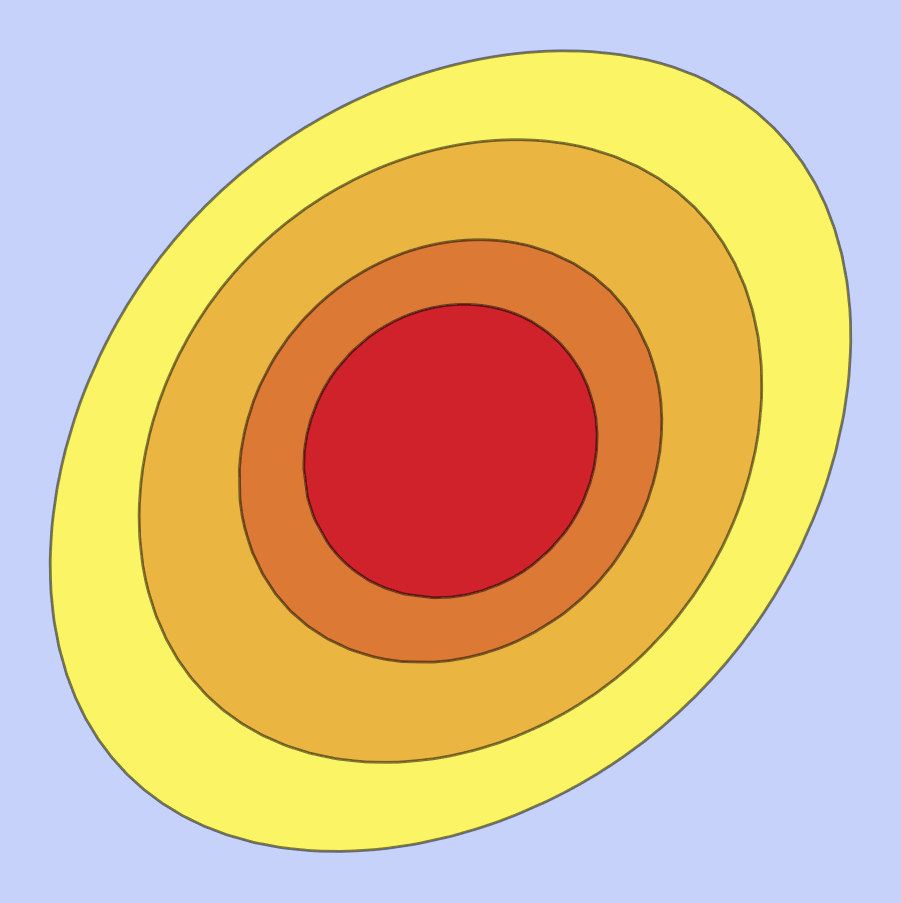} \hspace{5mm}
    \includegraphics[width=0.35\linewidth]{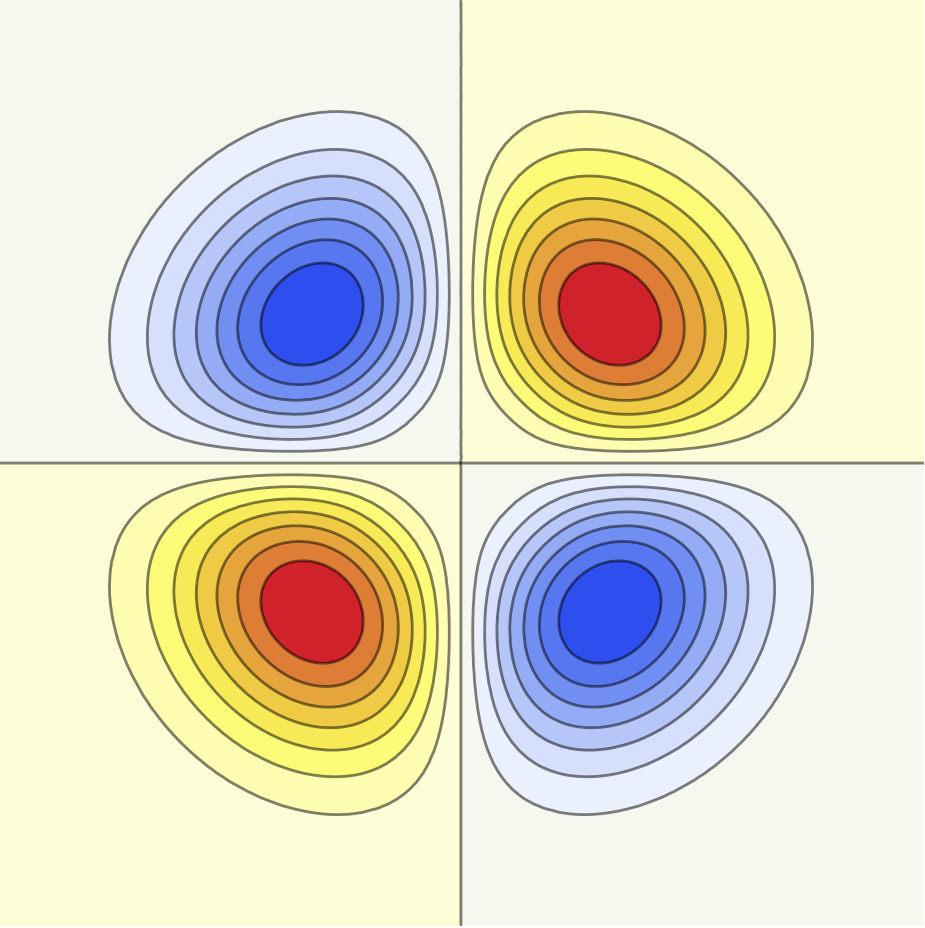}
    \caption{Level lines of the perturbation $\Omega_2: \xi \mapsto w_2(|\xi|)\sin(2\theta)$
    on the square $[-6,6]^2$ (right), and of the approximate solution $\Omega_0 +
     0.04*\Omega_2$ on the smaller square $[-1.4,1.4]^2$ (left).}
    \label{fig:levelLines}
\end{figure}

\subsubsection{Fourth order vorticity profile}\label{sssec233}
We start from the following observation.

\begin{lemma}\label{lem:4th} One has $E_4\cdot\nabla \Omega_0 +
\bigl(\bar U_2 + E_2\bigr)\cdot\nabla\bar\Omega_2 + T_0\partial_t \bar\Omega_2\in
\bigl(\cY_2 \oplus \cY_4\bigr) \cap \cZ$. 
\end{lemma}

\begin{proof}
Since $\nabla\Omega_0 = -(\xi/2)\Omega_0$, it follows from the definition
\eqref{def:Ek} that $E_4\cdot\nabla \Omega_0 = P_4 \Omega_0$ where
$P_4$ is a homogeneous polynomial of degree $4$ in the variable
$\xi = \bigl(r\cos(\theta),r\sin(\theta)\bigr) \in \R^2$. In particular
we have $E_4\cdot\nabla \Omega_0 \in \bigl(\cY_0 \oplus \cY_2 \oplus
\cY_4\bigr) \cap \cZ$. In addition, since $\Omega_0$ is radially symmetric,
we have for any $r> 0$:
\[
  \int_0^{2\pi} \bigl(E_4\cdot\nabla \Omega_0\bigr)\bigl(r\cos(\theta),
  r\sin(\theta)\bigr)\dd\theta \,=\, -\frac{r}{8\pi}\,e^{-r^2/4}
  \int_0^{2\pi} E_4\bigl(r\cos(\theta),r\sin(\theta)\bigr)\cdot e_r\dd\theta
  \,=\, 0\,,
\]
where in the last equality we used the divergence theorem and the fact that
$\nabla\cdot E_4 = 0$. This shows that $E_4\cdot\nabla \Omega_0$ has zero
radial average, hence zero projection onto the subspace $\cY_0$. 

Next we recall that $\bar\Omega_2(\xi,t) = T_0\,w_2(r) \bigl(a\sin(2\theta)
- b\cos(2\theta)\bigr)$, see Remark~\ref{rem:w2}. According to
\eqref{eq:Undef}, the associated velocity field takes the form
\[
  \bar U_2 \,=\, \frac{2T_0}{r}\,\varphi_2(r)\bigl(a\cos(2\theta)+ b\sin(2\theta)\bigr)\,e_r
  + T_0 \varphi_2'(r)\bigl(-a\sin(2\theta)+b\cos(2\theta)\bigr)\,e_\theta\,.
\]
By a direct calculation, we deduce that
\[
  \bar U_2\cdot\nabla\bar\Omega_2 \,=\, \frac{T_0^2}{r}\bigl(\varphi_2' w_2 -
  \varphi_2 w_2'\bigr)\Bigl((b^2-a^2)\sin(4\theta) + 2ab\cos(4\theta)\Bigr)
  \,\in\, \cY_4 \cap \cZ\,.
\]
Similarly, using the expression of $E_2$ in Lemma~\ref{lem:E2}, we obtain
\begin{align*}
  E_2\cdot\nabla\bar\Omega_2 \,&=\, \frac{T_0^2}{2}\bigl(2w_2 - r w_2'\bigr)
  \Bigl((b^2-a^2)\sin(4\theta) + 2ab\cos(4\theta)\Bigr) \\
  &\quad\, +2T_0^2 w_2 \bigl(ac\cos(2\theta) + bc\sin(2\theta)\bigr)
  \,\in\, \bigl(\cY_2 \oplus \cY_4\bigr) \cap \cZ\,.
\end{align*}
Finally, the expression above of $\bar\Omega_2$ shows that $T_0\partial_t \bar\Omega_2
\in \cY_2 \cap \cZ$. This concludes the proof. 
\end{proof}

According to Lemmas~\ref{lem:4th} and \ref{lem:Lambda}, there
exists a unique profile $\Omega_4 \in \bigl(\cY_2 \oplus \cY_4\bigr) \cap \cZ$
such that
\begin{equation}\label{def:Om4}
  \Lambda\Omega_4 + E_4\cdot\nabla \Omega_0 + \bigl(\bar U_2
  + E_2\bigr)\cdot\nabla\bar\Omega_2 + T_0\partial_t \bar\Omega_2 \,=\, 0\,.
\end{equation}
We that choice, we obviously have
\begin{equation}\label{eq:R4est}
  \cR_4 \,=\, \delta (2 - \cL)\Omega_4 + \delta \bigl(\bar U_2 + E_2\bigr)
  \cdot\nabla\tilde\Omega_2 + \delta \tilde U_2 \cdot\nabla \Omega_2 +
  \delta T_0\partial_t \tilde\Omega_2 \,=\, \cO(\delta)\,.
\end{equation}

The results obtained in this section can be summarized as follows.

\begin{proposition}\label{prop:Rest}
There exist $C > 0$ and $N \in \N$ such that, if the profiles
$\Omega_2,\Omega_3,\Omega_4$ of the approximate solution
\eqref{def:Omapp} are given by \eqref{def:Om2}, \eqref{def:Om3},    
and \eqref{def:Om4}, then the remainder $\cR_\app$ defined by
\eqref{eq:remdef} satisfies the estimate
\begin{equation}\label{eq:Rest}
  \big|\cR_\app(\xi,t)\bigr| \,\le\, C\bigl(\epsilon(t)^5 + \delta^2
  \epsilon(t)^2\bigr) (1+|\xi|)^N e^{-|\xi|^2/4}\,, \qquad
  \forall\,(\xi,t) \in \R^2 \times [0,T]\,.
\end{equation}
\end{proposition}

\begin{proof}
This is a straightforward consequence of the calculations above and of the
choice of the function space $\cZ$. We consider the expression
\eqref{eq:remdef2} of the remainder $\cR_\app$. Using Lemma~\ref{lem:fexpand}
and the fact that $\Omega_\app \in \cZ$, we see that the last term
$\cR_E \cdot \nabla\Omega_\app$ satisfies an estimate of the form
\eqref{eq:Rest}. This is also the case for the terms $\epsilon^k \cR_k$
for $k \ge 5$, because $\cR_k \in \cZ$ and $\epsilon^k \le \epsilon^5$.
Finally, in view of \eqref{eq:R2est}, \eqref{eq:R3est}, and \eqref{eq:R4est}, 
we have
\[
  \epsilon^2 \cR_2 + \epsilon^3 \cR_3 + \epsilon^4 \cR_4 \,=\,
  \cO\bigl(\delta^2 \epsilon^2 + \delta^2 \epsilon^3 + \delta \epsilon^4
  \bigr) \,=\, \cO\bigl(\delta^2 \epsilon^2 + \epsilon^6\bigr)\,,
\]
in the topology of $\cZ$, which again implies an inequality of the
form \eqref{eq:Rest}. 
\end{proof}

\begin{remark}\label{rem:moms}
Since $\Omega_2 \in \cY_2$, $\Omega_3 \in \cY_1' \oplus \cY_3$, and
$\Omega_4 \in \cY_2 \oplus \cY_4$, the approximate solution \eqref{def:Omapp}
satisfies for all positive times
\begin{equation}\label{eq:moms}
  \int_{\R^2} \Omega_\app(\xi,t)\dd \xi \,=\, 1\,, \qquad
  \int_{\R^2} \xi_1 \Omega_\app(\xi,t)\dd \xi \,=\,
  \int_{\R^2} \xi_2 \Omega_\app(\xi,t)\dd \xi \,=\, 0\,.
\end{equation}
\end{remark}

\section{The solution starting from a point vortex}\label{sec3}

In this section we complete the proof of Theorem~\ref{thm1}. We assume
throughout that the kinematic viscosity $\nu > 0$ is small compared to the
circulation parameter $\Gamma > 0$, which is fixed once and for all.  Given any
$z_0 \in \R^2$, we consider the unique solution $\omega(x,t)$ of the vorticity
equation \eqref{eq:NSf} and \eqref{eq:BS} satisfying the conditions \eqref{eq:omunique},
which imply that the initial vorticity is a Dirac mass of strength $\Gamma$ located
at point $z_0$. Since \eqref{eq:NSf} is a viscous conservation law, we know
in particular that $\int_{\R^2}\omega(x,t)\dd x = \Gamma$ for all positive times. 
To desingularize the solution in the regime where $\nu t$ is small, we make the
change of variables \eqref{eq:OmU}, where $z(t)$ denotes the unique solution of
the ODE \eqref{eq:modifz} such that $z(0) = z_0$. The rescaled vorticity
$\Omega(\xi,t)$ and the associated velocity $U(\xi,t)$ then satisfy
the evolution equation \eqref{eq:Omevol2} with initial data \eqref{eq:OmU0}.

To obtain precise estimates on $\Omega$ and $U$, we use the decomposition
\begin{equation}\label{OmUdecomp}
  \Omega(\xi,t) \,=\, \Omega_\app(\xi,t) + \delta w(\xi,t)\,, \qquad
  U(\xi,t) \,=\, U_\app(\xi,t) + \delta v(\xi,t)\,, 
\end{equation}
for all $t \in (0,T)$ and all $\xi \in \R^2$, where $\Omega_\app$ is the
approximate solution \eqref{def:Omapp}, $U_\app$ is the associated velocity
field, and $\delta = \nu/\Gamma$. By construction, the correction terms $w(\xi,t),
v(\xi,t)$ in \eqref{OmUdecomp} vanish at initial time $t = 0$, and our goal is to
show that they remain small in an appropriate topology for all $t \in [0,T]$.
In view of \eqref{eq:Omevol2}, the vorticity $w(\cdot,t)$ satisfies the evolution
equation
\begin{equation}\label{eq:wevol}
  t\partial_t w + \frac{1}{\delta}\bigl(U_\app + E(f,z)\bigr)\cdot
  \nabla w + \frac{1}{\delta}\,v\cdot\nabla\Omega_\app + v \cdot\nabla w
  \,=\, \cL w - \frac{1}{\delta^2}\,\cR_\app\,,
\end{equation}
for $t \in (0,T)$, where $\cR_\app$ is the remainder
term \eqref{eq:remdef}. The velocity $v(\cdot,t)$ is obtained from the vorticity
$w(\cdot,t)$ by the usual Biot-Savart formula \eqref{eq:BS}. Since
$\int \Omega(\xi,t)\dd \xi = 1$, it follows from Remark~\ref{rem:moms} that
\begin{equation}\label{eq:wmoml}
  \int_{\R^2} w(\xi,t)\dd \xi \,=\, 0\,, \qquad \forall\,t \in (0,T)\,.
\end{equation}

The main difficulty in our analysis is the necessity of controlling the solution
of \eqref{eq:wevol} uniformly in the small parameter $\delta > 0$.  Since we
consider zero initial data, the evolution is entirely driven by the source term
$\delta^{-2}\cR_\app$ in the right-hand side, which is of size $\cO(\epsilon^2
+ \delta^{-2}\epsilon^5)$ according to Proposition~\ref{prop:Rest}. Here and
in \eqref{eq:Rest}, the fifth power of $\epsilon$ is due to our choice of constructing
the approximate solution \eqref{def:Omapp} as a {\em fourth order} expansion in
$\epsilon$. This is sufficient to counterbalance the large factor $\delta^{-2}$
because $\delta^{-1}\epsilon^2 = t/T_0$ by \eqref{eq:xidef}, which means that
$\delta^{-2} \epsilon^5 = \cO(\epsilon)$ as long as $t$ remains comparable with $T_0$. 

To prove that the solution of \eqref{eq:wevol} remains of size $\cO(\epsilon)$
over the whole time interval $[0,T]$, we have to show that the linear terms in
\eqref{eq:wevol}, which are multiplied by the large factor $\delta^{-1}$, do
not create instabilities that could result in a rapid amplification of the
solution, on a timescale proportional to $\delta$. This can be done using an
appropriate energy estimate in a weighted $L^2$ space, where the weight function
is carefully chosen so as to minimize the contributions of the dangerous linear
terms.  Note that the nonlinearity $v\cdot\nabla w$ in \eqref{eq:wevol} is not
multipled by $\delta^{-1}$, because we chose to include a factor of $\delta$ in
the definition \eqref{OmUdecomp} of the corrections terms $w,v$.

\subsection{The short time estimate}\label{ssec31}

If the observation time $T$ is small compared to the time scale $T_0$ defined by
\eqref{eq:T0def}, the solution of \eqref{eq:wevol} can be controlled using a simple
energy estimate in the space $\cY$ defined by \eqref{def:cY}. To show this, 
we introduce the functionals
\begin{equation}\label{eq:EFdef}
  \cE[w] \,=\, \int_{\R^2} p(\xi) w(\xi)^2\dd\xi\,, \qquad
  \cF[w] \,=\, \int_{\R^2} p(\xi)\bigl(|\nabla w(\xi)|^2 + |\xi|^2
  w(\xi)^2 + w(\xi)^2\bigr)\dd\xi\,,
\end{equation}
where $p(\xi) = e^{|\xi|^2/4}$, and we observe that $\cE[w] = \|w\|_{\cY}^2$. We have the
following result:

\begin{proposition}\label{prop:smalltime}
There exist positive constants $K_3, \rho, \kappa$ such that, if $0 < \delta \le 1$
and $T/T_0 \le \rho$, the solution of \eqref{eq:wevol} with zero initial data satisfies
\begin{equation}\label{eq:westsmall}
  t \partial_t \cE[w(\cdot,t)] + \kappa\cF[w(\cdot,t)] \,\le\,
  K_3\Bigl(\frac{\epsilon^5}{\delta^2} + \epsilon^2\Bigr)\cE[w(\cdot,t)]^{1/2}\,,
  \qquad \forall\, t\in (0,T)\,.
\end{equation}
In particular
\begin{equation}\label{eq:westsmall2}
  \|w(\cdot,t)\|_\cY \,=\, \cE[w(\cdot,t)]^{1/2} \,\le\, K_3\Bigl(\frac{\epsilon^5}{\delta^2}
  + \epsilon^2\Bigr)\,, \qquad \forall\, t\in (0,T)\,.
\end{equation}
\end{proposition}

\begin{proof}
Using the definition \eqref{def:Lambda} of the operator $\Lambda$, we can write
the evolution equation \eqref{eq:wevol} in the more compact form
\[
  t\partial_t w + \frac{1}{\delta}\,\Lambda w + \frac{1}{\delta}\,\cA[w]
  + \cB[w,w] \,=\, \cL w - \frac{1}{\delta^2}\,\cR_\app\,,
\]
where $\cA$ is the (time-dependent) linear operator defined by
\begin{equation}\label{def:opcA}
  \cA[w] \,=\, (U_\app - U_0)\cdot\nabla w + \BS[w]\cdot\nabla
  (\Omega_\app-\Omega_0) +  E(f,z)\cdot \nabla w\,,
\end{equation}
and $\cB$ is the bilinear map
\begin{equation}\label{def:opcB}
  \cB[w_1,w_2] \,=\, v_1\cdot \nabla w_2\,, \qquad \text{where}\quad
  v_1 \,=\, \BS[w_1]\,.
\end{equation}
Since $\cE[w] = \|w\|_{\cY}^2$, it follows that
\begin{equation}\label{eq:wev1}
  \frac{t}{2}\,\partial_t \cE[w] + \frac{1}{\delta\,}\langle w,\cA[w]\rangle_\cY +
  \langle w,\cB[w,w]\rangle_\cY \,=\, \langle w,\cL w \rangle_\cY
  -\frac{1}{\delta^2}\,\langle w,\cR_\app\rangle_\cY\,,
\end{equation}
where $\langle \cdot,\cdot\rangle_\cY$ denotes the scalar product in the Hilbert
space $\cY$. Here we used the well known fact that $\langle w,\Lambda w\rangle_\cY = 0$
since $\Lambda$ is skew-symmetric in $\cY$, see Proposition~\ref{prop:Lam}.
Our task is to estimate the various terms in \eqref{eq:wev1}.

First of all, we know that the diffusion operator $\cL$ is negative in the subspace of
all $w \in \cY$ with zero integral, see Proposition~\ref{prop:Lam}. In fact, there
exists a constant $\kappa > 0$ such that
\begin{equation}\label{eq:short1}
  \langle w,\cL w \rangle_\cY \,=\,  \int_{\R^2} p w(\cL w)\dd\xi \,\le\, -\kappa \int_{\R^2}
  p\bigl(|\nabla w(\xi)|^2 + |\xi|^2w(\xi)^2 + w(\xi)^2\bigr)\dd\xi \,=\, -\kappa \cF[w]\,,
\end{equation}
see \cite[Lemma~5.1]{Ga18}. On the other hand, using Proposition~\ref{prop:Rest}, we
easily obtain
\begin{equation}\label{eq:short2}
  \frac{1}{\delta^2}\,\bigl|\langle w,\cR_\app\rangle_\cY\bigr| \,\le\, 
  \frac{1}{\delta^2}\,\|w\|_\cY \,\|\cR_\app\|_\cY \,\le\,
  C\Bigl(\frac{\epsilon^5}{\delta^2} + \epsilon^2\Bigr)\|w\|_\cY\,.
\end{equation}
Note that $\epsilon^2/\delta = t/T_0 \le \rho$, which also implies that $\epsilon^2 \le
\rho$ since we assumed that $\delta \le 1$. To bound the trilinear term in \eqref{eq:wev1},
we integrate by parts, using the incompressibility condition $\nabla\cdot v = 0$, to obtain
the convenient expression
\[
   \langle w,\cB[w,w]\rangle_\cY \,=\, \int_{\R^2} pw \bigl(v\cdot\nabla w\bigr)\dd\xi \,=\,
  - \frac12 \int_{\R^2} w^2 (v\cdot\nabla p)\dd \xi \,=\, -\frac14 \int_{\R^2} p w^2(v\cdot\xi)\dd \xi\,,
\]
where we used the fact that $\nabla p = (\xi/2)p$. Since $w \in \cY$ satisfies \eqref{eq:wmoml},
we can apply Lemma~\ref{lem:BS2} with $q = 3$ and $m = 5/3$ to obtain the bound 
$\|v\cdot\xi\|_{L^3} \le C\|w\|_{\cY}$. By H\"older's inequality, we thus find
\[
  \bigl|\langle w,\cB[w,w]\rangle_\cY\bigr| \,\le\, \frac14\,\|p^{1/2}w\|_{L^3}^2\,\|v\cdot\xi\|_{L^3}
  \,\le\, \, C\,\|\nabla(p^{1/2}w)\|_{L^2}^{2/3}\,\|p^{1/2}w\|_{L^2}^{4/3}\,\|w\|_\cY\,,
\]
where in the last step we applied the interpolation estimate $\|g\|_{L^3} \le C\,\|\nabla g\|_{L^2}^{1/3}
\,\| g \|_{L^2}^{2/3}$ to the function $g = p^{1/2}w$. Observing that $\nabla(p^{1/2}w) = p^{1/2}\nabla w +
(\xi/4)p^{1/2}w$ and using the notation \eqref{eq:EFdef}, we conclude that
\begin{equation}\label{eq:short3}
  \bigl|\langle w,\cB[w,w]\rangle_\cY\bigr| \,\le\, C\,\cF[w]^{1/3}\,\cE[w]^{7/6} \,\le\,
  C\,\cF[w]^{1/2}\cE[w]\,,
\end{equation}
where the last inequality follows from the fact that $\cE[w] \le \cF[w]$.

We now consider the quadratic term $\langle w,\cA[w]\rangle_\cY$ in \eqref{eq:wev1}
which is multiplied by the large factor $1/\delta$. Integrating by parts as before
we easily find
\[
  \biggl|\,\int_{\R^2} p w (U_\app-U_0) \cdot\nabla w\dd\xi\,\biggr| \,=\,
  \frac14\,\biggl|\,\int_{\R^2} pw^2 (U_\app-U_0)\cdot\xi\dd \xi\,\biggr| \,\le\,
  C\epsilon^2\|w\|_{\cY}^2\,,
\]
because $\|(U_\app - U_0)\cdot\xi\|_{L^\infty} \le C \epsilon^2$ in view of \eqref{def:Omapp}. Similarly
\[
  \biggl|\,\int_{\R^2} p w \Bigl(v\cdot\nabla(\Omega_\app-\Omega_0)\Bigr)\dd\xi\,\biggr|
  \,\le\, \|p^{1/2}w\|_{L^2}\,\|v\|_{L^2}\,\|p^{1/2}\nabla(\Omega_\app-\Omega_0)\|_{L^\infty}
  \,\le\,   C\epsilon^2\|w\|_{\cY}^2\,,
\]
because $\|v\|_{L^2} \le C\|w\|_\cY$ by Lemma~\ref{lem:BS2} and $\|p^{1/2}\nabla(\Omega_\app-
\Omega_0)\|_{L^\infty} \le C\epsilon^2$. Finally, we recall that
\[
  E(f,z\,;\xi,t) \,=\, \frac{\epsilon T_0}{d}\Bigl(f\bigl(z(t) + \sqrt{\nu t}\,\xi,t\bigr)
  - f(z(t),t) - \epsilon^2 d^2\Delta f(z(t),t)\Bigr)\,,
\]
so that
\[
  |E(f,z;\xi,t)| \,\le\, \epsilon^2 |\xi| + \epsilon^3 T_0 d \,\|\Delta f\|_{L^\infty}
  \,\le\, \epsilon^2 |\xi| + \cK \epsilon^3\,,
\]
where $\cK$ is defined by \eqref{eq:cKdef}. Assuming that $\rho$ is small enough so that
$\cK\epsilon \le 1$, we thus obtain
\[
  \biggl|\,\int_{\R^2} p w E(f,z)\cdot\nabla w\dd\xi\,\biggr|
  \,\le\, \epsilon^2\,\|\nabla w\|_\cY \Bigl(\|\xi w\|_\cY +
  \cK \epsilon \|w\|_\cY\Bigr) \,\le\, \epsilon^2\,\cF[w]\,.
\]
Altogether, recalling that $\epsilon^2/\delta = t/T_0$, we find
\begin{equation}\label{eq:short4}
  \frac{1}{\delta}\,\bigl|\langle w,\cA[w]\rangle_\cY\bigr| \,\le\, \frac{t}{T_0}
  \,\bigl(\cF[w] + C\cE[w]\bigr)\,, \qquad \text{for some } C > 0\,.
\end{equation}

Collecting all estimates \eqref{eq:short1}--\eqref{eq:short4}, we deduce from \eqref{eq:wev1} that
\begin{equation}\label{eq:wev2}
  t\partial_t \cE[w] + \Bigl(2\kappa - \frac{2t}{T_0}\Bigr)\cF[w]
  \,\le\, K_3\Bigl(\frac{\epsilon^5}{\delta^2} + \epsilon^2\Bigr)\cE[w]^{1/2}
  + C_0\Bigl(\frac{t}{T_0} + \cF[w]^{1/2}\Bigr)\cE[w]\,,
\end{equation}
for some positive constants $K_3$ and $C_0$. Since $\cE[w] \le \cF[w]$, it follows that 
\begin{equation}\label{eq:wev3}
  t\partial_t \cE[w] + \Bigl(2\kappa - (2+C_0)\frac{t}{T_0} - C_0\,\cE[w]^{1/2}\Bigr)\cF[w]
  \,\le\, K_3\Bigl(\frac{\epsilon^5}{\delta^2} + \epsilon^2\Bigr)\cE[w]^{1/2}\,.
\end{equation}
Taking $\rho > 0$ small enough, we can ensure that $(2+C_0)t/T_0 \le (2+C_0)\rho \le \kappa/2$
for all $t \in [0,T]$. We now define
\begin{equation}\label{eq:T1def}
    T_1 \,:=\, \inf\bigl\{t \in [0,T)\, ; \,C_0 \cE[w(\cdot,t)]^{1/2} > \kappa/2 \bigr\}\,,
\end{equation}
with the convention that $T_1 = T$ if the set above is empty. Since $w(\cdot,0) = 0$,
it is clear that $T_1 > 0$ by continuity. By construction, on the time interval $(0,T_1)$,
the differential inequality \eqref{eq:wev3} reduces to \eqref{eq:westsmall}. In particular,
we have
\[
  \partial_t \cE[w(\cdot,t)]^{1/2}  \,\le\, \frac{K_3}{2t}\,\biggl(\frac{\epsilon(t)^5}{\delta^2}
  + \epsilon(t)^2\biggr)\,, \qquad \forall\,t \in (0,T_1)\,,
\]
where we recall that $\epsilon(t) = \sqrt{\nu t}/d$. Since $\cE[w(\cdot,0)] = 0$ we deduce that,
for all $t \in (0,T_1)$, 
\begin{equation}\label{eq:wev4}
  \cE[w(\cdot,t)]^{1/2} \,\le\,\frac{K_3}{2} \int_0^t \Bigl(\frac{\epsilon(s)^5}{\delta^2}
  + \epsilon(s)^2\Bigr)\frac{\dd s}{s} \,=\, K_3\biggl(\frac{\epsilon(t)^5}{5\delta^2}
  + \frac{\epsilon(t)^2}{2}\biggr)\,.
\end{equation}
As $\epsilon^2/\delta = t/T_0 \le \rho$, the right-hand side of \eqref{eq:wev4} is no larger 
than $K_3\rho$ if $\rho$ is small enough. We assume finally that $C_0 K_3 \rho \le \kappa/4$.
Then $C_0\cE[w(\cdot,t)]^{1/2} \le \kappa/4$ for all $t \in (0,T_1)$, and in view of
the definition \eqref{eq:T1def} this implies that $T_1 = T$. Thus inequalities
\eqref{eq:wev3} and \eqref{eq:wev4} hold for all $t \in (0,T)$, and 
imply \eqref{eq:westsmall} and \eqref{eq:westsmall2}. This concludes the proof. 
\end{proof}

\subsection{Construction of the energy functional}\label{ssec32}

The approach of the previous section is relatively simple and provides a control
of the solution of \eqref{eq:wevol} in the natural function space $\cY$.
However, as can be seen from the left-hand side of \eqref{eq:wev2}, the argument
completely breaks down when $t/T_0 > \kappa$. To reach longer times, it is necessary
to use a more sophisticated energy functional which allows us to treat separately
three regions of the physical space: a small neighborhood of the vortex center,
an intermediate region, and a far field region where the influence of the
vortex is negligible. This idea is implemented in the previous work \cite{Ga11},
where the interaction of localized vortices is studied. In this section,
we provide a simplified version of the argument, which gives slightly stronger
results. 

Given a small parameter $\epsilon > 0$, we consider the non-radially symmetric
function
\begin{equation}\label{eq:qepsdef}
  q_\epsilon(\xi,t) \,=\, \frac{|\xi|^2}{4} + \frac{\epsilon^2 T_0}{4v_*(|\xi|)}\,
  \Bigl(b(t)(\xi_1^2-\xi_2^2) - 2 a(t)\xi_1\xi_2\Bigr)\,,
  \qquad \forall\,(\xi,t) \in \R^2\times (0,T)\,,
\end{equation}
where $a(t),b(t),T_0$ are as in Lemma~\ref{lem:E2} and
\begin{equation}\label{def:vdef}
  v_*(|\xi|) \,=\, \frac{1}{2\pi |\xi|^2}\Bigl(1-e^{-|\xi|^2/4}\Bigr)\,,
  \qquad \forall\,\xi \in \R^2\,.
\end{equation}
Next, given positive numbers $A,B$ with $A \ll 1 \ll B$ we define the time-dependent regions
\begin{equation}\label{eq:Iepsdef}
\begin{split}
  \I_\epsilon(t) \,&=\, \Bigl\{\xi \in \R^2\,;\, \epsilon|\xi| \,\le\, 2A\,,
  ~\epsilon^2 q_\epsilon(\xi,t) \,\le\, A^2/4\Bigr\}\,, \\
  \II_\epsilon(t) \,&=\, \Bigl\{\xi \in \R^2\setminus \I_\epsilon(t)\,;\,
  \epsilon|\xi| \,<\, B\Bigr\}\,, \\ 
  \III_\epsilon(t) \,&=\, \Bigl\{\xi \in \R^2\,;\, \epsilon |\xi| \,\ge\, B\Bigr\}\,,
\end{split}
\end{equation}
which are pairwise distinct and satisfy $\R^2 = \I_\epsilon(t) \cup \II_\epsilon(t) \cup
\III_\epsilon(t)$ for any $t \in (0,T)$. Finally, we introduce the weight function
$p_\epsilon : \R^2 \times (0,T) \to (0,+\infty)$ defined by the formula
\begin{equation}\label{eq:pepsdef}
  p_\epsilon(\xi,t) \,=\, \begin{cases}
  \exp\bigl(q_\epsilon(\xi,t)\bigr) & \text{ if } \xi \in \I_\epsilon(t)\,, \\[1mm]
  \exp\bigl(A^2/(4\epsilon^2)\bigr) & \text{ if } \xi \in \II_\epsilon(t)\,, \\[1mm]
  \exp\bigl(\gamma|\xi|^2/4\bigr) & \text{ if } \xi \in \III_\epsilon(t)\,,
  \end{cases}
\end{equation}
where $\gamma = A^2/B^2 \ll 1$. 

It is not difficult to verify that, if $A > 0$ is small enough and $0 < \epsilon
\ll A$, the inner region $\I_\epsilon(t)$ defined by \eqref{eq:Iepsdef} is 
a small deformation of the disk of radius $A/\epsilon$ centered at the origin: 

\begin{lemma}\label{lem:Ieps}
If $A > 0$ is sufficiently small and $\epsilon = \cO(A^\alpha)$ for some $\alpha > 1$, 
then for all $t \in (0,T)$ the inner region $\I_\epsilon(t)$ is given by
\begin{equation}\label{eq:Ieps2}
  \I_\epsilon(t) \,=\, \Bigl\{(r\cos(\theta),r\sin(\theta)) \in \R^2\,;\, 0 \le \theta
  \le 2\pi\,,~0 \le r \le \frac{A}{\epsilon}\bigl(1+\rho(\theta,t)\bigr)\Bigr\}\,,
\end{equation}
where $\rho(\cdot,t)$ is smooth, $2\pi$-periodic, and satisfies
\begin{equation}\label{def:rho}
    \rho(\theta,t) \,=\, \pi T_0 \bigl(a(t)\sin(2\theta) - b(t)\cos(2\theta)\bigr)A^2
  + \cO(A^4)\,.
\end{equation}
\end{lemma}

\begin{proof}
Fix $t \in (0,T)$, $\theta \in [0,2\pi]$, and assume that $\xi = (r\cos(\theta),r\sin(\theta))$. 
If we denote $s = r^2/4$, we observe that $q_\epsilon(\xi,t) = g_\epsilon(s)$ where
\begin{equation}\label{def:geps}
  g_\epsilon(s) \,=\, s + 8\pi \epsilon^2 T_0 \bigl(b(t)\cos(2\theta) - a(t)
  \sin(2\theta)\bigr)\phi(s)\,, \qquad \phi(s) \,=\, \frac{s^2}{1-e^{-s}}\,.
\end{equation}
The function $\phi$ is increasing on $\R_+$ with $\phi'(s) \le 1+2s$ for all
$s \ge 0$. Moreover, we know from \eqref{eq:T0def} that $T_0|a(t)| \le 1$ and
$T_0|b(t)| \le 1$. Assuming that $0 \le s \le A^2/\epsilon^2$, we thus find
\begin{equation}\label{eq:gpbounds}
  \bigl|g_\epsilon'(s) - 1\bigr| \,\le\, 16\pi \epsilon^2(1+2s)
  \,\le\, 16\pi \bigl(\epsilon^2 + 2A^2\bigr) \,\le\, \frac12\,,
\end{equation}
provided $A$ and $\epsilon$ are small enough. This implies that the function $g_\epsilon$ 
is strictly increasing on the interval $[0,A^2/\epsilon^2]$ with $g_\epsilon(0) = 0$ and 
$g_\epsilon\bigl(A^2/\epsilon^2) \ge A^2/(2\epsilon^2)$. By the intermediate value theorem, the 
equation $g_\epsilon(s) = A^2/(4\epsilon^2)$ has a unique solution $s = \bar s(\theta,t)$ in
that interval, and the implicit function theorem ensures that $\bar s(\theta,t)$ is a smooth
function of $\theta$ and $t$. Moreover, we easily deduce from \eqref{eq:gpbounds} that
$A^2/(6\epsilon^2) \le \bar{s} \le A^2/(2\epsilon^2)$. If we assume that $\epsilon = \cO(A^\alpha)$
for some $\alpha > 1$, this implies that $\phi(\bar s) = \bar s^2 + \cO(A^\infty)$.
Returning to \eqref{def:geps}, we deduce that
\begin{equation}\label{eq:barsexp}
  \bar{s}(\theta,t) \,=\, \frac{A^2}{4 \epsilon^2}\Bigl(1 - 2\pi T_0 \bigl(b(t)\cos(2\theta)
  - a(t)\sin(2\theta)\bigr)A^2 + \cO(A^4)\Bigr)\,.
\end{equation}
Now, in view of the definition \eqref{eq:Iepsdef}, we have $\xi \in \I_\epsilon(t)$ if and only
if $r^2/4 \le \bar s(\theta,t)$, which gives the formula \eqref{eq:Ieps2} where $\rho(\theta,t)$
is defined by the relation
\begin{equation}\label{eq:rhobars}
    \bigl(1+\rho(\theta,t)\bigr)^2 \,=\, \frac{4\epsilon^2}{A^2}\,\bar s(\theta,t)\,.
\end{equation}
The expansion \eqref{def:rho} follows directly from \eqref{eq:barsexp} and \eqref{eq:rhobars}. 
\end{proof}

According to the definition \eqref{eq:pepsdef}, the weight function $p_\epsilon(\xi,t)$
is equal to $\exp\bigl(q_\epsilon(\xi,t)\bigr)$ in the (elliptical) inner region $\I_\epsilon(t)$,
at the boundary of which $p_\epsilon(\xi,t) = \exp\bigl(A^2/(4\epsilon^2)\bigr)$ by
construction. The weight $p_\epsilon$ is then extended as a constant function in the intermediate
(annular) region $\II_\epsilon(t)$, and as a Gaussian function in the exterior region
$\III_\epsilon(t)$. Since $\gamma = A^2/B^2$, we observe that $\exp(\gamma |\xi|^2/4)
= \exp\bigl(A^2/(4\epsilon^2)\bigr)$ when $|\xi| = B/\epsilon$, which means that
$p_\epsilon$ has no jump at the common boundary of $\II_\epsilon(t)$ and $\III_\epsilon(t)$,
and is therefore a locally Lipschitz function of $\xi \in \R^2$. It is not difficult
to verify that $p_\epsilon$ satisfies uniform bounds of the form
\begin{equation}\label{eq:pepsbd}
  \exp(\gamma |\xi|^2/4) \,\le\, p_\epsilon(\xi,t) \,\le\, \exp(\mu |\xi|^2/4)\,,
  \qquad \forall\,(\xi,t) \in \R^2 \times (0,T)\,,
\end{equation}
where $\mu > 1$ and $\mu = 1 + \cO(A^2)$ as $A \to 0$. 

In analogy with \eqref{eq:EFdef}, we introduce the functionals that will be used
to control the solution of \eqref{eq:wevol}. The first one is the weighted energy
\begin{equation}\label{eq:cEdef}
  \cE_{\epsilon,t}[w] \,=\, \int_{\R^2} p_\epsilon(\xi,t) w(\xi)^2\dd\xi\,,
\end{equation}
which depends explicitly on time because the coefficients $a(t), b(t)$
in the definition \eqref{eq:qepsdef} are time-dependent. Our second functional
is
\begin{equation}\label{eq:cFdef}
  \cF_{\epsilon,t}[w] \,=\, \int_{\R^2} p_\epsilon(\xi,t)\Bigl\{|\nabla w(\xi)|^2 +
  \chi_\epsilon(\xi)w(\xi)^2 + w(\xi)^2\Bigr\}\dd\xi \,\ge\, \cE_{\epsilon,t}[w]\,,
\end{equation}
where
\begin{equation}\label{eq:chiepsdef}
  \chi_\epsilon(\xi) \,=\, \begin{cases}
  |\xi|^2 & \text{ if } |\xi| \le A/\epsilon\,, \\[1mm]
  A^2/\epsilon^2 & \text{ if } A/\epsilon < |\xi| < B/\epsilon\,, \\[1mm]
  \gamma |\xi|^2 & \text{ if } |\xi| \ge B/\epsilon\,. \end{cases}
\end{equation}

We can now state the main result of this section, which provides an accurate
estimate of the solution of \eqref{eq:wevol} on the whole time interval $(0,T)$.
Unlike in Proposition~\ref{prop:smalltime}, there is no smallness assumption
on the ratio $T/T_0$, but the various constants in the statement depend on
$T/T_0$ and on the quantity $\cK$ defined in \eqref{eq:cKdef}. 

\begin{proposition}\label{prop:largetime}
If $A > 0$ is small enough and $B > 0$ is large enough, there exist positive constants
$K_4$, $K_5$, $\kappa$ and $\delta_0$ such that, if $0 < \delta < \delta_0$, 
the solution of \eqref{eq:wevol} with zero initial data satisfies
\begin{equation}\label{eq:westlarge}
  t\partial_t \cE(t) + \kappa\cF(t)  \,\le\, K_4\Bigl(\frac{\epsilon^5}{\delta^2}
  + \epsilon^2\Bigr)\cE(t)^{1/2} + K_5\Bigl(\frac{t}{T_0} + \cF(t)^{1/2}\Bigr)
  \cE(t)\,,\qquad \forall\, t\in (0,T)\,,
\end{equation}
where $\cE(t) = \cE_{\epsilon,t}[w(\cdot,t)]$, $\cF(t) = \cF_{\epsilon,t}[w(\cdot,t)]$,
and $\epsilon = \sqrt{\nu t}/d$. In particular
\begin{equation}\label{eq:westlarge2}
  \cE(t)^{1/2} \,\le\, K_4 \Bigl(\frac{\epsilon^5}{\delta^2} + \epsilon^2\Bigr)
  \,\exp\Bigl(\frac{K_5 t}{2T_0}\Bigr)\,, \qquad \forall\, t\in (0,T)\,.
\end{equation}
\end{proposition}

\subsection{The large time estimate}\label{ssec33}

The goal of this section is to prove Proposition~\ref{prop:largetime}.
If $w(\cdot,t)$ is the solution of \eqref{eq:wevol} with zero initial
data, a direct calculation shows that the energy function \eqref{eq:cEdef}
satisfies
\begin{equation}\label{eq:cEevol}
  \frac{t}{2}\,\partial_t \cE_{\epsilon,t}[w(\cdot,t)] \,=\, \cD_{\epsilon,t}[w(\cdot,t)]
  -\frac{1}{\delta}\,\cA_{\epsilon,t}[w(\cdot,t)] - \cN_{\epsilon,t}[w(\cdot,t)]
  - \frac{1}{\delta^2}\,\cS_{\epsilon,t}[w(\cdot,t)]\,,
\end{equation}
for all $t \in (0,T)$, where the {\em diffusion terms} $\cD_{\epsilon,t}[w]$, the
{\em advection terms} $\cA_{\epsilon,t}[w]$, the {\em nonlinear term}
$\cN_{\epsilon,t}[w]$, and the {\em source term} $\cS_{\epsilon,t}[w]$ are defined by
\begin{equation}\label{def:DANS}
  \begin{split}
  \cD_{\epsilon,t}[w] \,&=\, \frac12 \int_{\R^2} \bigl(t\partial_t p_\epsilon\bigr) w^2\dd\xi
     + \int_{\R^2} p_\epsilon w\bigl(\cL w\bigr)\dd\xi\,, \\[1mm] 
   \cA_{\epsilon,t}[w] \,&=\, \int_{\R^2} p_\epsilon w \bigl(U_\app + E(f,z)\bigr)\cdot
     \nabla w\dd\xi + \int_{\R^2} p_\epsilon w \bigl(v\cdot\nabla\Omega_\app\bigr)\dd\xi\,,\\[1mm]
  \cN_{\epsilon,t}[w] \,&=\, \int_{\R^2} p_\epsilon w \bigl( v\cdot\nabla w\bigr)\dd\xi\,, \\[1mm]
  \cS_{\epsilon,t}[w] \,&=\, \int_{\R^2} p_\epsilon w \cR_\app\dd \xi\,.
  \end{split}
\end{equation}
In \eqref{eq:cEevol} it is understood that $\epsilon = \sqrt{\nu t}/d$ as usual,
so that $t\partial_t \epsilon = \epsilon/2$. Except for that relation, we can 
consider the quantities introduced in \eqref{def:DANS} as defined for any fixed
$t \in (0,T)$ and any fixed $\epsilon$ such that $0 < \epsilon \ll 1$. 
Useful estimates on these quantities are derived in the following paragraphs. 

\subsubsection{Control of the diffusion terms}\label{sssec331}

Using the definition \eqref{def:cL} of the differential operator $\cL$ and integrating
by parts, we see that
\begin{equation}\label{eq:diffrel}
  \cD_{\epsilon,t}[w] \,=\, \frac12 \int_{\R^2} (t\partial_tp_\epsilon)w^2\dd\xi
  - \cQ_\epsilon[w]\,,
\end{equation}
where
\begin{equation}\label{eq:cQdef}
  \cQ_{\epsilon,t}[w] \,=\, \int_{\R^2} \Bigl\{p_\epsilon|\nabla w|^2 + w
  (\nabla w\cdot\nabla p_\epsilon) + \frac14\,(\xi\cdot\nabla p_\epsilon)w^2
  - \frac12\,p_\epsilon w^2\Bigr\}\dd\xi\,.
\end{equation}
The quadratic form $\cQ_{\epsilon,t}$ is everywhere coercive except in the region $\II_\epsilon$
where $\nabla p_\epsilon = 0$. The following lower bound can be established as in
\cite[Prop.~4.15]{GaS24}. For the reader's convenience, we provide the details in
Section~\ref{ssecA4}.

\begin{lemma}\label{lem:Qeps}
Assume that $A > 0$ is small enough and $0 < \epsilon \ll A$. There exists a positive
constant $\kappa$ such that, if $\int_{\R^2}w\dd\xi = 0$, the following estimate holds 
\begin{equation}\label{eq:Qepslow}
  \cQ_{\epsilon,t}[w] \,\ge\,  \kappa\int_{\R^2} p_\epsilon|\nabla w|^2\dd\xi +
  \kappa \int_{\I_\epsilon \cup \III_\epsilon}\bigl(\chi_\epsilon + 1\bigr)p_\epsilon
  w^2\dd\xi - \int_{\II_\epsilon}p_\epsilon w^2\dd\xi\,,
\end{equation}
where $\chi_\epsilon$ is given by \eqref{eq:chiepsdef} and the regions $\I_\epsilon$,
$\II_\epsilon$, $\III_\epsilon$ are defined in \eqref{eq:Iepsdef}. 
\end{lemma}

\begin{corollary}\label{cor:Deps}
Under the assumptions of Lemma~\ref{lem:Qeps}, the diffusion term $\cD_{\epsilon,t}$ defined
in \eqref{def:DANS} satisfies
\begin{equation}\label{eq:Depsbd}
  \cD_{\epsilon,t}[w] \,\le\, -\frac{\kappa}{2}\int_{\R^2} p_\epsilon \Bigl\{|\nabla w|^2 +
  \chi_\epsilon w^2 + w^2\Bigr\}\dd\xi \,=\, -\frac{\kappa}{2}\,\cF_{\epsilon,t}[w]\,.
\end{equation}
\end{corollary}

\begin{proof}
In view of \eqref{eq:diffrel} and \eqref{eq:Qepslow}, it remains to estimate
the time derivative of the weight function $p_\epsilon$. In the region $\I_\epsilon$
we have $p_\epsilon = \exp(q_\epsilon)$  where $q_\epsilon$ is given by \eqref{eq:qepsdef}.
Recalling that $t\partial_t \epsilon^2 = \epsilon^2$, we find
\begin{equation}\label{eq:Deps1}
  t\partial_t q_\epsilon(\xi,t) \,=\, \frac{\epsilon^2 T_0}{4v_*(|\xi|)}\,
  \Bigl(\bigl[b(t) + tb'(t)\bigr](\xi_1^2-\xi_2^2) - 2 \bigl[a(t) + ta'(t)\bigr]\xi_1\xi_2
  \Bigr)\,. 
\end{equation}
Using the definitions of $a(t),b(t)$ in Lemma~\ref{lem:E2} and the ODE \eqref{eq:modifz}
for $z(t)$, it is straightforward to verify that
\begin{equation}\label{eq:Deps2}
  T_0 \bigl(|a(t)| + |b(t)|\bigr) + T_0^2 \bigl(|a'(t)| + |b'(t)|\bigr) \,\le\, C\,,
\end{equation}
where the constant $C$ only depends on the quantity $\cK$ defined in \eqref{eq:cKdef}.
Since $\epsilon^2 |\xi|^2 \le 2A^2$ in region $\I_\epsilon$ and $|v_*(|\xi|)| \ge C/(1+|\xi|^2)$
by \eqref{def:vdef}, we deduce from \eqref{eq:Deps1} and \eqref{eq:Deps2} that 
\[
  |t\partial_t q_\epsilon| \,\le\, C_0 A^2(1+|\xi|^2)\,, \quad \forall\,\xi \in \I_\epsilon\,, 
\]
where the constant $C_0$ depends only on $\cK$ and $T/T_0$.

On the other hand, it is clear that $t\partial_t p_\epsilon = -\bigl(A^2/(4\epsilon^2)
\bigr) p_\epsilon$ in region $\II_\epsilon$, and $t\partial_t p_\epsilon = 0$ in region
$\III_\epsilon$. Summarizing, we have shown that
\begin{equation}\label{eq:Deps3}
  \frac12 \int_{\R^2} \bigl(t\partial_t p_\epsilon\bigr) w^2\dd\xi \,\le\, 
  \frac12\,C_0 A^2 \int_{\I_\epsilon} (1+|\xi|^2) p_\epsilon w^2 \dd \xi -\frac{A^2}{8\epsilon^2}
  \int_{\II_\epsilon} p_\epsilon w^2 \dd \xi\,.
\end{equation}
Now, combining \eqref{eq:diffrel}, \eqref{eq:Qepslow} and \eqref{eq:Deps3}, we obtain
\begin{align*}
  \cD_{\epsilon,t}[w] \,\le\, &-\kappa\int_{\R^2} p_\epsilon|\nabla w|^2\dd\xi -
  \int_{\I_\epsilon} \Bigl[\kappa\bigl(\chi_\epsilon + 1\bigr)-\frac{C_0}{2} A^2
  \big(1+|\xi|^2\big)\Bigr]p_\epsilon w^2\dd\xi \\  &-
  \int_{\II_\epsilon}\Bigl(\frac{A^2}{8\epsilon^2} - 1\Bigr)p_\epsilon w^2\dd\xi 
  -\kappa \int_{\III_\epsilon}\bigl(\chi_\epsilon + 1\bigr)p_\epsilon w^2\dd\xi\,.
\end{align*}
As is easily verified, we have $|\xi|^2 \le 4\chi_\epsilon$ in region $\I_\epsilon$,
so that the quantity inside the square brackets is larger than $\kappa(\chi_\epsilon + 1)/2$
if $A > 0$ is small enough. Similarly $\chi_\epsilon \le A^2/\epsilon^2$ in region
$\II_\epsilon$, which implies that $A^2/(8\epsilon^2) - 1 \ge (\chi_\epsilon+1)/16$
if $\epsilon \ll A$. So, assuming that $\kappa \le 1/16$, we find 
\[
  \cD_{\epsilon,t}[w] \,\le\, - \kappa\int_{\R^2} p_\epsilon|\nabla w|^2\dd\xi
  - \frac{\kappa}{2}\int_{\I_\epsilon} (\chi_\epsilon+1)p_\epsilon w^2\dd\xi
  - \kappa\int_{\II_\epsilon \cup \III_\epsilon} (\chi_\epsilon+1)p_\epsilon w^2\dd \xi\,,
\]
and \eqref{eq:Depsbd} immediately follows. 
\end{proof}

\subsubsection{Control of the advection terms}\label{sssec332}

We first consider the local advection term 
\[
  \cA_{\epsilon,t}^{(1)}[w] \,:=\, \int_{\R^2} p_\epsilon w \bigl(U_\app + E(f,z)\bigr)
  \cdot\nabla w\dd\xi \,=\, -\frac12 \int_{\R^2} w^2\bigl(U_\app + E(f,z)\bigr)
  \cdot\nabla p_\epsilon \dd\xi\,,
\]
where the second expression is obtained after integrating by parts. 

\begin{lemma}\label{lem:A1}
There exists a positive constant $K_6$ (independent of $A,B$) such that
\begin{equation}\label{eq:A1eps}
  \bigl|\cA_{\epsilon,t}^{(1)}[w]\bigr| \,\le\, K_6\,\epsilon^2\Bigl(A + \frac{1}{B}\Bigr)
  \int_{\R^2} p_\epsilon \chi_\epsilon w^2\dd\xi + K_6\,\epsilon^2 \int_{\R^2}p_\epsilon w^2\dd\xi\,.
\end{equation}
\end{lemma}
  
\begin{proof}
In view of \eqref{def:Omapp} and Lemma~\ref{lem:fexpand} we can decompose
\[
  U_\app(\xi,t) \,=\, U_0(\xi) + \epsilon^2 \hat U_2(\xi,t)\,, \qquad
  E(f,z\,;\xi,t) \,=\, \epsilon^2 E_2(\xi,t) + \epsilon^3 \bar E_3(\xi,t)\,,
\]
where $U_0(\xi) = \xi^\perp v_*(|\xi|)$, $E_2$ is given by \eqref{eq:E2},
$|\hat U_2(\xi,t)| \le C/(1+|\xi|)$, and $|\bar E_3(\xi,t)| \le C(1+|\xi|^2)$.
To estimate the quantity $\cA_{\epsilon,t}^{(1)}[w]$, we first assume that 
$\xi \in \I_\epsilon$, so that $p_\epsilon = \exp(q_\epsilon)$. Using the explicit
expression \eqref{eq:qepsdef} and denoting $q_0(\xi) = |\xi|^2/4$, we find by
a direct calculation
\[
  U_0 \cdot \nabla q_\epsilon \,=\, \frac{\epsilon^2 T_0}{4}\,\xi^\perp \cdot
  \nabla \Bigl(b(\xi_1^2-\xi_2^2) - 2 a\xi_1\xi_2\Bigr) \,=\,
  -\frac{\epsilon^2 T_0}{2} \Bigl(a(\xi_1^2-\xi_2^2) + 2 b\xi_1\xi_2\Bigr)
  \,=\, -\epsilon^2E_2 \cdot \nabla q_0\,, 
\]
where the last equality follows from \eqref{eq:E2}. We thus have a {\em partial cancellation}
between the terms $U_\app \cdot \nabla q_\epsilon$ and $E(f,z) \cdot \nabla q_\epsilon$,
which is precisely the reason for which we included a non-radially symmetric correction
in the definition \eqref{eq:qepsdef} of the function $q_\epsilon$. It follows that
\[
  \bigl(U_\app + E(f,z)\bigr) \cdot\nabla q_\epsilon \,=\,
  \epsilon^2 E_2\cdot \nabla\bigl(q_\epsilon - q_0\bigr) +
  \epsilon^2 \hat U_2\cdot \nabla q_\epsilon + \epsilon^3 \bar E_3\cdot
  \nabla q_\epsilon\,.
\]
Since $|\nabla q_\epsilon| \le C|\xi|$ and $|\nabla(q_\epsilon -q_0)| \le C\epsilon^2
|\xi|(1+|\xi|^2)$ for $\xi \in \I_\epsilon$, we easily obtain
\begin{equation}\label{eq:A1first}
  \bigl|\bigl(U_\app + E(f,z)\bigr) \cdot\nabla q_\epsilon\bigr| \,\le\,
  C\epsilon^4 |\xi|^2(1+|\xi|^2) + C\epsilon^2 + C\epsilon^3|\xi|^3
  \,\le\, C\epsilon^2 \bigl(A|\xi|^2 + 1\bigr)\,,
\end{equation}
because $\epsilon^2 |\xi|^2 \le 2A^2$ in region $\I_\epsilon$ and $A \le 1$.

Obviously, we do not need to consider the case where $\xi \in \II_\epsilon$,
because $\nabla p_\epsilon = 0$ in that region. When $\xi \in \III_\epsilon$,
we have $\nabla p_\epsilon = (\gamma \xi/2)p_\epsilon$, so that $U_\app\cdot
\nabla p_\epsilon = \epsilon^2 \hat U_2 \cdot\nabla p_\epsilon$. Observing
that $|E(f,z)| \le \cK \epsilon$ where $\cK$ is defined in \eqref{eq:cKdef},
we conclude that
\begin{equation}\label{eq:A1second}
  \bigl|\bigl(U_\app + E(f,z)\bigr) \cdot\nabla p_\epsilon\bigr| \,\le\,
  C\gamma \Bigl(\epsilon^2 + \epsilon |\xi|\Bigr)p_\epsilon \,\le\,
  C\gamma \epsilon^2  \Bigl(1 + \frac{|\xi|^2}{B}\Bigr) p_\epsilon\,, 
\end{equation}
because $|\xi| \ge B/\epsilon$ in region $\III_\epsilon$. 
Combining \eqref{eq:A1first} and \eqref{eq:A1second}, we have shown that
\[
  \bigl|\bigl(U_\app + E(f,z)\bigr)\cdot\nabla p_\epsilon\bigr| \,\le\,
  C\epsilon^2 \Bigl\{1 + \Bigl(A + \frac{1}{B}\Bigr)\chi_\epsilon\Bigr\}\,p_\epsilon \,,
  \quad \forall\,\xi \in \R^2\,.
\]
Multiplying by $w^2$ and integrating over $\R^2$, we obtain \eqref{eq:A1eps}.
\end{proof}

We next consider the nonlocal term
\[
  \cA_{\epsilon,t}^{(2)}[w] \,:=\, \int_{\R^2} p_\epsilon w \bigl(v\cdot\nabla
  \Omega_\app\bigr)\dd\xi\,, \qquad \text{where}\quad v = \BS[w]\,.
\]

\begin{lemma}\label{lem:A2}
There exists a positive constant $K_7$ such that $|\cA_{\epsilon,t}^{(2)}[w]\bigr| \le
K_7\epsilon^2\cE_{\epsilon,t}[w]$. 
\end{lemma}

\begin{proof}
We deduce from \eqref{def:Omapp} that $\Omega_\app(\xi,t) \,=\, \Omega_0(\xi) + \epsilon^2
\hat\Omega_2(\xi,t)$, where $\Omega_0$ is given by \eqref{eq:OmU0} and the correction
satisfies $|\nabla \hat\Omega_2(\xi,t)| \le C(1+|\xi|)^N\Omega_0$ for some $N \in \N$, 
because $\hat\Omega_2$ belongs to the function space $\cZ$ defined in \eqref{def:cZ}. 
We can thus decompose $\cA_{\epsilon,t}^{(2)}[w] = \cA_{\epsilon,t}^{(3)}[w] + \cA_{\epsilon,t}^{(4)}[w]$ 
where
\[
  \cA_{\epsilon,t}^{(3)}[w] \,=\, \int_{\R^2} p_\epsilon w \bigl(v\cdot\nabla\Omega_0
  \bigr)\dd\xi\,, \qquad
  \cA_{\epsilon,t}^{(4)}[w] \,=\, \epsilon^2\int_{\R^2} p_\epsilon w
  \bigl(v\cdot\nabla\hat\Omega_2\bigr)\dd\xi\,.
\]
The second term is easily estimated using H\"older's inequality:  
\[
  \bigl|\cA_{\epsilon,t}^{(4)}[w]\bigr| \,\le\, \epsilon^2\, \|p_\epsilon^{1/2}w\|_{L^2}
  \|v\|_{L^4} \|p_\epsilon^{1/2}\nabla \hat\Omega_2\|_{L^4} \,\le\,
  C\epsilon^2\,\cE_{\epsilon,t}[w]\,,
\]
because $\|v\|_{L^4} \le C \|w\|_{L^{4/3}} \le C \|p_\epsilon^{1/2}w\|_{L^2}$, see
Lemma~\ref{lem:BS1}, and $\|p_\epsilon^{1/2}\nabla \hat\Omega_2\|_{L^4}$ is uniformly
bounded. Here and in what follows, we use the uniform bounds \eqref{eq:pepsbd}
satisfied by the weight function $p_\epsilon$. Next, denoting $p_0(\xi) = \exp(|\xi|^2/4)$,
we observe that
\[
  \cA_{\epsilon,t}^{(3)}[w] \,=\, \int_{\R^2} \bigl(p_\epsilon - p_0\bigr)
  w \bigl(v\cdot\nabla\Omega_0\bigr)\dd\xi\,, \qquad \text{because}\quad
  \int_{\R^2} p_0 w (v\cdot \nabla\Omega_0)\dd\xi \,=\, 0\,,
\]
see \cite[Lemma~4.8]{GW05}. If $|\xi| \le \epsilon^{-1/2}$, then  $|p_\epsilon-p_0|
\le C\epsilon^2 p_0 |\xi|^2(1+|\xi|^2)$ so that
\[
  \int_{\{|\xi|\le\epsilon^{-1/2}\}} \bigl|p_\epsilon - p_0\bigr|
 \,|w|\,|v|\,|\nabla\Omega_0|\dd\xi \,\le\, C\epsilon^2\,
 \int_{\I_\epsilon} |\xi|^3(1+|\xi|^2)\,|w|\,|v| \dd\xi \,\le\,
 C\epsilon^2\,\cE_{\epsilon,t}[w]\,,
\]
by the same arguments as before. When $|\xi| > \epsilon^{-1/2}$ we use the crude bound
$|p_\epsilon-p_0| \le p_\epsilon + p_0$ together with the estimates \eqref{eq:pepsbd}
to obtain
\begin{align*}
  \int_{\{|\xi|>\epsilon^{-1/2}\}} p_\epsilon \,|w|\,|v|\,|\nabla\Omega_0|\dd\xi \,&\le\,
  \|p_\epsilon^{1/2}w\|_{L^2}\|v\|_{L^4} \|p_\epsilon^{1/2}\nabla\Omega_0\|_{L^4
  (|\xi|>\epsilon^{-1/2})} \,\le\, C\,e^{-c/\epsilon}\cE_{\epsilon,t}[w]\,, \\
  \int_{\{|\xi|>\epsilon^{-1/2}\}} p_0 \,|w|\,|v| \,|\nabla\Omega_0|\dd\xi \,&\le\, 
  \|p_\epsilon^{1/2}w\|_{L^2}\|v\|_{L^4} \||\xi| p_\epsilon^{-1/2}\|_{L^4
    (|\xi|>\epsilon^{-1/2})} \,\le\, C \,e^{-c/\epsilon} \cE_{\epsilon,t}[w]\,,
\end{align*}
for some $c > 0$. We deduce that $|\cA_{\epsilon,t}^{(3)}[w]\bigr| \le C\epsilon^2
\cE_{\epsilon,t}[w]$, which concludes the proof. 
\end{proof}

\subsubsection{End of the proof of Proposition~\ref{prop:largetime} and
Theorem~\ref{thm1}}\label{sssec333}

In the previous paragraphs we obtained accurate bounds on the diffusion terms
$\cD_{\epsilon,t}[w]$ and the advection term $\cA_{\epsilon,t}[w]$ in \eqref{def:DANS}.
As for the nonlinear term $\cN_{\epsilon,t}[w]$ and the source term $\cS_{\epsilon,t}[w]$,
they can be estimated exactly as in the proof of Proposition~\ref{prop:smalltime}.
Using in particular the uniform bounds \eqref{eq:pepsbd} and the estimate
$|\nabla p_\epsilon| \le C\chi_\epsilon^{1/2}p_\epsilon$, which is established
as in Lemma~\ref{lem:A1}, we find
\begin{equation}\label{eq:NSbd}
  \bigl|\cN_{\epsilon,t}[w]\bigr| \,\le\, K_8\,\cF_{\epsilon,t}[w]^{1/2} \cE_{\epsilon,t}[w]\,,
  \qquad \bigl|\cS_{\epsilon,t}[w]\bigr| \,\le\, K_9 \bigl(\epsilon^5 + \delta^2
  \epsilon^2\bigr)\cE_{\epsilon,t}[w]^{1/2}\,,
\end{equation}
for some positive constants $K_8, K_9$ depending on $A$ and $B$.

It is now a straightforward task to complete the proof of Proposition~\ref{prop:largetime}.
If $w(\xi,t)$ is the solution of \eqref{eq:wevol} with zero initial data, we consider
the weighted energy $\cE(t) := \cE_{\epsilon,t}[w(\cdot,t)]$, which evolves according
to \eqref{eq:cEevol}. It is understood here that $\epsilon = \sqrt{\nu t}/d$, so that
$\epsilon^2/\delta = t/T_0$. Using the bounds \eqref{eq:NSbd} and the estimates collected
in Corollary~\ref{cor:Deps}, Lemma~\ref{lem:A1} and Lemma~\ref{lem:A2}, we obtain the
differential inequality
\begin{equation}\label{eq:cEdiff}
\begin{split}
  t\partial_t \cE(t) \,\le\, &-\kappa\cF(t) + 2K_6\Bigl(A + \frac{1}{B}\Bigr)\frac{t}{T_0}
  \,\cF(t) + 2\bigl(K_6+K_7\bigr)\frac{t}{T_0}\,\cE(t) \\ &+ 2K_8\,\cF(t)^{1/2}\cE(t)
  + 2K_9  \Bigl(\frac{\epsilon^5}{\delta^2} + \epsilon^2\Bigr)\cE(t)^{1/2}\,,
\end{split}
\end{equation}
where $\cF(t) := \cF_{\epsilon,t}[w(\cdot,t)]$. We next choose $A > 0$ small enough
and $B > 1$ large enough so that
\begin{equation}\label{eq:long1}
  2K_6\Bigl(A + \frac{1}{B}\Bigr)\frac{T}{T_0} \,\le\, \frac{\kappa}{2}\,.
\end{equation}
Under this assumption, inequality \eqref{eq:cEdiff} implies \eqref{eq:westlarge} with
$K_4 = 2K_9$, $K_5 = 2(K_6+K_7+K_8)$, and $\kappa$ replaced by $\kappa/2$. 

To deduce \eqref{eq:westlarge2} from \eqref{eq:westlarge}, we proceed as in the
proof of Proposition~\ref{prop:smalltime}. Since $\cE(0) = 0$, we know that
$K_5\,\cE(t)^{1/2} \le \kappa$ at least for short times. As long as that
inequality holds, we deduce from \eqref{eq:westlarge} that 
\[
  t\partial_t \cE(t) \,\le\, K_4\Bigl(\frac{\epsilon^5}{\delta^2}
  + \epsilon^2\Bigr)\cE(t)^{1/2} + K_5\frac{t}{T_0}\,\cE(t)\,,
\]
and that differential inequality can be integrated using Gr\"onwall's lemma  
to give the estimate in \eqref{eq:westlarge2}. As $\epsilon^2 = \delta t/T_0
\le \delta T/T_0$, the conclusion remains true for all $t \in (0,T)$ provided
$\delta > 0$ is chosen small enough so that
\begin{equation}\label{eq:long2}
  K_4K_5\Bigl(\frac{\epsilon^5}{\delta^2} + \epsilon^2\Bigr)\,\exp\Bigl(
  \frac{K_5 t}{2T_0}\Bigr) \,\le\, K_4K_5\biggl\{\delta^{1/2}\Bigl(
  \frac{T}{T_0}\Bigr)^{5/2} + \frac{\delta T}{T_0}\biggr\}\,\exp\Bigl(
  \frac{K_5 T}{2T_0}\Bigr) \,<\, \kappa\,.
\end{equation}
This concludes the proof of Proposition~\ref{prop:largetime}.\qed

\begin{remark}\label{rem:logtime}
In connection with Remark~\ref{rem:longT} we observe that, if the external flow $f$
is globally defined and satisfies uniform bounds, then inequality \eqref{eq:long2}
is still satisfied for small $\delta > 0$ if we take $T = c T_0 \log(1/\delta)$ 
with $0 < c < 1/K_5$, and condition \eqref{eq:long1} is also fulfilled if we choose
$A^{-1} = B = C \log(1/\delta)$ with $C = 8cK_6/\kappa$. In that situation, 
the lower bound in \eqref{eq:pepsbd} is not uniform anymore, since $\gamma  = 
A^2/B^2\to 0$ as $\delta \to 0$, but one can verify that quantities such as 
$\|p_\epsilon^{-1}\|_{L^1}$ are still uniformly bounded, and this is sufficient
to show that the constants $K_4$ and $K_5$ in \eqref{eq:westlarge} are independent 
of $\delta$. A detailed verification of these claims is left to the reader. 
\end{remark}

\begin{proof}[End of the proof of Theorem~\ref{thm1}]
We consider the solution $\omega(x,t)$ of \eqref{eq:NSf} and \eqref{eq:BS} satisfying
\eqref{eq:omunique}, and we make the change of variables \eqref{eq:OmU} where
$z(t)$ is the solution of the ODE \eqref{eq:modifz} with initial condition
$z(0) = z_0$. Comparing the definition \eqref{def:omapp} with the
expression of $\bar\Omega_2$ given in Remark~\ref{rem:w2}, we see that
\[
  \omega_\app\bigl(\Gamma,\sqrt{\nu t},z(t),f(t)\,; x\bigr) \,=\,
  \frac{\Gamma}{\nu t}\,\biggl\{\Omega_0\Bigl(\frac{x - z(t)}{\sqrt{\nu t}}\Bigr)
  + \epsilon^2\bar\Omega_2\Bigl(\frac{x - z(t)}{\sqrt{\nu t}},t\Bigr)\biggr\}\,,
\]
where $\Omega_0$ is defined in \eqref{eq:OmU0} and $\epsilon = \sqrt{\nu t}/d$.
It follows that
\begin{align*}
   &\frac{1}{\Gamma}\int_{\R^2}\Bigl|\,\omega(x,t) - \omega_\app\bigl(\Gamma,\sqrt{\nu t},z(t),f(t)
  \,; x\bigr)\Bigr|\dd x \,=\, \int_{\R^2} \bigl|\Omega(\xi,t) - \Omega_0(\xi) -
  \epsilon^2 \bar\Omega_2(\xi,t)\bigr|\dd\xi \\
  \,&\le\, \int_{\R^2} \bigl|\Omega(\xi,t) - \Omega_\app(\xi,t)\bigr|\dd\xi +
  \int_{\R^2} \bigl|\Omega_\app(\xi,t) - \Omega_0(\xi) - \epsilon^2 \bar\Omega_2(\xi,t)
  \bigr|\dd\xi \,=:\, \cI_1(t) + \cI_2(t)\,,
\end{align*}
where $\Omega_\app$ is defined in \eqref{def:Omapp}. Using \eqref{OmUdecomp},
\eqref{eq:pepsbd} and Proposition~\ref{prop:largetime}, we find
\begin{equation}\label{eq:cI1}
  \cI_1(t) \,=\, \delta\,\|w(t)\|_{L^1} \,\le\, C\delta\,\|p_\epsilon^{1/2}w(t)\|_{L^2}
  \,\le\, C\Bigl(\frac{\epsilon^5}{\delta} + \delta \epsilon^2\Bigr) \,\le\,
  C\epsilon^2(\epsilon + \delta)\,, \quad \forall\,t \in (0,T)\,.
\end{equation}
On the other hand, we have $\Omega_\app - \Omega_0 - \epsilon^2\bar\Omega_2 =
\delta \epsilon^2 \tilde\Omega_2 + \epsilon^3 \Omega_3 +  \epsilon^4 \Omega_4$, 
as can be seen from the expansion \eqref{def:Omapp} and the definition of $\Omega_2$
in Section~\ref{sssec231}. We deduce that $\cI_2(t) \le C\epsilon^2(\epsilon + \delta)$,
and together with \eqref{eq:cI1} this concludes the proof of estimate \eqref{eq:thm1}. 
\end{proof}

\subsubsection{Alternative definitions of the vortex position}\label{sssec334}

The results stated in the introduction are sensitive to the precise
definition of the vortex position, because they concern sharply
concentrated solutions. As is mentioned in Remark~\ref{rem:zbar},
it is natural to consider the center of vorticity $\bar z(t)$ defined
by \eqref{eq:zbardef}, which however does not satisfy an ODE
such as \eqref{eq:originz} or \eqref{eq:modifz}. Under the assumptions
of Theorem~\ref{thm1}, it turns out that $\bar z(t)$ stays very close
to the solution $z(t)$ of \eqref{eq:modifz} with $z(0) = z_0$. To see
this, we observe that
\[
  \bar z(t) \,=\, \frac{1}{\Gamma}\int_{\R^2} x\,\omega(x,t)\dd x \,=\, 
  \int_{\R^2} \frac{x}{\nu t}\,\Omega\biggl(\frac{x-z(t)}{\sqrt{\nu t}}\biggr)\dd x
  \,=\, \int_{\R^2} \bigl(z(t)+\epsilon d\xi\bigr)\,\Omega(\xi,t)\dd \xi\,,
\]
where $\Omega$ is defined in \eqref{eq:OmU}. Using the decomposition
\eqref{OmUdecomp} and the moment identities \eqref{eq:moms} and \eqref{eq:wmoml},
we deduce
\[
  \bar z(t) \,=\, \int_{\R^2} \bigl(z(t)+\epsilon d\xi\bigr)\Bigl(\Omega_\app(\xi,t)
  + \delta w(\xi,t)\Bigr)\dd \xi \,=\, z(t) + \delta \epsilon d \int_{\R^2}\xi
  w(\xi,t)\dd \xi\,.
\]
The integral in the right-hand side can be estimated using Proposition~\ref{prop:largetime},
which gives
\[
  \bigl|\bar z(t) - z(t)\bigr| \,\le\, \delta \epsilon d\,\bigl\||\xi|w(t)\bigr\|_{L^1}
  \,\le\, C\delta \epsilon d \,\bigl\|p_\epsilon^{1/2}w(t)\bigr\|_{L^2} \,\le\,
  Cd \,\epsilon^3\bigl(\epsilon + \delta\bigr)\,, \qquad \forall\,t \in (0,T)\,.
\]
As is easily verified, this implies that estimate \eqref{eq:thm1} in
Theorem~\ref{thm1} remains valid if the vortex position $z(t)$ is replaced by the
center of vorticity $\bar z(t)$. 

On the other hand, as computed below, the approximate vortex position $\hat z(t)$
given by the simple ODE \eqref{eq:originz} is only $\cO(\epsilon^2)$ close to
the solution $z(t)$ of \eqref{eq:modifz}, unless the additional term
$\Delta f(z(t),t)$ in \eqref{eq:modifz} vanishes identically. As was already
mentioned, this is the case if the external velocity field $f$ is
irrotational, see the discussion in Remark~\ref{rem:f}. In general, taking the
difference of \eqref{eq:originz} and \eqref{eq:modifz}, we obtain the inequality
\[
  \bigl|z'(t) - \hat z'(t)\bigr| \,\le\, \bigl|f(z(t),t) - f(\hat z(t),t)
  \bigr| + \nu t\,\bigl|\Delta f(z(t),t)\bigr| \,\le\, \frac{1}{T_0}\,
  \bigl|z(t) - \hat z(t)\bigr| + \cK\,\frac{\nu t}{T_0d}\,,
\]
which can be integrated using Gr\"onwall's lemma to give
\begin{equation}\label{eq:zdiffer}
  |z(t) - \hat z(t)| \,\le\, \cK \int_0^t e^{(t-s)/T_0} \,\frac{\nu s}{T_0d}\dd s
  \,\le\, \cK\,e^{t/T_0}\,\frac{\nu t}{d} \,\le\, C d\,\epsilon^2\,,
  \qquad \forall\,t \in [0,T]\,.
\end{equation}

With estimate \eqref{eq:zdiffer} at hand, it is not difficult to verify that
\eqref{eq:thm1} implies \eqref{eq:prop1}, so that Proposition~\ref{prop1}
follows from Theorem~\ref{thm1}. Indeed, using the bound \eqref{eq:w2est} with
$\ell = \sqrt{\nu t}$ and the fact that $\epsilon(t) \le 1$ and $\delta \le 1$ under
the assumptions of Theorem~\ref{thm1}, we deduce from \eqref{eq:thm1} that
\begin{equation}\label{eq:prop1bis}
  \frac{1}{\Gamma}\int_{\R^2}\Bigl|\,\omega(x,t) - \frac{\Gamma}{\nu t}\,\Omega_0\Bigl(
  \frac{x-z(t)}{\sqrt{\nu t}}\Bigr)\Bigr|\dd x \,\le\, C\,\epsilon(t)^2\,,
  \qquad \forall\,t \in (0,T)\,,
\end{equation}
for some constant $C > 0$. This estimate is similar to \eqref{eq:prop1}, except that the
Lamb-Oseen vortex is centered at the modified position $z(t)$, which differs
from $\hat z(t)$. To obtain \eqref{eq:prop1} from \eqref{eq:prop1bis}, we use the
elementary bound
\[
  \int_{\R^2}\Bigl|\Omega_0\Bigl(\frac{x-z}{\ell}\Bigr) - \Omega_0\Bigl(\frac{x-\hat z}{\ell}\Bigr)
  \Bigr| \dd x \,\le\, \ell\,\|\nabla \Omega_0\|_{L^1}\,|z - \hat z|\,,
\]
together with the estimate \eqref{eq:zdiffer} for the difference $z(t) - \hat z(t)$. 

\subsubsection{Alternative choice of the weight function}\label{sssec335}

There is a certain amount of freedom in the choice of the weight function that is
used to control the solution of \eqref{eq:wevol} for large times. In particular,
the diameter of the inner region $\I_\epsilon(t)$ should tend to infinity as
$\epsilon \to 0$, but does not need to be proportional to $\epsilon^{-1}$, as in
\eqref{eq:Iepsdef}. Arguably the definition \eqref{eq:pepsdef} gives the
largest possible weight $p_\epsilon(\xi,t)$ for which our energy estimates apply.  
Compared with the previous works \cite{Ga11,GaS24,DG24}, where a similar approach
is implemented, our argument here provides a stronger control of the solution of
\eqref{eq:NSf}. One can also observe that the diameters of the regions $\I_\epsilon$
and $\II_\epsilon$ defined by \eqref{eq:Iepsdef} are $\cO(\epsilon^{-1})$ in the
self-similar variable $\xi$, hence $\cO(1)$ in the original variable $x$, which is
quite remarkable. A minor drawback of the definition \eqref{eq:pepsdef} is that
the weight $p_\epsilon$ is not a small perturbation of $p_0$, even in the inner
region $\I_\epsilon(t)$. In particular, as indicated in \eqref{eq:pepsbd},
it satisfies an upper bound of the form $p_\epsilon(\xi,t) \le C \exp(\mu|\xi|^2/4)$
for $\mu = 1 +\cO(A^2)$, but not for $\mu = 1$. 

Another interesting possibility is to use the alternative weight
\begin{equation}\label{eq:hpepsdef}
  \hat p_\epsilon(\xi,t) \,=\, \begin{cases}
  \exp\bigl(q_\epsilon(\xi,t)\bigr) & \text{ if } \xi \in \hat\I_\epsilon(t)\,, \\[1mm]
  \exp\bigl(A^2/(4\epsilon)\bigr) & \text{ if } \xi \in \hat\II_\epsilon(t)\,, \\[1mm]
  \exp\bigl(\gamma|\xi|/4\bigr) & \text{ if } \xi \in \III_\epsilon(t)\,,
  \end{cases}
\end{equation}
where $\gamma = A^2/B$ and the new regions $\hat\I_\epsilon$, $\hat\II_\epsilon$
are defined by
\begin{align*}
  \hat\I_\epsilon(t) \,&=\, \Bigl\{\xi \in \R^2\,;\, \epsilon^{1/2}|\xi| \,\le\, 2A\,,
  ~\epsilon q_\epsilon(\xi,t) \,\le\, A^2/4\Bigr\}\,, \\
  \hat\II_\epsilon(t) \,&=\, \Bigl\{\xi \in \R^2\setminus \hat\I_\epsilon(t)\,;\,
  \epsilon|\xi| \,<\, B\Bigr\}\,,
\end{align*}
Up to inessential details, this is the choice made in \cite{Ga11} in a related context.
The main difference with the previous definition is that the inner region is
now smaller, with a diameter proportional to $\epsilon^{-1/2}$ instead of $\epsilon^{-1}$.
Since the outer region is unchanged, the intermediate region is proportionally larger.
The new weight satisfies uniform estimates of the form
\begin{equation}\label{eq:hpest}
  C^{-1}\,e^{\gamma |\xi|/4} \,\le\, \hat p_\epsilon(\xi,t) \,\le\, C\,e^{|\xi|^2/4}\,,
  \qquad \forall\,(\xi,t) \in \R^2 \times [0,T]\,,
\end{equation}
for some constant $C > 1$. As is easily verified, the analogue of
Proposition~\ref{prop:largetime} holds for the functionals \eqref{eq:cEdef},
\eqref{eq:cFdef} defined with the new weight $\hat p_\epsilon$, provided the
function $\chi_\epsilon$ introduced in \eqref{eq:chiepsdef} is replaced by 
\begin{equation}\label{eq:hchidef}
  \hat\chi_\epsilon(\xi) \,=\, \begin{cases}
  |\xi|^2 & \text{ if } |\xi| \le A/\epsilon^{1/2}\,, \\[1mm]
  A^2/\epsilon & \text{ if } A/\epsilon^{1/2} < |\xi| < B/\epsilon\,, \\[1mm]
  \gamma |\xi| & \text{ if } |\xi| \ge B/\epsilon\,. \end{cases}
\end{equation}
The estimates are even simpler because the intermediate region, where the
dangerous advection terms give no contribution, is now larger. Although
somewhat weaker, the result obtained in this way is still sufficient to imply
Theorem~\ref{thm1}. 

\section{The solution starting from a Gaussian vortex}\label{sec4}

In this section we prove our second main result, Theorem~\ref{thm2}, by combining
the energy estimates developed in Section~\ref{sec3} with the enhanced dissipation
estimates established by Li, Wei, and Zhang in \cite{LWZ20}. Let $\omega(x,t)$
denote the solution of \eqref{eq:NSf} and \eqref{eq:BS} with initial data
\eqref{eq:inGauss} at time $t_0 \in (0,T)$. We make the change of variables
\eqref{eq:OmU} for $t \ge t_0$, where $z(t)$ is the solution of the ODE \eqref{eq:modifz}
with initial condition $z(t_0) = z_0$. The new functions $\Omega(\xi,t)$ and
$U(\xi,t)$ still satisfy the evolution equation \eqref{eq:Omevol}. To eliminate
the slightly unusual time derivative $t\partial_t$, it is convenient here to introduce
the new dimensionless variable
\begin{equation}\label{eq:taudef}
  \tau \,=\, \log(t/t_0)\,,
\end{equation}
which is nonnegative and vanishes precisely at initial time $t = t_0$. 

In the spirit of \eqref{OmUdecomp}, we decompose 
\begin{equation}\label{OmUdecomp2}
  \Omega(\xi,t) \,=\, \Omega_\app(\xi,t) + \delta w\bigl(\xi,\log(t/t_0)\bigr)\,,
  \qquad U(\xi,t) \,=\, U_\app(\xi,t) + \delta v\bigl(\xi,\log(t/t_0)\bigr)\,,
\end{equation}
the only difference being that the perturbations $w,v$ are now considered as functions
of the dimensionless time $\tau \ge 0$, instead of the original time $t \ge t_0$. 
As in \eqref{eq:wevol}, they satisfy the evolution equation
\begin{equation}\label{eq:wevol3}
  \partial_\tau w + \frac{1}{\delta}\bigl(U_\app + E(f,z)\bigr)\cdot
  \nabla w + \frac{1}{\delta}\,v\cdot\nabla\Omega_\app + v \cdot\nabla w
  \,=\, \cL w - \frac{1}{\delta^2}\,\cR_\app\,,
\end{equation}
where $\cR_\app$ is defined in \eqref{eq:remdef}. Here and in what follows,
it is understood that all quantities that depend explicitly on time, such
as the vortex position $z(t)$ or the aspect ratio $\epsilon(t) = \sqrt{\nu t}/d$,
should be considered as functions of the dimensionless variable \eqref{eq:taudef}
via the relation $t = t_0 e^{\tau}$.

At initial time $\tau = 0$, we have by construction
\begin{equation}\label{def:winit0}
  w(\xi,0) \,=\, \phi_0(\xi) \,:=\, \frac{1}{\delta}\bigl(\Omega_0(\xi)
  - \Omega_\app(\xi,t_0)\bigr)\,, \qquad \forall\,\xi \in \R^2\,.
\end{equation}
Using the notations of Section~\ref{ssec22}, we observe that $\phi_0 \in \cZ
\cap \Ker(\Lambda)^\perp$, and that $\phi_0$ is of size $\cO(\epsilon_0^2/\delta)$
where $\epsilon_0 = \sqrt{\nu t_0}/d$. Since $\epsilon_0^2/\delta = t_0/T_0$,
this means that the initial data \eqref{def:winit0} are not small in the limit
where $\delta \to 0$, which is in contrast with the situation considered in
the proof of Theorem~\ref{thm1}. However, as we shall see, it is possible to decompose
the perturbation $w(\xi,\tau)$ into a linear component $w_0(\xi,\tau)$ that is initially
of size $\cO(1)$ but decays rapidly as time evolves, and a correction term $\tilde w(\xi,\tau)$
which vanishes initially and remains small for all times. 

\subsection{Enhanced dissipation estimates}\label{ssec41}

Given $\phi_0 \in \cY$ and $\delta > 0$, we define $w_0(\xi,\tau)$ as the unique solution of
the linear equation
\begin{equation}\label{def:w0}
  \partial_\tau w_0(\xi,\tau) \,=\, \bigl(\cL - \delta^{-1}\Lambda\bigr)w_0(\xi,\tau)\,,
  \qquad w_0(\xi,0) = \phi_0(\xi)\,.
\end{equation}
The spectral and pseudospectral properties of the linear operator $\cL- \delta^{-1}\Lambda$
have been thoroughly studied in previous works, including \cite{GW05,Ma11,De13,Ga18,LWZ20}. 
Combining the optimal resolvent estimates obtained by Li, Wei, Zhang \cite{LWZ20}
with the quantitative Gearhart-Pr\"uss theorem due to Wei \cite{Wei21}, we
immediately obtain the following result.

\begin{proposition}\label{prop:ED1}
There exist positive constants $C_0$ and $c_0$ such that, for all $\delta \in (0,1)$
and all $\phi_0 \in \cY \cap \Ker(\Lambda)^\perp$, the solution of \eqref{def:w0}
satisfies
\begin{equation}\label{eq:w0bd1}
  \|w_0(\tau)\|_\cY \,\le\, C_0\,\exp\Bigl(-\frac{c_0\tau}{\delta^{1/3}}\Bigr)
  \|\phi_0\|_\cY\,, \qquad \forall\,\tau \ge 0\,.
\end{equation}
In particular, denoting $C_1 = C_0^2/(2c_0)$, we have
\begin{equation}\label{eq:w0bd2}
  \int_0^\infty \|w_0(\tau)\|_\cY^2\dd\tau \,\le\, C_1 \delta^{1/3}\|\phi_0\|_\cY^2\,.
\end{equation}
\end{proposition}

\begin{remark}\label{rem:w0}
Unfortunately, it is impossible to obtain an estimate of the form \eqref{eq:w0bd2}
for the gradient norm $\|\nabla w_0\|_\cY$. Indeed, since the operator $\Lambda$
is skew-symmetric in $\cY$, an easy calculation shows that
\[
  \frac12 \frac{\dd}{\dd \tau}\,\|w_0\|_\cY^2 \,=\, \langle w_0\,,\bigl(\cL-\delta^{-1}
  \Lambda\bigr)w_0\rangle_\cY \,=\, \langle w_0\,,\cL w_0\rangle_\cY \,=\,
  -\|\nabla w_0\|_{\cY}^2 + \|w_0\|_{\cY}^2\,,
\]
where the last equality is obtained after integrating by parts. It follows that
\[
  \|w_0(\tau)\|_\cY^2 + 2\int_0^\tau \|\nabla w_0(s)\|_{\cY}^2\dd s \,=\,
  \|\phi_0\|_\cY^2 + 2\int_0^\tau \|w_0(s)\|_{\cY}^2\dd s\,,\qquad
  \forall\,\tau \ge 0\,.
\]
Taking the limit $\tau \to +\infty$ and using \eqref{eq:w0bd1}, we see that
$\int_0^\infty \|\nabla w_0(s)\|_{\cY}^2\dd s \ge \frac12 \|\phi_0\|_\cY^2$,
a lower bound that holds uniformly for all $\delta \in (0,1)$. 
\end{remark}

Our main result regarding the linear equation \eqref{def:w0} is the following
integrated estimate for the weighted norm $\||\xi|w_0\|_\cY$. 
\begin{proposition}\label{prop:ED2}
For any $\gamma > 1/8$, there exists a constant $C_2 > 0$ such that,
for all $\delta \in (0,1)$ and all $\phi_0 \in \cY \cap \Ker(\Lambda)^\perp$
satisfying
\[
  \|\phi_0\|_\gamma \,:=\, \sup_{\xi \in \R^2}e^{\gamma |\xi|^2} |\phi_0(\xi)|
  \,<\, \infty\,,
\]
the solution of \eqref{def:w0} satisfies
\begin{equation}\label{eq:w0bd3}
  \int_0^\infty \||\xi| w_0(\tau)\|_\cY^2\dd\tau \,\le\, C_2 \delta^{1/3}
  \log\Bigl(\frac{2}{\delta}\Bigr)\,\|\phi_0\|_\gamma^2\,.
\end{equation}
\end{proposition}

\begin{proof}
Since $\gamma > 1/8$, we first observe that $\|\phi_0\|_\cY \le C \|\phi_0\|_\gamma$ for
some constant $C > 0$ depending on $\gamma$. Next, we give ourselves a cut-off parameter
$\rho > 0$ satisfying $\rho^2 = N \log(2/\delta)$, for some (large) integer $N$ that will
be chosen later, depending only on $\gamma$. In view of \eqref{eq:w0bd2}, we have
\begin{equation}\label{eq:w0bdfirst}
  \int_0^\infty \||\xi| \mathbf{1}_{\{|\xi| \le \rho\}}w_0(\tau)\|_\cY^2\dd\tau
  \,\le\, \rho^2 \int_0^\infty \|w_0(\tau)\|_\cY^2\dd\tau \,\le\, C_1N \delta^{1/3}
  \log\Bigl(\frac{2}{\delta}\Bigr)\|\phi_0\|_\cY^2\,,
\end{equation}
which gives the first half of the desired bound. To complete the proof of \eqref{eq:w0bd3},
we need an integral estimate of the quantity $\||\xi| \mathbf{1}_{\{|\xi| > \rho\}}
w_0(\tau)\|_\cY^2$. This can be obtained using appropriate energy estimates, as we now explain. 

First of all, we introduce the function $h(\xi,\tau) = p(\xi)^{1/2}w_0(\xi,t)$, where
$p(\xi) = e^{|\xi|^2/4}$. It is easy to verify that $\partial_\tau h = \bigl(L - \delta^{-1}
\hat\Lambda\bigr)h$, where
\[
  L \,=\, \Delta - \frac{|\xi|^2}{16} + \frac12\,, \qquad
  \hat\Lambda h \,=\, U_0 \cdot\nabla h \,+\, p^{1/2}\,\BS\bigl[p^{-1/2}h\bigr]
  \cdot\nabla\Omega_0\,.
\]
As asserted in Proposition~\ref{prop:Lam}, the operator $L$ is self-adjoint in
the space $L^2(\R^2)$, whereas $\hat\Lambda$ is skew-adjoint. Moreover, we have
$\|h\|_{L^2} = \|w_0\|_\cY$ by definition. We are interested in estimating the
$L^2$ norm of the function $|\xi| h(\xi,\tau)$ outside the disk of radius $\rho$
centered at the origin. The desired localization is obtained by setting
$g(\xi,\tau) = \chi(\xi)h(\xi,\tau)$, where $\chi(\xi) = \psi(|\xi|)$ and
$\psi : \R_+ \to (0,1]$ is a smooth function satisfying
\[
  \psi(r) \,=\, \begin{cases} e^{-\rho^2} & \text{if }~ r \le \rho/2\,,\\
  1 & \text{if }~ r \ge \rho\,,\end{cases} \qquad \text{and}\qquad 0 \le \psi'(r) \le 1
  \quad \forall\, r > 0\,.  
\]
A direct calculation leads to the evolution equation $\partial_\tau g = \bigl(L_\chi -
\delta^{-1}\hat\Lambda_\chi\bigr)g$, where
\[
  L_\chi g \,=\, \Delta g - \frac{2\nabla\chi}{\chi}\cdot\nabla g - \frac{|\xi|^2}{16}\,g
  + \Bigl(\frac12 + \frac{2|\nabla \chi|^2}{\chi^2} - \frac{\Delta \chi}{\chi}\Bigr)g\,,
\]
and $\hat\Lambda_\chi g = U_0\cdot\nabla g + \tilde\Lambda_\chi g$ with 
\begin{equation}\label{eq:tLamdef}
  \tilde\Lambda_\chi g \,=\, \chi p^{1/2}\,\BS\bigl[p^{-1/2}\chi^{-1}g\bigr]
  \cdot\nabla\Omega_0 \,=\, -\frac{1}{8\pi}\,\chi p^{-1/2}\,\xi\cdot
  \BS\bigl[p^{-1/2}\chi^{-1}g\bigr]\,.
\end{equation}

The idea is now to perform a standard energy estimate in $L^2(\R^2)$, namely
\[
  \frac12 \frac{\dd}{\dd \tau}\,\|g\|_{L^2}^2 \,=\, -\|\nabla g\|_{L^2}^2 -\frac{1}{16}
  \,\||\xi| g\|_{L^2}^2 + \frac12\,\|g\|_{L^2}^2 + \Bigl\|\frac{|\nabla \chi|}{\chi}\,g\Bigr\|_{L^2}^2
  -\frac{1}{\delta}\,\langle g\,, \tilde \Lambda_\chi g\rangle_{L^2}\,.
\]
In particular, since $|\nabla\chi| \le 1$ and $|g| \le \chi^{-1}|g| = |h|$, we have
\begin{equation}\label{eq:genergy}
  \frac18 \int_0^\infty \||\xi|g(\tau)\|_{L^2}^2\dd\tau \,\le\, \|g_0\|_{L^2}^2
  + \int_0^\infty\Bigl(3 \|h(\tau)\|_{L^2}^2 + \frac{2}{\delta}
  \,\bigl|\langle g(\tau)\,,\tilde \Lambda_\chi g(\tau)\rangle_{L^2}\bigr|\Bigr)\dd \tau\,,
\end{equation}
where $g_0 = \chi p^{1/2} \phi_0$. It remains to estimate the various terms in the
right-hand side of \eqref{eq:genergy}. We already know from \eqref{eq:w0bd2} that
\[
  \int_0^\infty \|h(\tau)\|_{L^2}^2\dd\tau \,=\, \int_0^\infty \|w_0(\tau)\|_\cY^2\dd\tau
  \,\le\, C_1 \delta^{1/3}\|\phi_0\|_\cY^2\,.
\]
To bound the other terms, the following elementary observation will be useful.
Since $\chi(\xi) = e^{-\rho^2}$ when $|\xi| \le \rho/2$ and $\chi(\xi) \le 1$
otherwise, we have for any $\mu \in (0,4)$:
\begin{equation}\label{eq:muchi}
  \sup_{\xi \in \R^2}\Bigl(\chi(\xi)\,e^{-\mu|\xi|^2}\Bigr) \,\le\, \max\Bigl(e^{-\rho^2},
  \,e^{-\mu\rho^2/4}\Bigr) \,=\, e^{-\mu\rho^2/4} \,\le\, \frac{\delta}{2}\,,
\end{equation}
provided $\rho^2 = N \log(2/\delta)$ with $N \ge 4/\mu$.

As a first application of \eqref{eq:muchi}, we consider the initial data $g_0 = \chi
p^{1/2} \phi_0$. Taking $\mu \in (0,4)$ such that $2\mu \le \gamma - 1/8$, we can bound
\[
  |g_0(\xi)| \,=\, \chi(\xi)\,e^{|\xi|^2/8}|\phi_0(\xi)| \,\le\, 
  \chi(\xi)\,e^{-(\gamma-1/8)|\xi|^2}\|\phi_0\|_\gamma \,\le\, e^{-\mu|\xi|^2}
  \Bigl(\chi(\xi)\,e^{-\mu|\xi|^2}\Bigr)\|\phi_0\|_\gamma\,, 
\]
and using \eqref{eq:muchi} we deduce that $\|g_0\|_{L^2} \le C \delta \|\phi_0\|_\gamma$.
Similarly, if $\mu \le 1/16$, we have
\[
  \chi(\xi) p(\xi)^{-1/2} \,=\, \chi(\xi)\,e^{-|\xi|^2/8} \,\le\, e^{-\mu|\xi|^2}
  \Bigl(\chi(\xi)\,e^{-\mu|\xi|^2}\Bigr)\,, \quad \text{hence}\quad
  \bigl\||\xi| \chi p^{-1/2}\bigr\|_{L^4} \,\le\, C\delta\,.
\]
Using \eqref{eq:tLamdef} and applying H\"older's inequality, we thus obtain
\begin{align*}
  \bigl|\langle g\,,\tilde \Lambda_\chi g\rangle_{L^2}\bigr| \,\le\,
  \bigl\|g\,\chi p^{-1/2}\,\xi\cdot \BS\bigl[p^{-1/2}h\bigr]\bigr\|_{L^1} \,&\le\, 
  C \|g\|_{L^2} \,\bigl\||\xi| \chi p^{-1/2}\bigr\|_{L^4}\,\bigl\|\BS\bigl[p^{-1/2}h\bigr]\bigr\|_{L^4} \\
  \,&\le\, C \delta\,\|g\|_{L^2}\,\|p^{-1/2}h\|_{L^{4/3}} \,\le\, C \delta \,\|h\|_{L^2}^2\,.
\end{align*}
Altogether, we deduce from \eqref{eq:genergy} that
\begin{equation}\label{eq:w0bdsec}
  \int_0^\infty \||\xi|g(\tau)\|_{L^2}^2\dd\tau \,\le\, C\biggl(\|g_0\|_{L^2}^2 + \int_0^\infty
  \|h(\tau)\|_{L^2}^2\dd\tau\biggr) \,\le\, C\Bigl(\delta^2 \|\phi_0\|_\gamma^2 +
  \delta^{1/3}\|\phi_0\|_\cY^2\Bigr)\,.
\end{equation}
Observing that $\||\xi|g(\tau)\|_{L^2} \ge \||\xi| \mathbf{1}_{\{|\xi| > \rho\}}h(\tau)\|_{L^2} = 
\||\xi| \mathbf{1}_{\{|\xi| > \rho\}}w_0(\tau)\|_\cY$, we see that estimate \eqref{eq:w0bd3}
is a direct consequence of \eqref{eq:w0bdfirst} and \eqref{eq:w0bdsec}. 
\end{proof}

It is unclear if enhanced dissipation estimates of the form \eqref{eq:w0bd1}
hold in weighted $L^q$ norms for $q > 2$. The following bound is certainly
not optimal, but will be sufficient for our purposes.

\begin{lemma}\label{lem:Lqweight}
Assume that $\phi_0 \in \cY \cap \Ker(\Lambda)^\perp$ satisfies $p^{1/2}\phi_0 \in
L^q(\R^2)$ for some $q \in (2,+\infty)$. Then there exists a constant $C_3 > 0$
such that, for any $\delta \in (0,1)$, the solution of \eqref{def:w0} satisfies
\begin{equation}\label{eq:sgbounds}
  \sup_{\tau \ge 0} \|p^{1/2}w_0(\tau)\|_{L^q} \,\le\, C_3\Bigl(
  \|p^{1/2}\phi_0\|_{L^q} + \frac{1}{\delta}\,\|\phi_0\|_\cY \Bigr)\,.
\end{equation}
\end{lemma}

\begin{proof}
We recall that the linear operator $\cL$ is the generator of a
strongly continuous semigroup in $\cY$ which satisfies, for any $q \ge 2$ and
any $\tau > 0$, the following estimates
\begin{equation}\label{eq:sgbounds2}
  \bigl\|p^{1/2}e^{\tau\cL}\phi\bigr\|_{L^q} \,\le\, C \|p^{1/2}\phi\|_{L^q}\,, \qquad
  \bigl\|p^{1/2}e^{\tau\cL}\nabla\phi\bigl\|_{L^q} \,\le\, \frac{C\,e^{-\tau/2}}{a(\tau)^{1-1/q}}
  \,\|p^{1/2}\phi\|_{L^2}\,,
\end{equation}
where $a(\tau) = 1-e^{-\tau}$, see \cite[Section~5.1]{Ga18}. To prove \eqref{eq:sgbounds}
we start from the integral formulation of equation \eqref{def:w0}, namely
\[
  w_0(\tau) \,=\, e^{\tau\cL}\phi_0 - \frac{1}{\delta}\int_0^\tau
  e^{(\tau-s)\cL}\,\nabla\cdot\Bigl(U_0 w_0(s) + v_0(s) \Omega_0\Bigr)\dd s\,,
\]
where $v_0 = \BS[w_0]$. Using estimates \eqref{eq:sgbounds2} we thus find
\begin{equation}\label{eq:Lqinteg}
  \|p^{1/2}w_0(\tau)\|_{L^q} \,\le\, C\|p^{1/2}\phi_0\|_{L^q} + \frac{C}{\delta}
   \int_0^\tau \frac{e^{-(\tau-s)/2}}{a(\tau-s)^{1-1/q}}\,\bigl\|U_0 w_0(s) + v_0(s)
  \Omega_0\bigr\|_\cY \dd s\,.
\end{equation}
To bound the integral term, we observe that $\|U_0 w_0(s)\|_\cY \le \|U_0\|_{L^\infty}
\|w_0(s)\|_\cY \le C \|\phi_0\|_\cY$ and that $\|v_0(s)\Omega_0\|_\cY \le
\|p^{1/2}\Omega_0\|_{L^4} \|v_0(s)\|_{L^4} \le C \|w_0(s)\|_{L^{4/3}} \le C \|w_0(s)\|_\cY
\le C \|\phi_0\|_\cY$. Since $2 < q < \infty$, we also have 
\[
  \int_0^\tau \frac{e^{-(\tau-s)/2}}{a(\tau-s)^{1-1/q}}\dd s \,\le\,
  \int_0^\infty \frac{e^{-s/2}}{a(s)^{1-1/q}}\dd s \,<\, \infty\,,
\]
hence \eqref{eq:sgbounds} follows directly from \eqref{eq:Lqinteg}. 
\end{proof}

\subsection{Energy estimates}\label{ssec42}

We now come back to the evolution equation \eqref{eq:wevol3} for the 
vorticity perturbation $w(\xi,\tau)$. We introduce the following decomposition 
\begin{equation}\label{eq:wdecomp}
  w(\xi,\tau) \,=\, w_0(\xi,\tau) + \tilde w(\xi,\tau)\,,
\end{equation}
where $w_0$ is the solution of the linear equation \eqref{def:w0}. The correction
$\tilde w(\xi,\tau)$ satisfies
\begin{equation}\label{eq:twevol}
  \partial_\tau \tilde w + \frac{1}{\delta}\,\Lambda \tilde w + \frac{1}{\delta}\,
  \cA[w_0 + \tilde w] + \cB[w_0 + \tilde w,w_0 + \tilde w] \,=\, \cL \tilde w
  - \frac{1}{\delta^2}\,\cR_\app\,,
\end{equation}
where $\cA$ is the linear operator \eqref{def:opcA} and $\cB$ the bilinear map
\eqref{def:opcB}.  Since $w_0(0) = w(0) = \phi_0$, the correction $\tilde w$
vanishes at initial time, and our goal is to show that this quantity remains
small in an appropriate function space for all $\tau \in [0,\log(T/t_0)]$. As in
Section~\ref{sec3}, the proof is much simpler if we assume that the observation
time $T$ is small compared to $T_0$.  So, for the sake of clarity, we first
provide the details of the argument in the simpler situation, and we return to
the general case at the end of this section.

\subsubsection{Relaxation for small time}\label{sssec421}

Under the assumption that $T/T_0 \ll 1$, the solution of \eqref{eq:twevol} with zero
initial data can be controlled using a simple energy estimate in the function
space $\cY$, as in Section~\ref{ssec31}.

\begin{proposition}\label{prop:twshort}
There exists a constant $C_4 > 0$ such that, if $T/T_0 \ll 1$ and $\delta
\in (0,1)$, the solution of \eqref{eq:twevol} with zero initial data
satisfies
\begin{equation}\label{eq:twshort}
  \|\tilde w(\tau)\|_\cY \,\le\, C_4 \Bigl(\frac{\epsilon^5}{\delta^2} + \epsilon^2\Bigr)
  + C_4\,\delta^{1/6}\Bigl(\log\frac{2}{\delta}\Bigr)^{1/2}\,\frac{t}{T_0}\,
  \Theta\Bigl(\frac{t_0}{T_0}\Bigr)\,,
\end{equation}
for all $\tau \in \bigl[0,\log(T/t_0)\bigr]$, where $\Theta(s) = s\bigl(1+\log_+(1/s)
\bigr)^{1/2}$ and $\log_+(s) = \max(\log(s),0)$. 
\end{proposition}

\begin{proof}
Following \eqref{eq:EFdef} we introduce the energy functional $\cE[\tilde w] =
\|\tilde w\|_\cY^2$ which satisfies
\begin{equation}\label{eq:twenergy}
  \frac12 \partial_\tau \cE[\tilde w] + \frac{1}{\delta}\langle \tilde w, \cA[w_0 + \tilde w]
  \rangle_\cY + \langle \tilde w, \cB[w_0 + \tilde w,w_0 + \tilde w]\rangle_\cY \,=\,
  \langle \tilde w,\cL\tilde w\rangle_\cY - \frac{1}{\delta^2} \langle \tilde w,\cR_\app
  \rangle_\cY\,.
\end{equation}
We first recall the estimates already obtained in Section~\ref{ssec31}. According to
\eqref{eq:short1}, there exists a constant $\kappa > 0$ such that
$\langle \tilde w,\cL\tilde w \rangle_\cY \le -\kappa\cF[\tilde w]$, where $\cF[\tilde w]$
is defined in \eqref{eq:EFdef}. In view of \eqref{eq:short4} and \eqref{eq:short3}, there
exists $C > 0$ such that
\begin{equation}\label{eq:old1}
  \frac{1}{\delta}\,\bigl|\langle \tilde w, \cA[\tilde w]\rangle_\cY\bigr| \,\le\,
  \frac{t}{T_0}\bigl(\cF[\tilde w] + C \cE[\tilde w]\bigr)\,, \quad \text{and} \quad 
  \bigl|\langle \tilde w,\cB[\tilde w,\tilde w]\rangle_\cY\bigr| \,\le\, C
  \cF[\tilde w]^{1/2}\cE[\tilde w]\,.
\end{equation}
Finally, the contribution of the source term $\cR_\app$ can be bounded as in \eqref{eq:short2}: 
\[
  \frac{1}{\delta^2}\,\bigl|\langle \tilde w,\cR_\app\rangle_\cY\bigr| \,\le\,
  C\Bigl(\frac{\epsilon^5}{\delta^2} + \epsilon^2\Bigr)\cE[\tilde w]^{1/2}\,.
\]

It remains to estimate all the terms in \eqref{eq:twenergy} that involve
the solution $w_0$ of the linear equation \eqref{def:w0}. We start with the
advection term $\cA[w_0]$. Integrating by parts, we obtain
\[
  \langle \tilde w, \cA[w_0]\rangle_\cY \,=\, -\int_{\R^2} \nabla \bigl(p\tilde w\bigr)
  \cdot \Bigl((U_\app - U_0)w_0 + v_0(\Omega_\app-\Omega_0) +  E(f,z)w_0\Bigr)\dd \xi\,,
\]
where $p(\xi) = e^{|\xi|^2/4}$. We know that $\|U_\app - U_0\|_{L^\infty} \le C\epsilon^2$,
that $\|p^{1/2}(\Omega_\app-\Omega_0)\|_{L^q} \le C\epsilon^2$ for any $q \ge 1$, and
that $|E(f,z)| \le C\epsilon^2(1+|\xi|)$. This gives
\[
  \bigl\| (U_\app - U_0)w_0 + v_0(\Omega_\app-\Omega_0)+  E(f,z)w_0\bigr\|_\cY
  \,\le\, C\epsilon^2 \bigl(\|w_0\|_{\cY} + \| |\xi| w_0\|_{\cY}\bigr)\,.
\]
On the other hand we have $\nabla(p\tilde w) = p\bigl(\nabla \tilde w + \xi\tilde w/2\bigr)$
where $\|\nabla \tilde w + \xi\tilde w/2\|_\cY \le 2\cF[\tilde w]^{1/2}$. Since
$\epsilon^2/\delta = t/T_0$ we conclude that
\begin{equation}\label{eq:Abdd}
  \frac{1}{\delta}\,\bigl|\langle \tilde w, \cA[w_0]\rangle_\cY\bigr| \,\le\, \frac{Ct}{T_0}\,
  \cF[\tilde w]^{1/2}\bigl(\|w_0\|_{\cY} + \| |\xi| w_0\|_{\cY}\bigr)\,.
\end{equation}

We next consider the various terms involving $w_0$ in the quadratic form
$\cB[w_0 + \tilde w,w_0 + \tilde w]$. Integrating by parts as above, we
observe that, for all $w_1,w_2 \in \cY$, 
\[
  \bigl|\langle \tilde w, \cB[w_1,w_2]\rangle_\cY\bigr| \,\le\, 2\cF[\tilde w]^{1/2}
  \bigl\|v_1 w_2\bigr\|_\cY\,, \qquad \text{where}\quad v_1 \,=\, \BS[w_1]\,. 
\]
We first take $w_1 = w_0$ and $w_2 = w_0 + \tilde w$, in which case
$\|v_1 w_2\|_\cY \le \|v_0\|_{L^\infty}\bigl(\|w_0\|_\cY + \|\tilde  w\|_\cY\bigr)$.
To estimate the $L^\infty$ norm of $v_0 = \BS[w_0]$, we invoke \cite[Lemma~5.5]{Ga18}
which asserts that
\[
  \|v_0\|_{L^\infty} \,\le\, C \|w_0\|_{L^1 \cap L^2} \Bigl( 1 +
  \log_+ \frac{\|w_0\|_{L^3}}{\|w_0\|_{L^1 \cap L^2}}\Bigr)^{1/2} \,\le\,
  C\Bigl(\log\frac{2}{\delta}\Bigr)^{1/2}\,\Theta\bigl(\|w_0\|_\cY\bigr)\,,
\]
where in the second inequality we used the fact that $\|w_0\|_{L^1 \cap L^2} \le
C\|w_0\|_\cY$ and $\|w_0\|_{L^3} \le C\delta^{-1}$, see \eqref{eq:sgbounds}. 
We deduce that
\begin{equation}\label{eq:Bbdd1}
  \bigl|\langle \tilde w, \cB[w_0,w_0 + \tilde w]\rangle_\cY\bigr| \,\le\,
  C \Bigl(\log\frac{2}{\delta}\Bigr)^{1/2} \cF[\tilde w]^{1/2} \,
  \Theta\bigl(\|w_0\|_\cY\bigr)\bigl(\|w_0\|_\cY + \|\tilde w\|_\cY\bigr)\,.
\end{equation}
The second case is $w_1 = \tilde w$ and $w_2 = w_0$. Here we invoke \cite[Lemma~5.6]{Ga18}
which gives
\[
  \|\tilde v w_0\|_\cY \,\le\, C \|w_0\|_\cY \|\tilde w\|_{L^1\cap L^2}
  \Bigl( 1 + \log_+ \frac{\|p^{1/2}w_0\|_{L^3}}{\|w_0\|_\cY}\Bigr)^{1/2} \,\le\,
  C \Bigl(\log\frac{2}{\delta}\Bigr)^{1/2}\Theta\bigl(\|w_0\|_\cY\bigr) \|\tilde w\|_\cY\,,
\]
and we conclude that 
\begin{equation}\label{eq:Bbdd2}
  \bigl|\langle \tilde w, \cB[\tilde w,w_0]\rangle_\cY\bigr| \,\le\,
  C \Bigl(\log\frac{2}{\delta}\Bigr)^{1/2} \cF[\tilde w]^{1/2}\,
  \Theta\bigl(\|w_0\|_\cY\bigr)\|\tilde w\|_\cY\,.
\end{equation}

Summarizing, if we collect all estimates \eqref{eq:old1}--\eqref{eq:Bbdd2} we obtain
the inequality
\begin{align*}
  \partial_\tau \cE[\tilde w] &+ \Bigl(2\kappa - \frac{2t}{T_0}\Bigr)\cF[\tilde w] \,\le\,
  C\Bigl(\frac{\epsilon^5}{\delta^2} + \epsilon^2\Bigr)\cE[\tilde w]^{1/2} +                  
  \frac{Ct}{T_0}\,\cE[\tilde w] + C_0 \cF[\tilde w]^{1/2}\cE[\tilde w] \\
  \,&+\, \frac{Ct}{T_0}\,\cF[\tilde w]^{1/2}\bigl(\|w_0\|_{\cY} + \| |\xi| w_0\|_{\cY}\bigr)
  + C \Bigl(\log\frac{2}{\delta}\Bigr)^{1/2} \cF[\tilde w]^{1/2} \Theta\bigl(\|w_0\|_\cY\bigr)
  \bigl(\|w_0\|_\cY + \|\tilde w\|_\cY\bigr)\,,
\end{align*}
for some positive constants $C$ and $C_0$. As in Section~\ref{ssec31}, we suppose
that $t/T_0 \le T/T_0 \le \kappa/2$, and we work under the assumption that
$C_0 \cE[\tilde w]^{1/2} \le \kappa/2$, which will be verified a posteriori. 
Using Young's inequality and the fact that $\cE[\tilde w] \le \cF[\tilde w]$, 
we obtain the simpler relation
\begin{equation}\label{eq:diffcE}
  \partial_\tau \cE[\tilde w(\tau)] \,\le\, \cM(\tau)\,\cE[\tilde w(\tau)]
  + \cS(\tau)\,, \qquad 0 \le \tau \le \log(T/t_0)\,,
\end{equation}
where
\begin{align*}
  \cM(\tau) \,&=\, \frac{Ct}{T_0} + C \Bigl(\log\frac{2}{\delta}\Bigr)\Theta\bigl(
  \|w_0(\tau)\|_\cY\bigr)^2\,, \\
  \cS(\tau) \,&=\, C\Bigl(\frac{\epsilon^5}{\delta^2} + \epsilon^2\Bigr)^2
  \!+ C\Bigl(\frac{t}{T_0}\Bigr)^2\Bigl(\|w_0(\tau)\|_{\cY}^2 + \| |\xi| w_0(\tau)\|_{\cY}^2\Bigr)
  \\ &\quad\, + C \Bigl(\log\frac{2}{\delta}\Bigr)\Theta\bigl(\|w_0(\tau)\|_\cY\bigr)^2
  \|w_0(\tau)\|_\cY^2\,.
\end{align*}
Since $t = t_0 e^\tau$ and $\|w_0(\tau)\|_\cY$ satisfies \eqref{eq:w0bd1}, we have
\[
  \int_0^{\log(T/t_0)} \cM(\tau)\dd\tau \,\le\, \frac{CT}{T_0} + C \delta^{1/3}
  \Bigl(\log\frac{2}{\delta}\Bigr)\Theta\bigl(\|\phi_0\|_\cY\bigr)^2 \,\le\, K\,,
\]
for some constant $K > 0$. Similarly, using Propositions~\ref{prop:ED1} and
\ref{prop:ED2} with $\gamma \in (1/8,1/4)$, we find
\begin{align*}
  \int_0^\tau \cS(\tau')\dd\tau' \,&\le\, C\Bigl(\frac{\epsilon^5}{\delta^2} + \epsilon^2\Bigr)^2 
  +  C \Bigl(\frac{t}{T_0}\Bigr)^2\,\delta^{1/3}\Bigl(\log\frac{2}{\delta}\Bigr)\|\phi_0\|_{\gamma}^2
  + C\,\delta^{1/3}\Bigl(\log\frac{2}{\delta}\Bigr) \Theta\bigl(\|\phi_0\|_\cY\bigr)^2\|\phi_0\|_\cY^2 \\
  \,&\le\, C\Bigl(\frac{\epsilon^5}{\delta^2} + \epsilon^2\Bigr)^2 + 
  C\,\delta^{1/3}\Bigl(\log\frac{2}{\delta}\Bigr) \Bigl(\frac{t}{T_0}\Bigr)^2
  \Theta\Bigl(\frac{t_0}{T_0}\Bigr)^2\,,
\end{align*}
because $\|\phi_0\|_\cY \le C\|\phi_0\|_\gamma \le Ct_0/T_0$ and $t_0 \le t$. 
So, applying Gr\"onwall's lemma to the differential inequality \eqref{eq:diffcE}, we
obtain
\begin{equation}\label{eq:boundcE}
  \cE[\tilde w(\tau)] \,\le\, e^K \int_0^\tau \cS(\tau')\dd\tau' \,\le\,
  C\Bigl(\frac{\epsilon^5}{\delta^2} + \epsilon^2\Bigr)^2 +
  C\,\delta^{1/3}\Bigl(\log\frac{2}{\delta}\Bigr) \Bigl(\frac{t}{T_0}\Bigr)^2
  \Theta\Bigl(\frac{t_0}{T_0}\Bigr)^2\,.
\end{equation}
If $\delta$ is sufficiently small, this ensures that the a priori estimate $C_0\cE[\tilde w]^{1/2}
\le \kappa/2$ holds whenever $\tau \le \log(T/t_0)$. Finally, since $\|\tilde w(\tau)\|_\cY =
\cE[\tilde w(\tau)]^{1/2}$ we see that \eqref{eq:twshort} follows from \eqref{eq:boundcE}.
\end{proof}

\subsubsection{Relaxation for large time}\label{sssec422}

In the general situation where $T/T_0$ is not small, more complicated energy
functionals are needed to control the solutions of \eqref{eq:twevol}, even in
the particular case $w_0 = 0$ which was considered in Sections~\ref{ssec32} and
\ref{ssec33}. As a matter of fact, the terms involving $w_0$ in \eqref{eq:twevol}
do not create any real trouble, and can be treated exactly as in the proof of
Proposition~\ref{prop:twshort}. Since the estimates given by
Propositions~\ref{prop:ED1} and \ref{prop:ED2} hold in the space $\cY$, it is
preferable here to use the weight function $\hat p_\epsilon$ defined in
\eqref{eq:hpepsdef}, which satisfies the upper bound in \eqref{eq:hpest}.
Also, as already explained, it is convenient to express all quantities in terms
of the logarithmic time \eqref{eq:taudef}, instead of the original time $t \in [t_0,T]$. 
We thus consider the energy functionals
\begin{align*}
  \hat\cE(\tau) \,&=\, \int_{\R^2} \hat p_\epsilon(\xi,t) \tilde w(\xi,\tau)^2\dd\xi\,, \\
  \hat\cF(\tau) \,&=\, \int_{\R^2} \hat p_\epsilon(\xi,t)\Bigl\{|\nabla \tilde w(\xi,\tau)|^2 +
  \hat\chi_\epsilon(\xi) \tilde w(\xi,\tau)^2 + \tilde w(\xi,\tau)^2\Bigr\}\dd\xi\,,
\end{align*}
where $\hat\chi_\epsilon$ is defined in \eqref{eq:hchidef}. Here and in what
follows it is understood that $\epsilon = \sqrt{\nu t}/d$ with $t = t_0 e^\tau$.
The analogue of Proposition~\ref{prop:twshort} is:

\begin{proposition}\label{prop:twlong}
There exists a constant $C_5 > 0$ such that, if $\delta > 0$ is small enough, 
the solution of \eqref{eq:twevol} with zero initial data satisfies
\begin{equation}\label{eq:twlong}
  \hat\cE(\tau)^{1/2} \,\le\, C_5 \Bigl(\frac{\epsilon^5}{\delta^2} + \epsilon^2\Bigr)
  + C_5\,\delta^{1/6}\Bigl(\log\frac{2}{\delta}\Bigr)^{1/2}\,\frac{t}{T_0}\,
  \Theta\Bigl(\frac{t_0}{T_0}\Bigr)\,,
\end{equation}
for all $\tau \in \bigl[0,\log(T/t_0)\bigr]$, where $\Theta(s) = s\bigl(1+\log_+(1/s)
\bigr)^{1/2}$.
\end{proposition}

\begin{proof}
We only give a sketch of the argument, which simply combines the estimates already
obtained in the proofs of Propositions~\ref{prop:largetime} and \ref{prop:twshort}.
In analogy with \eqref{eq:westlarge} we have
\begin{equation}\label{eq:westlarge3}
  \partial_\tau\hat\cE(\tau) + \kappa\hat\cF(\tau) \,\le\, K_4\Bigl(\frac{\epsilon^5}{\delta^2}
  + \epsilon^2\Bigr)\hat\cE(\tau)^{1/2} + K_5\Bigl(\frac{t}{T_0} + \hat\cF(\tau)^{1/2}\Bigr)
  \hat\cE(\tau) - \cG(\tau)\,,
\end{equation}
where the additional term $\cG(\tau)$ collects all contributions due to $w_0$,
namely
\[
  \cG(\tau) \,=\, 2\int_{\R^2} \hat p_\epsilon \tilde w \,\Bigl(\frac{1}{\delta}\,
  \cA[w_0] + \cB[w_0,w_0] + \cB[w_0,\tilde w]+ \cB[\tilde w,w_0]\Bigr)\dd\xi\,.
\]
Integrating by parts and proceeding as in \eqref{eq:Abdd}, \eqref{eq:Bbdd1},
and \eqref{eq:Bbdd2}, we easily find
\begin{equation}\label{eq:cGbdd}
  |\cG(\tau)| \,\le\, C\hat\cF(\tau)^{1/2}\biggl\{\frac{t}{T_0}\bigl(\|w_0\|_{\cY} +
  \| |\xi| w_0\|_{\cY}\bigr) + \Bigl(\log\frac{2}{\delta}\Bigr)^{1/2} \Theta\bigl(\|w_0\|_\cY\bigr)
  \bigl(\|w_0\|_\cY + \|\tilde w\|_\cY\bigr)\biggr\}\,.
\end{equation}
Here we used the fact that $\hat p_\epsilon(\xi,t) \le C e^{|\xi|^2/4}$, see \eqref{eq:hpest}.
Applying Young's inequality to \eqref{eq:cGbdd} and returning to \eqref{eq:westlarge3},
we arrive at a differential inequality of the form 
\[
  \partial_\tau \hat\cE(\tau) \,\le\, \cM(\tau)\,\hat\cE(\tau) + \cS(\tau)\,,
  \qquad 0 \le \tau \le \log(T/t_0)\,,
\]
where $\cM(\tau)$ and $\cS(\tau)$ are exactly as in \eqref{eq:diffcE}. 
Estimate \eqref{eq:twlong} is then obtained by the same argument as
in Proposition~\ref{prop:twshort}.
\end{proof}

\begin{proof}[End of the proof of Theorem~\ref{thm2}]
As in the proof of Theorem~\ref{thm1} we estimate the quantity
\[
  \cI_1(t) \,=\, \int_{\R^2} \bigl|\Omega(\xi,t) - \Omega_\app(\xi,t)\bigr|\dd\xi
 \,\le\, \delta \int_{\R^2} |w_0(\xi,\tau)|\dd \xi \,+\, \delta \int_{\R^2} |\tilde
  w(\xi,\tau)|\dd \xi\,,
\]
where we used the decompositions \eqref{OmUdecomp2} and \eqref{eq:wdecomp}.
To bound the first term in the right-hand side, we apply Proposition~\ref{prop:ED1},
and we recall that $\|\phi_0\|_\cY \le C \epsilon_0^2/\delta$ where $\epsilon_0 =
\sqrt{\nu t_0}/d$, see \eqref{def:winit0}. Since $\tau = \log(t/t_0)$, 
we find
\[
  \delta\,\|w_0(\cdot,\tau)\|_{L^1} \,\le\, C \delta\,\|w_0(\cdot,\tau)\|_\cY
  \,\le\, C \delta\,\|\phi_0\|_\cY\,\exp\Bigl(-\frac{c_0\tau}{\delta^{1/3}}\Bigr) \,\le\,
  C \epsilon_0^2\,\Bigl(\frac{t_0}{t}\Bigr)^\beta\,,
\]
where $\beta = c_0 \delta^{-1/3} \gg 1$. For the second term we invoke
Proposition~\ref{prop:twlong} which gives
\[
  \delta\,\|\tilde w(\cdot,\tau)\|_{L^1} \,\le\, C \delta \,\|\tilde w(\cdot,\tau)\|_\cY
  \,=\, C\delta \,\hat\cE(\tau)^{1/2} \,\le\, C\epsilon^2\bigl(\epsilon + \delta\bigr)
  + C \epsilon^2 \,\delta^{1/6}\Bigl(\log\frac{2}{\delta}\Bigr)^{1/2}
  \Theta\Bigl(\frac{t_0}{T_0}\Bigr)\,,
\]
where we used again the relation $\epsilon^2 = \delta t/T_0$. We thus arrive at
\[
  \cI_1(t) \,\le\, C \epsilon^2 \biggl\{\epsilon + \delta + \delta^{1/6}
  \Bigl(\log\frac{2}{\delta}\Bigr)^{1/2}\Theta\Bigl(\frac{t_0}{T_0}\Bigr)
  + \Bigl(\frac{t_0}{t} \Bigr)^{\beta+1}\biggr\}\,, \qquad t \in (t_0,T)\,,
\]
and the second integral that appears in the proof of Theorem~\ref{thm1} satisfies
$\cI_2(t) \le C\epsilon^2(\epsilon + \delta)$. Altogether we obtain estimate \eqref{eq:thm2}. 
\end{proof}

\appendix

\section{Appendix}\label{appendix}

\subsection{Comparison with the Burgers vortex}\label{ssecA1}

We consider here the simple example of a time-independent velocity field 
of the form
\begin{equation}\label{eq:fstrain}
  f(x) \,=\, \frac{\gamma}{2}\,\begin{pmatrix} -x_1 \\ x_2\end{pmatrix}\,,
  \qquad \forall\,x \in \R^2\,,
\end{equation}
where the strain rate $\gamma > 0$ is a parameter. Using the definitions
\eqref{eq:T0def} and \eqref{def:ab}, we see that $T_0 = 1/\gamma$, $a_f(z) = -\gamma/2$,
and $b_f(z) = 0$ in the present case. Let $\omega(x,t)$ be the solution of
\eqref{eq:NSf} and \eqref{eq:BS} satisfying $\omega(\cdot,t) \weakto \Gamma \delta_0$
as $t \to 0$. Applying the self-similar change of coordinates \eqref{eq:OmU} with
$z(t) = 0$, we obtain the evolution equation \eqref{eq:Omevol} which takes the form
\begin{equation}\label{eq:OmSS}
  t\partial_t \Omega + \frac{1}{\delta}\,U\cdot\nabla \Omega \,=\,
  \cL \Omega + \gamma t \cM \Omega\,, \qquad U = \BS[\Omega]\,,
\end{equation}
where $\cL$ is defined in \eqref{def:cL} and $\cM = \frac{1}{2}\bigl(\xi_1
\partial_1 - \xi_2\partial_2\bigr)$. If we freeze time in \eqref{eq:OmSS}, we arrive
at the elliptic equation
\begin{equation}\label{eq:Omell}
  \frac{1}{\delta}\,U\cdot\nabla \Omega \,=\, \cL \Omega + \lambda\cM \Omega\,,
  \qquad U = \BS[\Omega]\,,
\end{equation}
where $\lambda = \gamma t$. This is exactly the equation satisfied by
the profile of a Burgers vortex with Reynolds number $1/\delta$
and asymmetry parameter $\lambda$, see \cite{RS84,MKO94,GW07}. In particular
the results of \cite{GW07,Ma09a,Ma09b} show that, if $\lambda \in (0,1)$ is
fixed and $\delta > 0$ is sufficiently small, equation~\eqref{eq:Omell}
has a unique solution $\Omega_{\lambda,\delta} \in L^1(\R^2)$ which satisfies
\begin{equation}\label{eq:OmBurg}
  \Omega_{\lambda,\delta}(\xi) \,=\, \Omega_0(\xi) - \frac{1}{2}\,\lambda \delta\,w_2(|\xi|)
  \sin(2\theta) + \cO(\delta^2)\,,
\end{equation}
where $w_2$ is precisely the function considered in Remark~\ref{rem:w2}.
Comparing \eqref{eq:OmBurg} with approximate solution \eqref{def:omapp}
in the particular case of the external field \eqref{eq:fstrain}, we deduce
that
\[
  \int_{\R^2}\biggl|\frac{\Gamma}{\nu t}\,\Omega_{\gamma t,\delta}\Bigl(\frac{x}{
  \sqrt{\nu t}}\Bigr) - \omega_\app\bigl(\Gamma,\sqrt{\nu t},0,f\,;x)\,\biggr|\dd x
  \,=\, \cO\bigl(\Gamma\delta^2\bigr)\,, \qquad \text{as}~\, \delta \to 0\,.
\]
This means that, for $t \in (0,T_0)$ and $\delta > 0$ sufficiently small,
the approximate solution \eqref{def:omapp} is essentially a rescaling of the
vorticity profile of a Burgers vortex with asymmetry parameter $\gamma t
\in (0,1)$ and Reynolds number $1/\delta \gg 1$. 

\subsection{Partial inverse for the advection operator $\Lambda$}\label{ssecA2}

In this section, for completeness, we recall the known formulas for the
(partial) inverse of the integro-differential operator $\Lambda$ defined in
\eqref{def:Lambda}. More details can be found in the references
\cite{Ga11,GaS24}. Since $\Lambda$ leaves invariant the direct sum decomposition
\eqref{Ydecomp}, it is sufficient to study the restriction of $\Lambda$ to each
subspace $\cY_n$. Actually $\cY_0 \subset \Ker(\Lambda)$ by \eqref{eq:KerLam},
so we can assume that $n \ge 1$.  To exploit the rotational symmetry, we use polar
coordinates in $\R^2$ defined by $\xi = (r\cos\theta,r\sin\theta)$, and we
consider the radially symmetric functions
\begin{equation}\label{def:vgh}
  v_*(r) \,=\, \frac{1}{2\pi r^2}\bigl(1-e^{-r^2/4}\bigr)\,,\qquad
  g(r) \,=\, \frac{1}{8\pi}\,e^{-r^2/4}\,, \qquad h(r) \,=\,
  \frac{g(r)}{v_*(r)} \,=\, \frac{r^2/4}{e^{r^2/4}-1}\,.
\end{equation}
Note that $U_0(\xi) = v_*(|\xi|)\xi^\perp$ and $\nabla\Omega_0(\xi) = -g(|\xi|)\xi$,
where $\Omega_0$ and $U_0$ are defined in \eqref{eq:OmU0}. 

Assume that $\Omega \in \cY_n$ takes the form $\Omega = -w(r)\cos(n\theta)$ for some
function $w : \R_+ \to \R$. As is easily verified, the associated velocity field is
\begin{equation}\label{eq:Undef}
  U \,=\, \BS[\Omega] \,=\, \frac{n}{r}\,\varphi(r)\sin(n\theta)\,e_r +
  \varphi'(r)\cos(n\theta)\,e_\theta\,,
\end{equation}
where $e_r = \xi/|\xi|$, $e_\theta = \xi^\perp/|\xi|$, and  $\varphi$ is the unique
solution of the ordinary differential equation
\begin{equation}\label{eq:varphidef}
  -\varphi'' (r) - \frac{1}{r}\,\varphi'(r) + \frac{n^2}{r^2}\,\varphi(r)
  \,=\, w(r)\,, \qquad r > 0\,,
\end{equation}
satisfying the boundary conditions $\varphi(r) = \cO(r^n)$ as $r \to 0$ and
$\varphi(r) = \cO(r^{-n})$ as $r \to +\infty$. Using \eqref{def:vgh} and \eqref{eq:Undef},
we easily obtain
\begin{equation}\label{eq:Lamn}
  \Lambda \Omega \,=\, U_0 \cdot\nabla\Omega + U \cdot\nabla\Omega_0 \,=\,
  n\bigl(v_* w - g\varphi\bigr)\sin(n\theta)\,.
\end{equation}
Similarly, if $\Omega = w(r)\sin(n\theta)$, then $\Lambda \Omega =
n\bigl(v_* w - g\varphi\bigr)\cos(n\theta)$.

Now we give ourselves a function $F \in \cY_n$ of the form $F = b(r)\sin(n\theta)$. 
If the inhomogeneous differential equation
\begin{equation}\label{eq:varphidef2}
  -\varphi'' (r) - \frac{1}{r}\,\varphi'(r) + \Bigl(\frac{n^2}{r^2} - h(r)\Bigr)
  \varphi(r) \,=\, \frac{b(r)}{nv_*(r)}\,, \qquad r > 0\,,
\end{equation}
has a (unique) solution $\varphi$ satisfying the boundary conditions, we define
$\Omega = -w(r)\cos(n\theta)$ with 
\begin{equation}\label{def:wsol}
  w(r) \,=\, \varphi(r)h(r) + \frac{b(r)}{nv_*(r)}\,, \qquad r > 0\,,
\end{equation}
Then \eqref{eq:varphidef} is obviously satisfied, and \eqref{eq:Lamn} implies
that $\Lambda\Omega = F$. The same conclusion holds if $F = b(r)\cos(n\theta)$ and
$\Omega = w(r)\sin(n\theta)$. The level lines of $\Omega$ in the case $n=2$ are depicted on
the right of Figure~\ref{fig:levelLines}.

This discussion shows that the invertibility of the operator $\Lambda$ in
the subspace $\cY_n$ is linked to the solvability of the ODE \eqref{eq:varphidef2}.
The favorable case is $n \ge 2$, because the coefficient $n^2/r^2 - h(r)$ is positive,
which ensures that \eqref{eq:varphidef2} has a unique solution satisfying the boundary
conditions. If $\cZ$ is the function space \eqref{def:cZ}, we thus obtain the
following result: 

\begin{lemma}\label{lem:Lambda} {\bf \cite{Ga11}}
If $n\ge 2$ and $F\in \cY_n\cap \cZ$, there exists a unique $\Omega\in \cY_n
\cap \cZ$ such that $\Lambda\Omega =F$. Moreover, if $F=b(r)\sin(n\theta)$
(respectively, $F = b(r)\cos(n\theta)$) then $\Omega = -w(r)\cos(n\theta)$
(respectively, $\Omega = w(r)\sin(n\theta)$) where $w$ is defined
by \eqref{def:wsol} with $\varphi$ given by \eqref{eq:varphidef2}. 
\end{lemma}

If $n = 1$, the homogeneous differential equation \eqref{eq:varphidef2} with
$b = 0$ has a nontrivial solution $\varphi = rv_*$ which satisfies the boundary
conditions. As a consequence, the inhomogeneous equation can be solved only if
the right-hand side satisfies $\int_0^\infty b(r)r^2\dd r = 0$, and the solution
is never unique. The solvability condition ensures that $F$ belongs to the subspace
$\cY_1'$ defined by
\begin{equation}\label{eq:Y1Kerperp}
  \cY_1' \,=\, \cY_1\cap \Ker(\Lambda)^\perp \,=\, \Bigl\{F\in \cY_1 \, ; \,
  \int_{\R^2}\xi_1F(\xi)\dd\xi = \int_{\R^2}\xi_2F(\xi)\dd\xi = 0\Bigr\}\,,
\end{equation}
see also \eqref{eq:solvcond}. We have the following result, which
complements Lemma~\ref{lem:Lambda}. 

\begin{lemma}\label{lem:Lambda2} {\bf \cite{GaS24}}
If $n = 1$ and $F\in \cY_1'\cap \cZ$, there exists a unique $\Omega\in \cY_1'
\cap \cZ$ such that $\Lambda\Omega =F$. Moreover, if $F=b(r)\sin(\theta)$
(respectively, $F = b(r)\cos(\theta)$) then $\Omega = -w(r)\cos(\theta)$
(respectively, $\Omega = w(r)\sin(\theta)$) where $w$ is defined
by \eqref{def:wsol} with $\varphi$ given by \eqref{eq:varphidef2}. 
\end{lemma}

\subsection{Estimates on the velocity field}\label{ssecA3}

We collect here, for easy reference, a few classical estimates on the Biot-Savart
operator \eqref{eq:BS} which are used in the proof of our main results. Given
a vorticity distribution $\omega$, we define $u = \BS[\omega]$ as in \eqref{eq:BS}.

\begin{lemma}\label{lem:BS1} {\bf \cite[Lemma~2.1]{GW02}} Assume that
$1 \le p < 2 < q \le \infty$.\\[1mm]
1) If $\frac{1}{q} = \frac{1}{p} -\frac12$ then $\|u\|_{L^q} \le C\|\omega\|_{L^p}$.\\[1mm]
2) If $\frac12 = \frac{\theta}{p} + \frac{1-\theta}{q}$ with $\theta \in (0,1)$, 
then $\|u\|_{L^\infty} \le C\|\omega\|_{L^p}^\theta \|\omega\|_{L^q}^{1-\theta}$. 
\end{lemma}

Let $b$ be the weight function defined by $b(x) = (1+|x|^2)^{1/2}$ for $x \in \R^2$. 

\begin{lemma}\label{lem:BS2} {\bf \cite[Proposition~B.1]{GW02}}
Assume that $m \in (1,2)$. \\[1mm]
If $b^m \omega \in L^2(\R^2)$ and $\int_{\R^2} \omega \dd x = 0$, then
$\|b^{m-2/q}u\|_{L^q} \le C \|b^m\omega\|_{L^2}$ for all $q \in (2,\infty)$. \\[1mm]
In particular, by H\"older's inequality, $\|u\|_{L^2} \le C \|b^m\omega\|_{L^2}$. 
\end{lemma}

Let $\cZ$ be the function space defined by \eqref{def:cZ}. The following
statement can be established using the same arguments as in
\cite[Appendix~B]{GW02}. 

\begin{lemma}\label{lem:BS3}
If $\omega \in \cZ$, then $u \in \cS_*(\R^2)$ and $b u \in L^\infty(\R^2)$. Moreover:
\\[1mm]
1) If $\int_{\R^2} \omega\dd x = 0$, then $b^2 u \in L^\infty(\R^2)$; \\[1mm]
2) If $\int_{\R^2} \omega\dd x = 0$ and $\int_{\R^2} x_j \omega\dd x = 0$ for $j = 1,2$,
then $b^3 u \in L^\infty(\R^2)$. 
\end{lemma}

\subsection{Proof of Lemma~\ref{lem:Qeps}}\label{ssecA4}

The parameter $t \in (0,T)$ does not play any role in the argument here, so we
omit the time dependence of all quantities. Given $\epsilon > 0$ sufficiently
small, we give ourselves two smooth and radially symmetric functions
$\zeta_1, \zeta_2 : \R^2 \to [0,1]$ such that $\zeta_1(\xi)^2 + \zeta_2(\xi)^2 = 1$
for all $\xi \in \R^2$, $\zeta_1(\xi) = 1$ whenever $|\xi| \le \epsilon^{-1/4}$,
and $\zeta_2(\xi) = 1$ whenever $|\xi| \ge 2\epsilon^{-1/4}$. It is well-known
that such a partition of unity exists, and we can also assume that $|\nabla\zeta_1(\xi)|
+ |\nabla\zeta_2(\xi)| \le C\epsilon^{1/4}$ for all $\xi \in \R^2$. 

Given $w : \R^2 \to \R$ we define $w_1 = \zeta_1 w$ and $w_2 = \zeta_2 w$, so
that $w^2 = w_1^2 + w_2^2$. A direct calculation shows that $|\nabla w|^2 =
|\nabla w_1|^2 + |\nabla w_2|^2 - w^2\bigl(|\nabla\zeta_1|^2 + |\nabla\zeta_2|^2\bigr)$.
Thus, recalling the definition \eqref{eq:cQdef} of $\cQ_\epsilon[w]$, we have
\begin{equation}\label{eq:Qeps0}
\begin{split}
  \cQ_\epsilon[w] \,&=\, \cQ_\epsilon[w_1] + \cQ_\epsilon[w_2] - \int_{\R^2}
  p_\epsilon w^2 \bigl(|\nabla\zeta_1|^2{+}|\nabla\zeta_2|^2\bigr)\dd \xi \\
  \,&\ge\, \cQ_\epsilon[w_1] + \cQ_\epsilon[w_2] - C\epsilon^{1/2}\cE_\epsilon[w]\,.
\end{split}
\end{equation}
It is therefore sufficient to obtain lower bounds on the quantities $\cQ_\epsilon[w_1]$
and $\cQ_\epsilon[w_2]$. 

To estimate $\cQ_\epsilon[w_2]$ we apply H\"older's inequality to obtain
\[
  \biggl| \int_{\R^2} w_2 (\nabla w_2\cdot\nabla p_\epsilon)\dd\xi \biggr| \,\le\,
  \frac34 \int_{\R^2} p_\epsilon|\nabla w_2|^2\dd\xi + \frac13 \int_{\R^2}
  \frac{|\nabla p_\epsilon|^2}{p_\epsilon}\,w_2^2\dd\xi\,,
\]
so that
\begin{equation}\label{eq:Qeps1}
  \cQ_\epsilon[w_2] \,\ge\, \int_{\R^2} p_\epsilon\Bigl\{\frac14\,|\nabla w_2|^2
  + \Bigl(V_\epsilon -\frac12\Bigr) w_2^2\Bigr\}\dd\xi\,, \qquad
  V_\epsilon(\xi) \,=\, \frac{\xi\cdot\nabla p_\epsilon}{4p_\epsilon}  -
  \frac{|\nabla p_\epsilon|^2}{3p_\epsilon^2}\,.
\end{equation}
Using the definition \eqref{eq:pepsdef} of the weight $p_\epsilon$, it is not
difficult to verify that, under the assumptions of Lemma~\ref{lem:Ieps}, 
\[
  V_\epsilon(\xi) \,=\, \begin{cases} \frac{|\xi|^2}{24}\bigl(
  1 + \cO(A^2)\bigr) & \text{ if } \xi \in \I_\epsilon\,, \\[1mm]
  ~0 & \text{ if } \xi \in \II_\epsilon\,, \\[1mm]
  |\xi|^2 \Bigl(\frac{\gamma}{8} - \frac{\gamma^2}{12}\Bigr) & \text{ if } \xi \in \III_\epsilon\,.
  \end{cases}
\]
We further observe that $\gamma/8 - \gamma^2/12 \ge \gamma/16$ as soon as
$\gamma \le 3/4$, and that $w_2$ vanishes for $|\xi| \le \epsilon^{-1/4}$, which
implies that
\[
  \Bigl(V_\epsilon(\xi) -\frac{1}{2}\Bigr) w_2(\xi)^2 \,\ge\, \kappa\bigl(|\xi|^2+1\bigr)
  w_2(\xi)^2\,, \qquad \forall \xi \in \I_\epsilon \,,
\]
for some $\kappa < 1/24$. Thus, assuming that $\epsilon, A$ are as in
Lemma~\ref{lem:Ieps}, we obtain the lower bound
\begin{equation}\label{eq:Qeps2}
  \cQ_\epsilon[w_2] \,\ge\, \frac14 \int_{\R^2} p_\epsilon|\nabla w_2|^2\dd\xi +
  \kappa \int_{\I_\epsilon \cup \III_\epsilon}\bigl(\chi_\epsilon + 1\bigr) p_\epsilon
  w_2^2\dd\xi - \frac12 \int_{\II_\epsilon}p_\epsilon w_2^2\dd\xi\,.
\end{equation}

On the other hand, since $w_1$ is supported in the region $\I_\epsilon$ where the weight
$p_\epsilon = \exp(q_\epsilon)$ is smooth, we can define $h = e^{q_\epsilon/2} w_1$ and
integrate by parts to show that $\cQ_\epsilon[w_1] = \hat\cQ_\epsilon[h]$, where 
\[
  \hat\cQ_\epsilon[h] \,=\, \int_{\R^2}\bigl(|\nabla h|^2+ U_\epsilon\,h^2\bigr)\dd \xi\,,
  \qquad U_\epsilon(\xi) \,=\, \frac14\,\xi\cdot\nabla q_\epsilon - \frac14\,
  |\nabla q_\epsilon|^2 - \frac12\,.
\]
It is easy to verify that $|U_\epsilon(\xi) - U_0(\xi)| \le C\epsilon$ when $|\xi| \le
2\epsilon^{-1/4}$, so that $\hat\cQ_\epsilon[h]$ is close to $\hat\cQ_0[h]$ when $\epsilon$
is small. Note that $q_0(\xi) = |\xi|^2/4$ and $U_0(\xi) = |\xi|^2/16 - 1/2$,
so that $\cQ_0$ is the quadratic form of the quantum harmonic oscillator with
ground state $\psi(\xi) = e^{-|\xi|^2/8}$, see \cite[Appendix~A]{GW02}. In particular,
if $\langle h,\psi\rangle_{L^2} = 0$, it is known that $\hat\cQ_0[h] \ge \frac12 \|h\|_{L^2}^2$.

In our case, since we assume that $\int_{\R^2}w\dd\xi = 0$, the orthogonality
condition above is nearly satisfied in the sense that
\[
  \bigl|\langle h,\psi\rangle_{L^2}\bigr| \,=\, \biggl|\,\int_{\R^2} e^{(q_\epsilon-q_0)/2}w_1
  \dd\xi\,\biggr| \,=\, \biggl|\,\int_{\R^2}\Bigl(e^{(q_\epsilon-q_0)/2}\zeta_1 - 1\Bigr)w\dd\xi\,\biggr|
  \,\le\, C\epsilon\,\|p_\epsilon^{1/2}w\|_{L^2}\,,
\]
where we used the fact that $e^{(q_\epsilon-q_0)/2} = 1 + \cO(\epsilon)$ for $|\xi| \le \epsilon^{-1/4}$, 
and that $p_\epsilon^{-1/2} \in L^2(\R^2)$. If $\epsilon$ is sufficiently small, we deduce that
\begin{equation}\label{eq:Qeps3}
  \cQ_\epsilon[w_1] \,=\, \hat\cQ_\epsilon[h] \,\ge\, \frac12\,\|h\|_{L^2}^2 - C\epsilon 
 \|p_\epsilon^{1/2}w\|_{L^2}^2 \,=\, \frac12\,\int_{\I_\epsilon}p_\epsilon w_1^2\dd\xi
  - C\epsilon \int_{\R^2}p_\epsilon w^2\dd\xi\,.
\end{equation}
Observing that inequality \eqref{eq:Qeps1} also holds with $w_2$ replaced by $w_1$,
we add $3/4$ of \eqref{eq:Qeps3} and $1/4$ of \eqref{eq:Qeps1} to arrive at the
lower bound
\begin{equation}\label{eq:Qeps4}
  \cQ_\epsilon[w_1] \,\ge\, \kappa\int_{\I_\epsilon} p_\epsilon\Bigl(|\nabla w_1|^2 + 
  \bigl(\chi_\epsilon + 1\bigr)w_1^2\Bigr)\dd\xi - C\epsilon\int_{\R^2}p_\epsilon w^2\dd\xi\,,
\end{equation}
for some $\kappa > 0$. Finally, estimate \eqref{eq:Qepslow} is  a direct
consequence of \eqref{eq:Qeps0}, \eqref{eq:Qeps2}, and \eqref{eq:Qeps4}. 

\medskip\noindent{\bf Acknowledgements.} The research of both authors is supported 
by the grant BOURGEONS ANR-23-CE40-0014-01 of the French National Research Agency.
The first author was also supported by the Simons Collaboration on Wave Turbulence.

\bigskip\noindent
{\bf Martin Donati}\\
Institut Fourier, Universit\'e Grenoble Alpes, CNRS\\
and CNRS, Universit\'e de Poitiers, LMA, Poitiers, France \\
Email\: {\tt Martin.Donati@math.univ-poitiers.fr}

\bigskip\noindent
{\bf Thierry Gallay}\\
Institut Fourier, Universit\'e Grenoble Alpes, CNRS, Institut Universitaire de France\\
100 rue des Maths, 38610 Gi\`eres, France\\
Email\: {\tt Thierry.Gallay@univ-grenoble-alpes.fr}

\end{document}